\documentclass{cpamart1}     
\received{Month 200X}       
\volume{000}
\startingpage{1}

\authorheadline{K. L. Ho and L. Ying}
\titleheadline{HIF-IE}

\usepackage[breaklinks]{hyperref}
\usepackage{algorithm,algpseudocode,amssymb,booktabs,graphicx,fixltx2e,multicol,multirow,subcaption}

\usepackage{mathptmx}

\newtheorem{theorem}{Theorem}[section]
\newtheorem{lemma}[theorem]{Lemma}
\newtheorem{corollary}[theorem]{Corollary}

\theoremstyle{definition}

\theoremstyle{remark}
\newtheorem{remark}[theorem]{Remark}

\numberwithin{algorithm}{section}

\DeclareMathOperator{\diag}{diag}
\DeclareMathOperator{\rank}{rank}

\newcommand{\alg}[1]{{\sf #1}}
\newcommand{\cmp}{\mathsf{C}}
\newcommand{\defn}[1]{{\em #1}}
\newcommand{\eqv}{\mathsf{E}}
\newcommand{\far}{\mathsf{F}}
\newcommand{\krn}{\mathsf{K}}
\newcommand{\nbr}{\mathsf{N}}
\newcommand{\poly}{\mathcal{P}}
\newcommand{\range}{\mathcal{R}}
\newcommand{\rd}[1]{\check{#1}}
\newcommand{\sch}{\mathsf{S}}
\newcommand{\sk}[1]{\hat{#1}}
\newcommand{\skel}{\mathcal{Z}}
\newcommand{\spy}{\mathcal{S}}
\newcommand{\sumprime}{\sideset{}{'}\sum}
\newcommand{\tabdash}{\multicolumn{1}{|c}{---}}
\newcommand{\trans}{\mathsf{T}}

\begin{document}                        


\title{Hierarchical Interpolative Factorization for Elliptic Operators: Integral Equations}

\author{Kenneth L. Ho}{Stanford University}
\author{Lexing Ying}{Stanford University}





\begin{abstract}
 This paper introduces the hierarchical interpolative factorization for integral equations (HIF-IE) associated with elliptic problems in two and three dimensions. This factorization takes the form of an approximate generalized LU decomposition that permits the efficient application of the discretized operator and its inverse. HIF-IE is based on the recursive skeletonization algorithm but incorporates a novel combination of two key features: (1) a matrix factorization framework for sparsifying structured dense matrices and (2) a recursive dimensional reduction strategy to decrease the cost. Thus, higher-dimensional problems are effectively mapped to one dimension, and we conjecture that constructing, applying, and inverting the factorization all have linear or quasilinear complexity. Numerical experiments support this claim and further demonstrate the performance of our algorithm as a generalized fast multipole method, direct solver, and preconditioner. HIF-IE is compatible with geometric adaptivity and can handle both boundary and volume problems. MATLAB codes are freely available.
\end{abstract}

\maketitle






\section{Introduction}
This paper considers integral equations (IEs) of the form
\begin{align}
 a(x) u(x) + b(x) \int_{\Omega} K(\| x - y \|) c(y) u(y) \, d \Omega (y) = f(x), \quad x \in \Omega \subset \mathbb{R}^{d}
 \label{eqn:ie}
\end{align}
associated with elliptic partial differential equations (PDEs), where $a(x)$, $b(x)$, $c(x)$, and $f(x)$ are given functions; the integral kernel $K(r)$ is related to the fundamental solution of the underlying PDE; and $d = 2$ or $3$. Such equations encompass both boundary and volume problems and can be derived from PDEs in various ways. We give two prototypical examples below:
\begin{enumerate}
 \itemsep 1ex
 \item
  Consider the interior Dirichlet Laplace problem
  \begin{subequations}
   \label{eqn:dirich-laplace}
   \begin{align}
    \Delta u(x) &= 0, & &x \in \mathcal{D} \subset \mathbb{R}^{d}, \label{eqn:laplace-pde}\\
    u(x) &= f(x), & &x \in \partial \mathcal{D} \equiv \Gamma \label{eqn:laplace-bc}
   \end{align}
  \end{subequations}
  in a smooth, simply connected domain, which can be solved by writing $u(x)$ as the double-layer potential
  \begin{align}
   u(x) = \int_{\Gamma} \frac{\partial G}{\partial \nu_{y}} (\| x - y \|) \sigma (y) \, d \Gamma (y), \quad x \in \mathcal{D}
   \label{eqn:double-layer}
  \end{align}
  over an unknown surface density $\sigma (x)$, where
  \begin{align}
   G(r) =
   \begin{cases}
    -1/(2 \pi) \log r, & d = 2\\
    1/(4 \pi r), & d = 3
   \end{cases}
   \label{eqn:laplace-kernel}
  \end{align}
  is the fundamental solution of the free-space PDE and $\nu_{y}$ is the unit outer normal at $y \in \Gamma$. By construction, \eqref{eqn:double-layer} satisfies \eqref{eqn:laplace-pde}. To enforce the boundary condition \eqref{eqn:laplace-bc}, take the limit as $x \to \Gamma$ and use standard results from potential theory \cite{guenther:1988:prentice-hall} to obtain
  \begin{align}
   -\frac{1}{2} \sigma (x) + \int_{\Gamma} \frac{\partial G}{\partial \nu_{y}} (\| x - y \|) \sigma (y) \, d \Gamma (y) = f(x), \quad x \in \Gamma,
   \label{eqn:2k-bie}
  \end{align}
  where the integral is defined in the principal value sense. This is a boundary IE for $\sigma (x)$ of the form \eqref{eqn:ie} (up to a straightforward generalization to matrix-valued kernels).

  Alternatively, one could use the single-layer potential representation
  \begin{align*}
   u(x) = \int_{\Gamma} G(\| x - y \|) \sigma (y) \, d \Gamma (y), \quad x \in \mathcal{D},
  \end{align*}
  which immediately gives the IE
  \begin{align*}
   \int_{\Gamma} G(\| x - y \|) \sigma (y) \, d \Gamma (y) = f(x), \quad x \in \Gamma
  \end{align*}
  upon taking the limit as $x \to \Gamma$ since the integral is well-defined. Note that this has $a(x) \equiv 0$ in \eqref{eqn:ie}. Such equations are called first-kind Fredholm IEs and are generally ill-conditioned. Second-kind Fredholm IEs such as \eqref{eqn:2k-bie}, on the other hand, have $a(x) \neq 0$ for all $x$ and are usually well-conditioned.
 \item
  Consider the divergence-form PDE
  \begin{align*}
   \nabla \cdot (a(x) \nabla u(x)) = f(x), \quad x \in \Omega \subset \mathbb{R}^{d}
  \end{align*}
  and let
  \begin{align*}
   u(x) = \int_{\Omega} G(\| x - y \|) \sigma (y) \, d \Omega (y),
  \end{align*}
  where $G(r)$ is as defined in \eqref{eqn:laplace-kernel}. Then the PDE becomes the volume IE
  \begin{align*}
   a(x) \sigma(x) + \nabla a(x) \cdot \int_{\Omega} \nabla_{x} G(\| x - y \|) \sigma (y) \, d \Omega (y) = f(x), \quad x \in \Omega
  \end{align*}
  upon substitution, which again has the form \eqref{eqn:ie}.
\end{enumerate}

IEs can similarly be derived for many of the PDEs of classical physics including the Laplace, Helmholtz, Stokes, and time-harmonic Maxwell equations. In such cases, the kernel function $K(r)$ is typically singular near zero but otherwise smooth with non-compact support. For this paper, we will also require that $K(r)$ not be too oscillatory.

Discretization of \eqref{eqn:ie} using, e.g., the Nystr\"{o}m, collocation, or Galerkin method leads to a linear system
\begin{align}
 Au = f,
 \label{eqn:linear-system}
\end{align}
where $A \in \mathbb{C}^{N \times N}$ is dense with $u$ and $f$ the discrete analogues of $u(x)$ and $f(x)$, respectively. This paper is concerned with the efficient factorization and solution of such systems.

\subsection{Previous Work}
Numerical methods for solving \eqref{eqn:linear-system} can be classified into several groups. The first consists of classical direct methods like Gaussian elimination or other standard matrix factorizations \cite{golub:1996:johns-hopkins-univ}, which compute the solution exactly (in principle, to machine precision, up to conditioning) without iteration. These methods are useful when $N$ is small. However, since $A$ is dense, such algorithms generally have $O(N^{3})$ complexity, which quickly makes them infeasible as $N$ increases.

The second group is that of iterative methods, among the most popular of which are Krylov subspace methods such as conjugate gradient \cite{hestenes:1952:j-res-nat-bur-stand,van-der-vorst:1992:siam-j-sci-stat-comput} or GMRES \cite{saad:1986:siam-j-sci-stat-comput}. The number of iterations required depends on the problem and is typically small for second-kind IEs but can grow rapidly for first-kind ones. The main computational cost is the calculation of matrix-vector products at each iteration. Combined with fast multipole methods (FMMs) \cite{fong:2009:j-comput-phys,greengard:1987:j-comput-phys,greengard:1997:acta-numer,ying:2004:j-comput-phys} or other accelerated matrix multiplication schemes \cite{barnes:1986:nature,hackbusch:1989:numer-math}, such techniques can yield asymptotically optimal or near-optimal solvers with $O(N)$ or $O(N \log N)$ complexity. However, iterative methods are not as robust as their direct counterparts, especially when $a(x)$, $b(x)$, or $c(x)$ lacks regularity or has high contrast. In such cases, convergence can be slow and specialized preconditioners are often needed. Furthermore, iterative methods can be inefficient for systems involving multiple right-hand sides or low-rank updates, which is an important setting for many applications of increasing interest, including time stepping, inverse problems, and design.

The third group covers rank-structured direct solvers, which exploit the observation that certain off-diagonal blocks of $A$ are numerically low-rank in order to dramatically lower the cost. The seminal work in this area is due to Hackbusch et al.\ \cite{hackbusch:1999:computing,hackbusch:2002:computing,hackbusch:2000:computing}, whose $\mathcal{H}$- and $\mathcal{H}^{2}$-matrices have been shown to achieve linear or quasilinear complexity. Although their work has had significant theoretical impact, in practice, the constants implicit in the asymptotic scalings tend to be large due to the recursive nature of the inversion algorithms and the use of expensive hierarchical matrix-matrix multiplication.

More recent developments aimed at improving practical performance include solvers for hierarchically semiseparable (HSS) matrices \cite{chandrasekaran:2006a:siam-j-matrix-anal-appl,chandrasekaran:2006b:siam-j-matrix-anal-appl,xia:2010:numer-linear-algebra-appl} and methods based on recursive skeletonization (RS) \cite{gillman:2012:front-math-china,greengard:2009:acta-numer,ho:2012:siam-j-sci-comput,martinsson:2005:j-comput-phys}, among other related schemes \cite{ambikasaran:2013:j-sci-comput,bremer:2012:j-comput-phys,chen:2002:adv-comput-math}. These can be viewed as special cases of $\mathcal{H}^{2}$-matrices and are optimal in one dimension (1D) (e.g., boundary IEs on curves) but have superlinear complexities in higher dimensions. In particular, RS proceeds analogously to the nested dissection multifrontal method (MF) for sparse linear systems \cite{duff:1983:acm-trans-math-software,george:1973:siam-j-numer-anal}, with the so-called skeletons characterizing the off-diagonal blocks corresponding to the separator fronts. These grow as $O(N^{1/2})$ in two dimensions (2D) and $O(N^{2/3})$ in three dimensions (3D), resulting in solver complexities of $O(N^{3/2})$ and $O(N^{2})$, respectively.

Recently, Corona, Martinsson, and Zorin \cite{corona:2015:appl-comput-harmon-anal} constructed an $O(N)$ RS solver in 2D by exploiting further structure among the skeletons and using hierarchical matrix algebra. The principal observation is that for a broad class of integral kernels, the generic behavior of RS is to retain degrees of freedom (DOFs) only along the boundary of each cell in a domain partitioning. Thus, 2D problems are reduced to 1D, and the large skeleton matrices accumulated throughout the algorithm can be handled efficiently using 1D HSS techniques. However, this approach is quite involved and has yet to be realized in 3D or in complicated geometries.

\subsection{Contributions}
In this paper, we introduce the hierarchical interpolative factorization for IEs (HIF-IE), which produces an approximate generalized LU decomposition of $A$ with linear or quasilinear complexity estimates. HIF-IE is based on RS but augments it with a novel combination of two key features: (1) a matrix factorization formulation via a sparsification framework similar to that developed in \cite{chandrasekaran:2006b:siam-j-matrix-anal-appl,xia:2013:siam-j-sci-comput,xia:2010:numer-linear-algebra-appl} and (2) a recursive dimensional reduction scheme as pioneered in \cite{corona:2015:appl-comput-harmon-anal}. Unlike \cite{corona:2015:appl-comput-harmon-anal}, however, which keeps large skeleton sets but works with them implicitly using fast structured methods, our sparsification approach allows us to reduce the skeletons explicitly. This obviates the need for internal hierarchical matrix representations, which substantially simplifies the algorithm and enables it to extend naturally to 3D and to complex geometries, in addition to promoting a more direct view of the dimensional reduction process.

Figure \ref{fig:schematic} shows a schematic of HIF-IE as compared to RS in 2D.
\begin{figure}
 \includegraphics{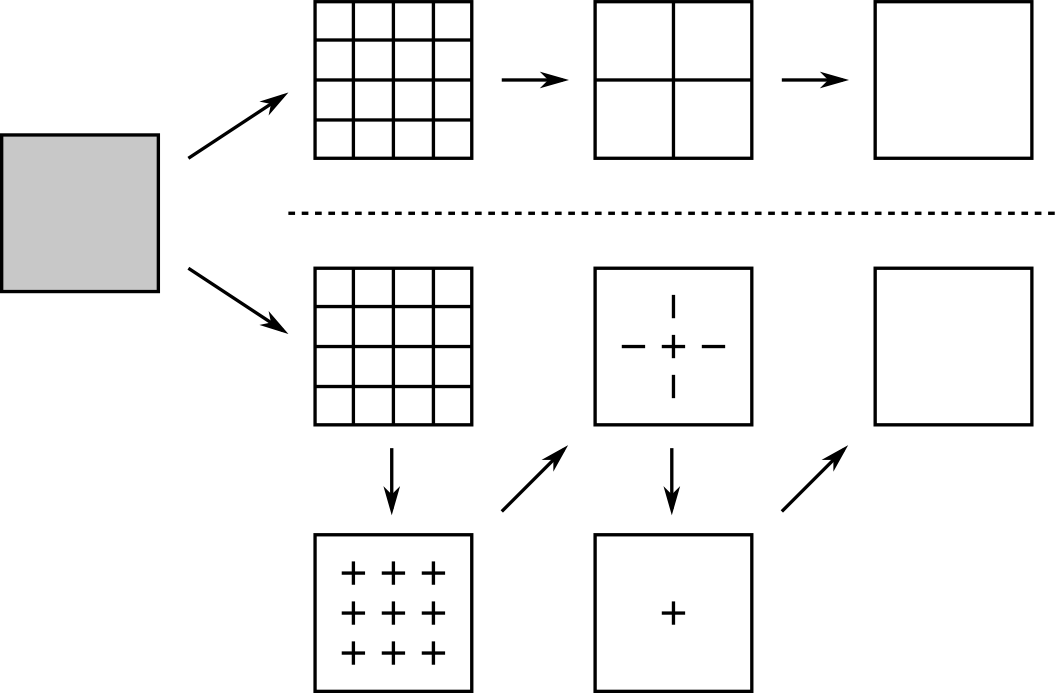}
 \caption{Schematic of RS (top) and HIF-IE (bottom) in 2D. The gray box (left) represents a uniformly discretized square; the lines in the interior of the boxes (right) denote the remaining DOFs after each level of skeletonization.}
 \label{fig:schematic}
\end{figure}
In RS (top), the domain is partitioned into a set of square cells at each level of a tree hierarchy. Each cell is skeletonized from the finest level to the coarsest, leaving DOFs only along cell interfaces. The size of these interfaces evidently grows as we march up the tree, which ultimately leads to the observed $O(N^{3/2})$ complexity.

In contrast, in HIF-IE (bottom), we start by skeletonizing the cells at the finest level as in RS but, before proceeding further, perform an additional level of edge skeletonization by grouping the remaining DOFs by cell edge. This respects the 1D structure of the interface geometry and allows more DOFs to be eliminated. The combination of cell and edge compression is then repeated up the tree, with the result that the skeleton growth is now suppressed. The reduction from 2D (square cells) to 1D (edges) to zero dimensions (0D) (points) is completely explicit. Extension to 3D is immediate by skeletonizing cubic cells, then faces, then edges at each level to execute a reduction from 3D to 2D to 1D to 0D. This tight control of the skeleton size is essential for achieving near-optimal scaling.

Once the factorization has been constructed, it can be used to rapidly apply both $A$ and $A^{-1}$, thereby serving as a generalized FMM, direct solver, or preconditioner (depending on the accuracy). Other capabilities are possible, too, though they will not be pursued here. As such, HIF-IE is considerably more general than many previous non--factorization-based fast direct solvers \cite{chandrasekaran:2006a:siam-j-matrix-anal-appl,corona:2015:appl-comput-harmon-anal,gillman:2012:front-math-china,ho:2012:siam-j-sci-comput,martinsson:2005:j-comput-phys}, which facilitate only the application of the inverse.

Extensive numerical experiments reveal strong evidence for quasilinear complexity and demonstrate that HIF-IE can accurately approximate various integral operators in both boundary and volume settings with high practical efficiency.

\subsection{Outline}
The remainder of this paper is organized as follows. In Section \ref{sec:prelim}, we introduce the basic tools needed for our algorithm, including an efficient matrix sparsification operation that we call skeletonization. In Section \ref{sec:rskelf}, we describe the recursive skeletonization factorization (RSF), a reformulation of RS using our new factorization approach. This will serve to familiarize the reader with our sparsification framework as well as to highlight the fundamental difficulty associated with RS methods in 2D and 3D. In Section \ref{sec:hifie}, we present HIF-IE as an extension of RSF with additional levels of skeletonization corresponding to recursive dimensional reductions. Although we cannot yet provide a rigorous complexity analysis, estimates based on well-supported rank assumptions suggest that HIF-IE achieves linear or quasilinear complexity. This conjecture is borne out by numerical experiments, which we detail in Section \ref{sec:results}. Finally, Section \ref{sec:conclusion} concludes with some discussion and future directions.

\section{Preliminaries}
\label{sec:prelim}
In this section, we first list our notational conventions and then describe the basic elements of our algorithm.

Uppercase letters will generally denote matrices, while the lowercase letters $c$, $p$, $q$, $r$, and $s$ denote ordered sets of indices, each of which is associated with a DOF in the problem. For a given index set $c$, its cardinality is written $|c|$. The (unordered) complement of $c$ is given by $c^{\cmp}$, with the parent set to be understood from the context. The uppercase letter $C$ is reserved to denote a collection of disjoint index sets.

Given a matrix $A$, $A_{pq}$ is the submatrix with rows and columns restricted to the index sets $p$ and $q$, respectively. We also use the MATLAB notation $A_{:,q}$ to denote the submatrix with columns restricted to $q$.

Throughout, $\| \cdot \|$ refers to the $2$-norm.

\subsection{Sparse Elimination}
Let
\begin{align}
 A =
 \begin{bmatrix}
  A_{pp} & A_{pq}\\
  A_{qp} & A_{qq} & A_{qr}\\
  & A_{rq} & A_{rr}
 \end{bmatrix}
 \label{eqn:sparse-matrix}
\end{align}
be a matrix defined over the indices $(p, q, r)$. This matrix structure often appears in sparse PDE problems, where, for example, $p$ corresponds to the interior DOFs of a region $\mathcal{D}$, $q$ to the DOFs on the boundary $\partial \mathcal{D}$, and $r$ to the external region $\Omega \setminus \bar{\mathcal{D}}$, which should be thought of as large. In this setting, the DOFs $p$ and $r$ are separated by $q$ and hence do not directly interact, resulting in the form \eqref{eqn:sparse-matrix}.

Our first tool is quite standard and concerns the efficient elimination of DOFs from such sparse matrices.

\begin{lemma}
 \label{lem:sparse-elim}
 Let $A$ be given by \eqref{eqn:sparse-matrix} and write $A_{pp} = L_{p} D_{p} U_{p}$ in factored form, where $L_{p}$ and $U_{p}$ are unit triangular matrices (up to permutation). If $A_{pp}$ is nonsingular, then
 \begin{align}
  R_{p}^{*} A S_{p} =
  \begin{bmatrix}
   D_{p}\\
   & B_{qq} & A_{qr}\\
   & A_{rq} & A_{rr}
  \end{bmatrix},
  \label{eqn:sparse-elim}
 \end{align}
 where
 \begin{align*}
  R_{p}^{*} &=
  \begin{bmatrix}
   I\\
   -A_{qp} U_{p}^{-1} D_{p}^{-1} & I\\
   & & I
  \end{bmatrix}
  \begin{bmatrix}
   L_{p}^{-1}\\
   & I\\
   & & I
  \end{bmatrix},\\
  S_{p} &=
  \begin{bmatrix}
   U_{p}^{-1}\\
   & I\\
   & & I
  \end{bmatrix}
  \begin{bmatrix}
   I & -D_{p}^{-1} L_{p}^{-1} A_{pq}\\
   & I\\
   & & I
  \end{bmatrix}
 \end{align*}
 and $B_{qq} = A_{qq} - A_{qp} A_{pp}^{-1} A_{pq}$ is the associated Schur complement.
\end{lemma}

Note that the indices $p$ have been decoupled from the rest. Regarding the subsystem in \eqref{eqn:sparse-elim} over the indices $(q, r)$ only, we may therefore say that the DOFs $p$ have been eliminated. The operators $R_{p}$ and $S_{p}$ carry out this elimination, which furthermore is particularly efficient since the interactions involving the large index set $r$ are unchanged.

\subsection{Interpolative Decomposition}
Our next tool is the interpolative decomposition (ID) \cite{cheng:2005:siam-j-sci-comput} for low-rank matrices, which we present in a somewhat nonstandard form below.

\begin{lemma}
 \label{lem:id}
 Let $A = A_{:,q} \in \mathbb{C}^{m \times n}$ with rank $k \leq \min (m, n)$. Then there exist a partitioning $q = \sk{q} \cup \rd{q}$ with $|\sk{q}| = k$ and a matrix $T_{q} \in \mathbb{C}^{k \times n}$ such that $A_{:,\rd{q}} = A_{:,\sk{q}} T_{q}$.
\end{lemma}

\begin{proof}
 Let
 \begin{align*}
  A \Pi = QR = Q
  \begin{bmatrix}
   R_{1} & R_{2}
  \end{bmatrix}
 \end{align*}
 be a so-called thin pivoted QR decomposition of $A$, where $Q \in \mathbb{C}^{m \times k}$ is unitary, $R \in \mathbb{C}^{k \times n}$ is upper triangular, and the permutation matrix $\Pi \in \{ 0, 1 \}^{n \times n}$ has been chosen so that $R_{1} \in \mathbb{C}^{k \times k}$ is nonsingular. Then identifying the first $k$ pivots with $\sk{q}$ and the remainder with $\rd{q}$,
 \begin{align*}
  A_{:,\rd{q}} = Q R_{2} = (Q R_{1}) (R_{1}^{-1} R_{2}) \equiv A_{:,\sk{q}} T_{q}
 \end{align*}
 for $T_{q} = R_{1}^{-1} R_{2}$.
\end{proof}

The ID can also be written more traditionally as
\begin{align*}
 A = A_{:,\sk{q}}
 \begin{bmatrix}
  I & T_{q}
 \end{bmatrix} \Pi
\end{align*}
where $\Pi$ is the permutation matrix associated with the ordering $(\sk{q}, \rd{q})$. We call $\sk{q}$ and $\rd{q}$ the \defn{skeleton} and \defn{redundant} indices, respectively. Lemma \ref{lem:id} states that the redundant columns of $A$ can be interpolated from its skeleton columns. The following shows that the ID can also be viewed as a sparsification operator.

\begin{corollary}
 Let $A = A_{:,q}$ be a low-rank matrix. If $q = \sk{q} \cup \rd{q}$ and $T_{q}$ are such that $A_{:,\rd{q}} = A_{:,\sk{q}} T_{q}$, then
 \begin{align*}
  \begin{bmatrix}
   A_{:,\rd{q}} & A_{:,\sk{q}}
  \end{bmatrix}
  \begin{bmatrix}
   I\\
   -T_{q} & I
  \end{bmatrix} =
  \begin{bmatrix}
   0 & A_{:,\sk{q}}
  \end{bmatrix}.
 \end{align*}
\end{corollary}

In general, let $A_{:,\rd{q}} = A_{:,\sk{q}} T_{q} + E$ for some error matrix $E$ and characterize the ID by the functions $\alpha (n, k)$ and $\beta (n, k)$ such that
\begin{align}
 \| T_{q} \| \leq \alpha (n, k), \quad \| E \| \leq \beta (n, k) \sigma_{k + 1} (A),
 \label{eqn:id-bounds}
\end{align}
where $\sigma_{k + 1} (A)$ is the $(k + 1)$st largest singular value of $A$. If $|\alpha (n, k)|$ and $|\beta (n, k)|$ are not too large, then \eqref{eqn:id-bounds} implies that the reconstruction of $A_{:,\rd{q}}$ is stable and accurate. There exists an ID with
\begin{align}
 \alpha (n, k) = \sqrt{f^{2} k(n - k)}, \quad \beta (n, k) = \sqrt{1 + f^{2} k(n - k)}
 \label{eqn:id-bounds-fcn}
\end{align}
for $f = 1$, but computing it can take exponential time, requiring the combinatorial maximization of a submatrix determinant. However, an ID satisfying \eqref{eqn:id-bounds-fcn} with any $f > 1$ can be computed in polynomial time \cite{gu:1996:siam-j-sci-comput}. In this paper, we use the algorithm of \cite{cheng:2005:siam-j-sci-comput} based on a simple pivoted QR decomposition, which has a possibility of failure but seems to consistently achieve \eqref{eqn:id-bounds-fcn} with $f = 2$ in practice at a cost of $O(kmn)$ operations. Fast algorithms based on random sampling are also available \cite{halko:2011:siam-rev}, but these can incur some loss of accuracy (see also Section \ref{sec:hifie:accel-comp}).

The ID can be applied in both fixed and adaptive rank settings. In the former, the rank $k$ is specified, while, in the latter, the approximation error is specified and the rank adjusted to achieve (an estimate of) it. Hereafter, we consider the ID only in the adaptive sense, using the relative magnitudes of the pivots to adaptively select $k$ such that $\| E \| \lesssim \epsilon \| A \|$ for any specified relative precision $\epsilon > 0$.

\subsection{Skeletonization}
We now combine Lemmas \ref{lem:sparse-elim} and \ref{lem:id} to efficiently eliminate redundant DOFs from dense matrices with low-rank off-diagonal blocks.

\begin{lemma}
 \label{lem:skel}
 Let
 \begin{align*}
  A =
  \begin{bmatrix}
   A_{pp} & A_{pq}\\
   A_{qp} & A_{qq}
  \end{bmatrix}
 \end{align*}
 with $A_{pq}$ and $A_{qp}$ low-rank, and let $p = \sk{p} \cup \rd{p}$ and $T_{p}$ be such that
 \begin{align*}
  \begin{bmatrix}
   A_{q \rd{p}}\\
   A_{\rd{p} q}^{*}
  \end{bmatrix} =
  \begin{bmatrix}
   A_{q \sk{p}}\\
   A_{\sk{p} q}^{*}
  \end{bmatrix} T_{p},
 \end{align*}
 i.e., $A_{q \rd{p}} = A_{q \sk{p}} T_{p}$ and $A_{\rd{p} q} = T_{p}^{*} A_{\sk{p} q}$. Without loss of generality, write
 \begin{align*}
  A =
  \begin{bmatrix}
   A_{\rd{p} \rd{p}} & A_{\rd{p} \sk{p}} & A_{\rd{p} q}\\
   A_{\sk{p} \rd{p}} & A_{\sk{p} \sk{p}} & A_{\sk{p} q}\\
   A_{q \rd{p}} & A_{q \sk{p}} & A_{qq}
  \end{bmatrix}
 \end{align*}
 and define
 \begin{align*}
  Q_{p} =
  \begin{bmatrix}
   I\\
   -T_{p} & I\\
   & & I
  \end{bmatrix}.
 \end{align*}
 Then
 \begin{align}
  Q_{p}^{*} A Q_{p} =
  \begin{bmatrix}
   B_{\rd{p} \rd{p}} & B_{\rd{p} \sk{p}}\\
   B_{\sk{p} \rd{p}} & A_{\sk{p} \sk{p}} & A_{\sk{p} q}\\
   & A_{q \sk{p}} & A_{qq}
  \end{bmatrix},
  \label{eqn:id-sparse}
 \end{align}
 where
 \begin{align*}
  B_{\rd{p} \rd{p}} &= A_{\rd{p} \rd{p}} - T_{p}^{*} A_{\sk{p} \rd{p}} - A_{\rd{p} \sk{p}} T_{p} + T_{p}^{*} A_{\sk{p} \sk{p}} T_{p},\\
  B_{\rd{p} \sk{p}} &= A_{\rd{p} \sk{p}} - T_{p}^{*} A_{\sk{p} \sk{p}},\\
  B_{\sk{p} \rd{p}} &= A_{\sk{p} \rd{p}} - A_{\sk{p} \sk{p}} T_{p},
 \end{align*}
 so
 \begin{align}
  \label{eqn:skel}
  R_{\rd{p}}^{*} Q_{p}^{*} A Q_{p} S_{\rd{p}} =
  \begin{bmatrix}
   D_{\rd{p}}\\
   & B_{\sk{p} \sk{p}} & A_{\sk{p} q}\\
   & A_{q \sk{p}} & A_{qq}
  \end{bmatrix} \equiv \skel_{p} (A),
 \end{align}
 where $R_{\rd{p}}$ and $S_{\rd{p}}$ are the elimination operators of Lemma \ref{lem:sparse-elim} associated with $\rd{p}$ and $B_{\sk{p} \sk{p}} = A_{\sk{p} \sk{p}} - B_{\sk{p} \rd{p}} B_{\rd{p} \rd{p}}^{-1} B_{\rd{p} \sk{p}}$, assuming that $B_{\rd{p} \rd{p}}$ is nonsingular.
\end{lemma}

In essence, the ID sparsifies $A$ by decoupling $\rd{p}$ from $q$, thereby allowing it to be eliminated using efficient sparse techniques. We refer to this procedure as \defn{skeletonization} since only the skeletons $\sk{p}$ remain. Note that the interactions involving $q = p^{\cmp}$ are unchanged. A very similar approach has previously been described in the context of HSS ULV decompositions \cite{chandrasekaran:2006b:siam-j-matrix-anal-appl} by combining the structure-preserving rank-revealing factorization \cite{xia:2012:siam-j-matrix-anal-appl} with reduced matrices \cite{xia:2013:siam-j-sci-comput}.

In general, the ID often only approximately sparsifies $A$ (for example, if its off-diagonal blocks are low-rank only to a specified numerical precision) so that \eqref{eqn:id-sparse} and consequently \eqref{eqn:skel} need not hold exactly. In such cases, the skeletonization operator $\skel_{p} (\cdot)$ should be interpreted as also including an intermediate truncation step that enforces sparsity explicitly. For notational convenience, however, we will continue to identify the left- and right-hand sides of \eqref{eqn:skel} by writing $\skel_{p} (A) \approx R_{\rd{p}}^{*} Q_{p}^{*} A Q_{p} S_{\rd{p}}$, with the truncation to be understood implicitly.

In this paper, we often work with a collection $C$ of disjoint index sets, where $A_{c,c^{\cmp}}$ and $A_{c^{\cmp},c}$ are numerically low-rank for all $c \in C$. Applying Lemma \ref{lem:skel} to all $c \in C$ gives
\begin{align*}
 \skel_{C} (A) \approx U^{*} AV, \qquad U = \prod_{c \in C} Q_{c} R_{\rd{c}}, \quad V = \prod_{c \in C} Q_{c} S_{\rd{c}},
\end{align*}
where the redundant DOFs $\rd{c}$ for each $c \in C$ have been decoupled from the rest and the matrix products over $C$ can be taken in any order. The resulting skeletonized matrix $\skel_{C} (A)$ is significantly sparsified and has a block diagonal structure over the index groups
\begin{align*}
 \theta = \left( \bigcup_{c \in C} \{ \rd{c} \} \right) \cup \left\{ s \setminus \bigcup_{c \in C} \rd{c} \right\},
\end{align*}
where the outer union is to be understood as acting on collections of index sets and $s = \{ 1, \dots, N \}$ is the set of all indices.

\section{Recursive Skeletonization Factorization}
\label{sec:rskelf}
In this section, we present RSF, a reformulation of RS \cite{gillman:2012:front-math-china,greengard:2009:acta-numer,ho:2012:siam-j-sci-comput,martinsson:2005:j-comput-phys} as a matrix factorization using the sparsification view of skeletonization as developed in Lemma \ref{lem:skel}. Mathematically, RSF is identical to RS but expresses the matrix $A$ as a (multiplicative) multilevel generalized LU decomposition instead of as an additive hierarchical low-rank update. This representation enables much simpler algorithms for applying $A$ and $A^{-1}$ as well as establishes a direct connection with MF \cite{duff:1983:acm-trans-math-software,george:1973:siam-j-numer-anal} for sparse matrices, which produces a (strict) LU decomposition using Lemma \ref{lem:sparse-elim}. Indeed, RSF is essentially just MF with pre-sparsification via the ID at each level. This point of view places methods for structured dense and sparse matrices within a common framework, which provides a potential means to transfer techniques from one class to the other.

Note that because RSF is based on elimination, it requires that certain intermediate matrices be invertible, which in general means that $A$ must be square. This is a slight limitation when compared to RS, which can be used, for example, as a generalized FMM \cite{gillman:2012:front-math-china,ho:2012:siam-j-sci-comput} or least squares solver \cite{ho:2014:siam-j-matrix-anal-appl} for rectangular matrices.

We begin with a detailed description of RSF in 2D before extending to 3D in the natural way (the 1D case will not be treated but should be obvious from the discussion). The same presentation framework will also be used for HIF-IE in Section \ref{sec:hifie}, which we hope will help make clear the specific innovations responsible for its improved complexity estimates.

\subsection{Two Dimensions}
\label{sec:rskelf:2d}
Consider the IE \eqref{eqn:ie} on $\Omega = (0, 1)^{2}$, discretized using a piecewise constant collocation method over a uniform $n \times n$ grid for simplicity. More general domains and discretizations can be handled without difficulty, but the current setting will serve to illustrate the main ideas. Let $h$ be the step size in each direction and assume that $n = 1/h = 2^{L} m$, where $m = O(1)$ is a small integer. Integer pairs $j = (j_{1}, j_{2})$ index the elements $\Omega_{j} = h(j_{1} - 1, j_{1}) \times h(j_{2} - 1, j_{2})$ and their centers $x_{j} = h(j_{1} - 1/2, j_{2} - 1/2)$ for $1 \leq j_{1}, j_{2} \leq n$. With $\{ x_{j} \}$ as the collocation points, the discrete system \eqref{eqn:linear-system} reads
\begin{align*}
 a_{i} u_{i} + b_{i} \sum_{j} K_{ij} c_{j} u_{j} = f_{i}
\end{align*}
at each $x_{i}$, where $a_{j} = a(x_{j})$, $b_{j} = b(x_{j})$, $c_{j} = c(x_{j})$, and $f_{j} = f(x_{j})$; $u_{j}$ is the approximation to $u(x_{j})$; and
\begin{align}
 K_{ij} = \int_{\Omega_{j}} K(\| x_{i} - y \|) \, d \Omega (y).
 \label{eqn:kernel-integral}
\end{align}
Note that $A$ is not stored since it is dense; rather, its entries are generated as needed. The total number of DOFs is $N = n^{2}$, each of which is associated with a point $x_{j}$ and an index in $s$.

The algorithm proceeds by eliminating DOFs level by level. At each level $\ell$, the set of DOFs that have not been eliminated are called \defn{active} with indices $s_{\ell}$. Initially, we set $A_{0} = A$ and $s_{0} = s$. Figure \ref{fig:rskelf2} shows the active DOFs at each level for a representative example.
\begin{figure}
 \centering
 \begin{subfigure}{0.2\textwidth}
  \includegraphics[width=\textwidth]{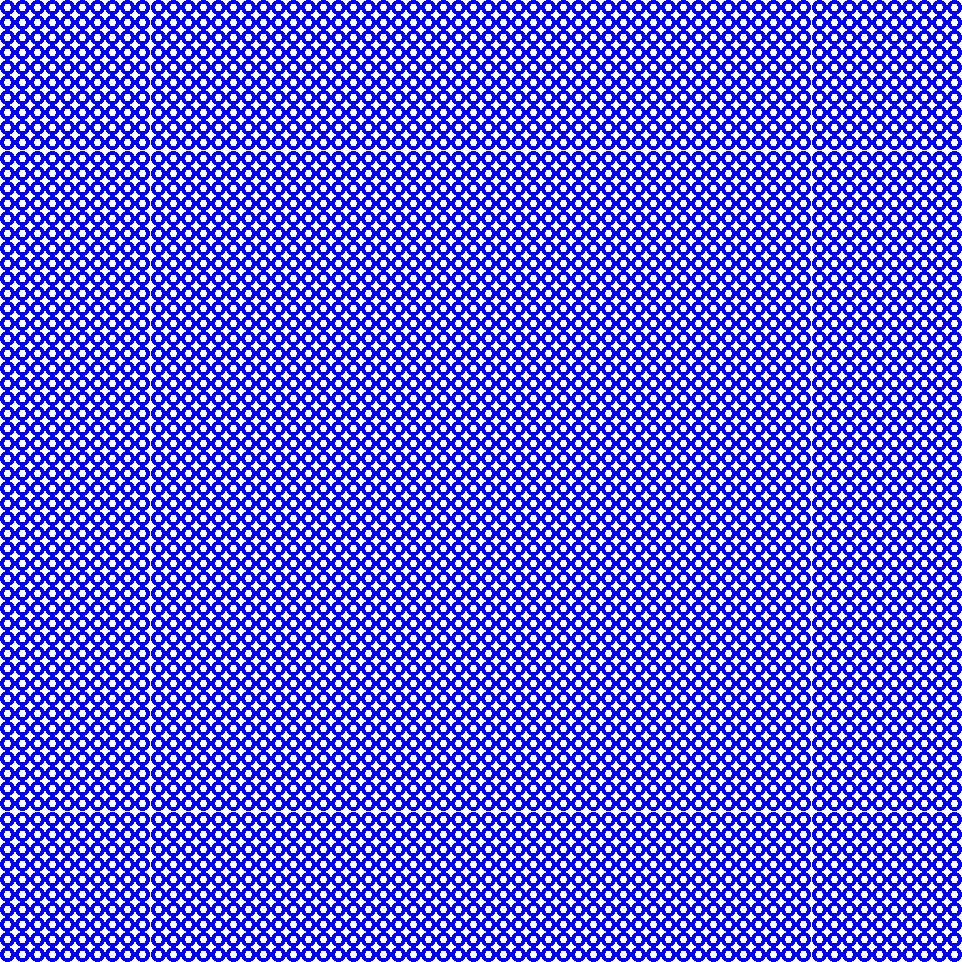}
  \caption*{$\ell = 0$}
 \end{subfigure}
 \quad
 \begin{subfigure}{0.2\textwidth}
  \includegraphics[width=\textwidth]{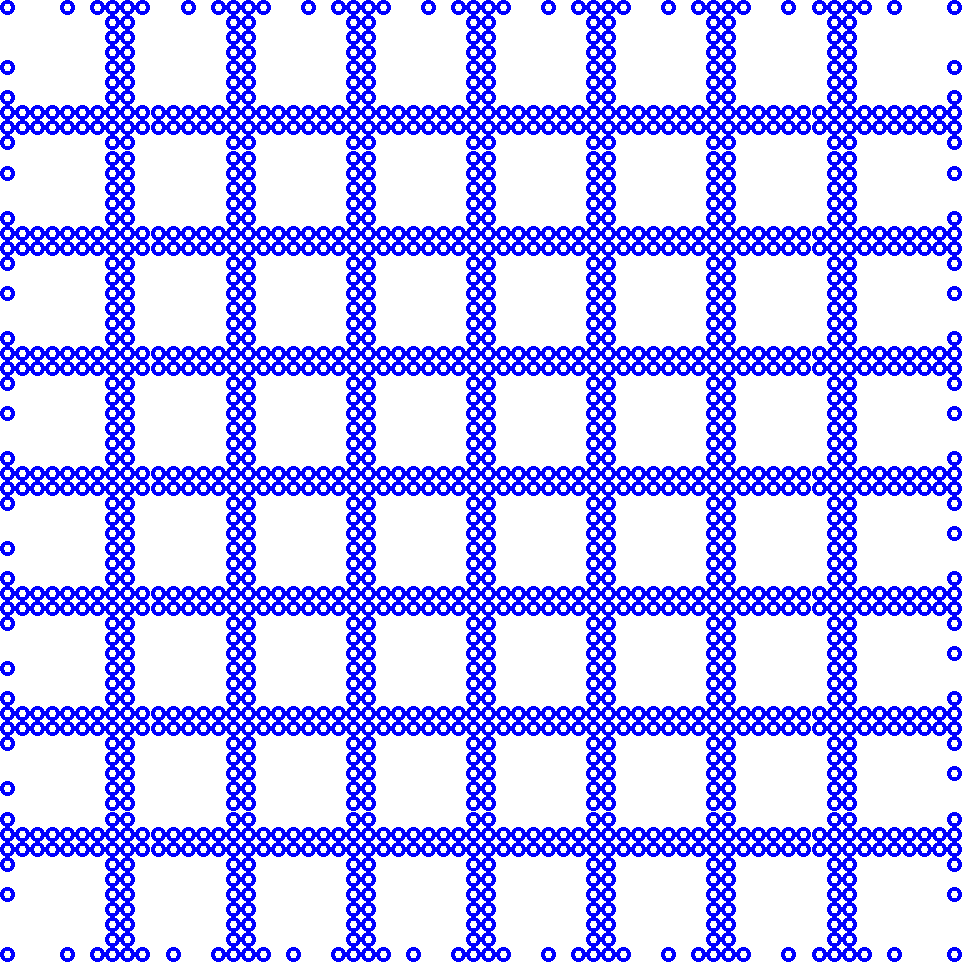}
  \caption*{$\ell = 1$}
 \end{subfigure}
 \quad
 \begin{subfigure}{0.2\textwidth}
  \includegraphics[width=\textwidth]{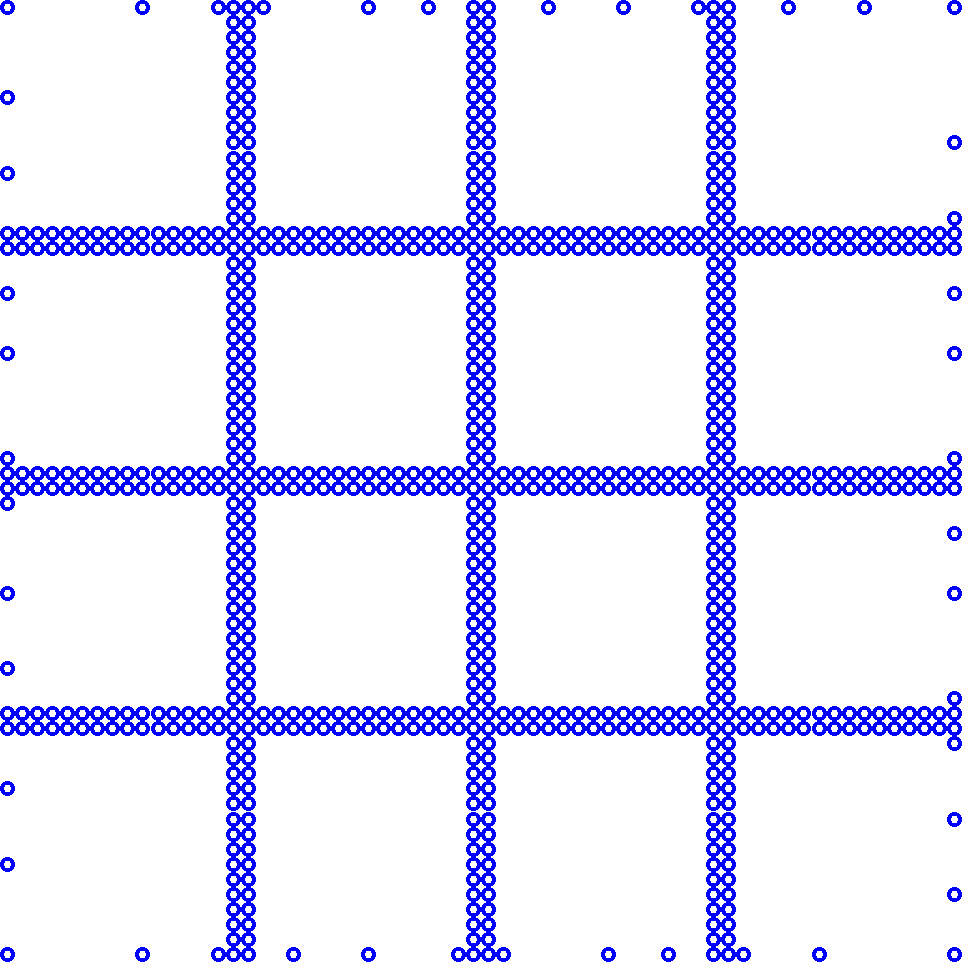}
  \caption*{$\ell = 2$}
 \end{subfigure}
 \quad
 \begin{subfigure}{0.2\textwidth}
  \includegraphics[width=\textwidth]{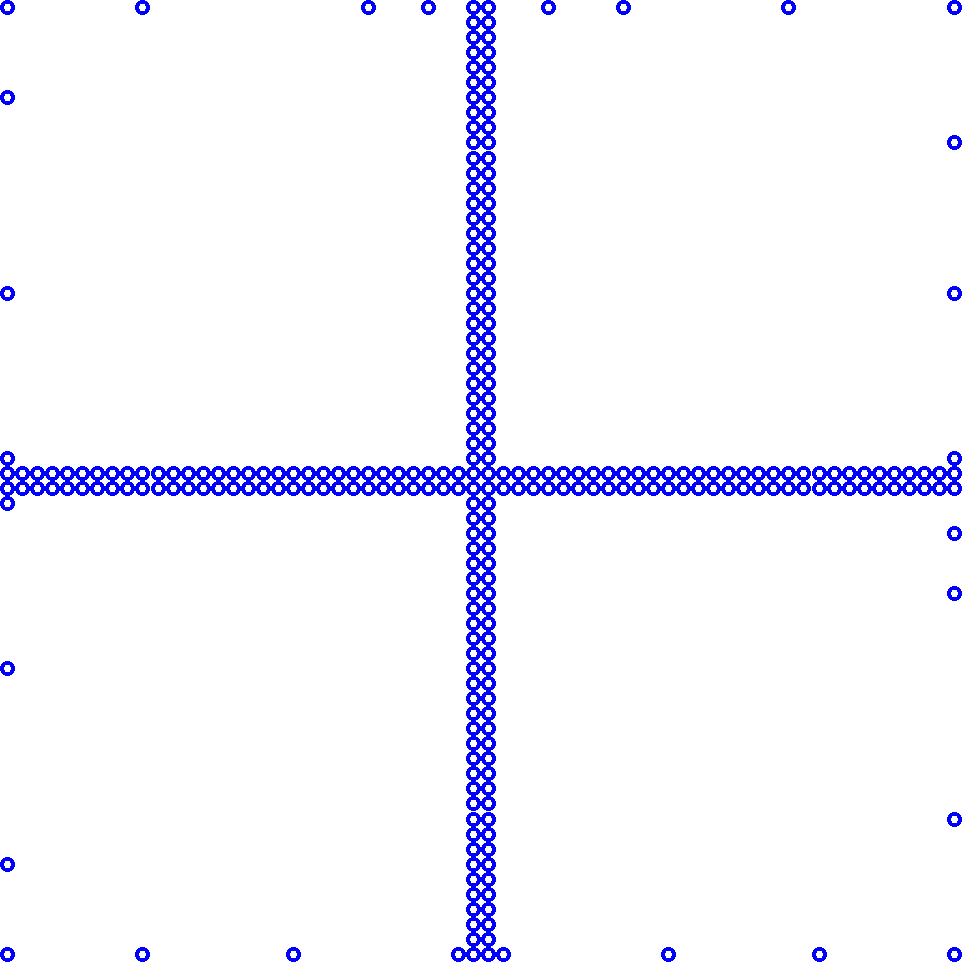}
  \caption*{$\ell = 3$}
 \end{subfigure}
 \caption{Active DOFs at each level $\ell$ of RSF in 2D.}
 \label{fig:rskelf2}
\end{figure}

\subsubsection*{Level $0$}
Defined at this stage are $A_{0}$ and $s_{0}$. Partition $\Omega$ into the Voronoi cells \cite{aurenhammer:1991:acm-comput-surv} $mh(j_{1} - 1, j_{1}) \times mh(j_{2} - 1, j_{2})$ of width $mh = n / 2^{L}$ about the centers $mh(j_{1} - 1/2, j_{2} - 1/2)$ for $1 \leq j_{1}, j_{2} \leq 2^{L}$. Let $C_{0}$ be the collection of index sets corresponding to the active DOFs of each cell. Clearly, $\bigcup_{c \in C_{0}} c = s_{0}$. Then skeletonization with respect to $C_{0}$ gives
\begin{align*}
 A_{1} = \skel_{C_{0}} (A_{0}) \approx U_{0}^{*} A_{0} V_{0}, \qquad U_{0} = \prod_{c \in C_{0}} Q_{c} R_{\rd{c}}, \quad V_{0} = \prod_{c \in C_{0}} Q_{c} S_{\rd{c}},
\end{align*}
where the DOFs $\bigcup_{c \in C_{0}} \rd{c}$ have been eliminated (and marked inactive). Let $s_{1} = s_{0} \setminus \bigcup_{c \in C_{0}} \rd{c} = \bigcup_{c \in C_{0}} \sk{c}$ be the remaining active DOFs. The matrix $A_{1}$ is block diagonal with block partitioning
\begin{align*}
 \theta_{1} = \left( \bigcup_{c \in C_{0}} \{ \rd{c} \} \right) \cup \{ s_{1} \}.
\end{align*}

\subsubsection*{Level $\ell$}
Defined at this stage are $A_{\ell}$ and $s_{\ell}$. Partition $\Omega$ into the Voronoi cells $2^{\ell} mh(j_{1} - 1, j_{1}) \times 2^{\ell} mh(j_{2} - 1, j_{2})$ of width $2^{\ell} mh = n / 2^{L - \ell}$ about the centers $2^{\ell} mh(j_{1} - 1/2, j_{2} - 1/2)$ for $1 \leq j_{1}, j_{2} \leq 2^{L - \ell}$. Let $C_{\ell}$ be the collection of index sets corresponding to the active DOFs of each cell. Clearly, $\bigcup_{c \in C_{\ell}} c = s_{\ell}$. Skeletonization with respect to $C_{\ell}$ then gives
\begin{align*}
 A_{\ell + 1} = \skel_{C_{\ell}} (A_{\ell}) \approx U_{\ell}^{*} A_{\ell} V_{\ell}, \qquad U_{\ell} = \prod_{c \in C_{\ell}} Q_{c} R_{\rd{c}}, \quad V_{\ell} = \prod_{c \in C_{\ell}} Q_{c} S_{\rd{c}},
\end{align*}
where the DOFs $\bigcup_{c \in C_{\ell}} \rd{c}$ have been eliminated. The matrix $A_{\ell + 1}$ is block diagonal with block partitioning
\begin{align*}
 \theta_{\ell + 1} = \left( \bigcup_{c \in C_{0}} \{ \rd{c} \} \right) \cup \cdots \cup \left( \bigcup_{c \in C_{\ell}} \{ \rd{c} \} \right) \cup \{ s_{\ell + 1} \},
\end{align*}
where $s_{\ell + 1} = s_{\ell} \setminus \bigcup_{c \in C_{\ell}} \rd{c} = \bigcup_{c \in C_{\ell}} \sk{c}$.

\subsubsection*{Level $L$}
Finally, we have $A_{L}$ and $s_{L}$, where $D \equiv A_{L}$ is block diagonal with block partitioning
\begin{align*}
 \theta_{L} = \left( \bigcup_{c \in C_{0}} \{ \rd{c} \} \right) \cup \cdots \cup \left( \bigcup_{c \in C_{L - 1}} \{ \rd{c} \} \right) \cup \{ s_{L} \}.
\end{align*}
Combining the approximation over all levels gives
\begin{align*}
 D \approx U_{L - 1}^{*} \cdots U_{0}^{*} A V_{0} \cdots V_{L - 1},
\end{align*}
where each $U_{\ell}$ and $V_{\ell}$ are products of unit triangular matrices, each of which can be inverted simply by negating its off-diagonal entries. Therefore,
\begin{subequations}
 \label{eqn:rskelf}
 \begin{align}
  A &\approx U_{0}^{-*} \cdots U_{L - 1}^{-*} D V_{L - 1}^{-1} \cdots V_{0}^{-1} \equiv F,\\
  A^{-1} &\approx V_{0} \cdots V_{L - 1} D^{-1} U_{L - 1}^{*} \cdots U_{0}^{*} = F^{-1}.
 \end{align}
\end{subequations}
The factorization $F$ permits fast multiplication and can be used as a generalized FMM. Its inverse $F^{-1}$ can be used as a direct solver at high accuracy or as a preconditioner otherwise. If $D$ is stored in factored form, e.g., as an LU decomposition, then the same factorization can readily be used for both tasks. We call \eqref{eqn:rskelf} an (approximate) generalized LU decomposition since while each $U_{\ell}$ and $V_{\ell}$ are composed of triangular factors, they are not themselves triangular, being the product of both upper and lower triangular matrices. We emphasize that $F$ and $F^{-1}$ are not assembled explicitly and are applied only in factored form.

The entire procedure is summarized compactly as Algorithm \ref{alg:rskelf}. In general, we construct the cell partitioning at each level using an adaptive quadtree \cite{samet:1984:acm-comput-surv}, which recursively subdivides the domain until each node contains only $O(1)$ DOFs.
\begin{algorithm}
 \caption{RSF.}
 \label{alg:rskelf}
 \begin{algorithmic}
  \State $A_{0} = A$ \Comment{initialize}
  \For{$\ell = 0, 1, \dots, L - 1$} \Comment{loop from finest to coarsest level}
   \State $A_{\ell + 1} = \skel_{C_{\ell}} (A_{\ell}) \approx U_{\ell}^{*} A_{\ell} V_{\ell}$ \Comment{skeletonize cells}
  \EndFor
  \State $A \approx U_{0}^{-*} \cdots U_{L - 1}^{-*} A_{L} V_{L - 1}^{-1} \cdots V_{0}^{-1}$ \Comment{generalized LU decomposition}
 \end{algorithmic}
\end{algorithm}

\subsection{Three Dimensions}
\label{sec:rskelf:3d}
Consider now the analogous setting in 3D, where $\Omega = (0, 1)^{3}$ is discretized using a uniform $n \times n \times n$ grid with $\Omega_{j} = h(j_{1} - 1, j_{1}) \times h(j_{2} - 1, j_{2}) \times h(j_{3} - 1, j_{3})$ and $x_{j} = h(j_{1} - 1/2, j_{2} - 1/2, j_{3} - 1/2)$ for $j = (j_{1}, j_{2}, j_{3})$. The total number of DOFs is $N = n^{3}$.

The algorithm extends in the natural way with cubic cells $2^{\ell} mh(j_{1} - 1, j_{1}) \times 2^{\ell} mh(j_{2} - 1, j_{2}) \times 2^{\ell} mh(j_{3} - 1, j_{3})$ about the centers $2^{\ell} mh(j_{1} - 1/2, j_{2} - 1/2, j_{3} - 1/2)$ replacing the square cells in 2D at level $\ell$ for $1 \leq j_{1}, j_{2}, j_{3} \leq 2^{L - \ell}$. With this modification, the rest of the algorithm remains unchanged. Figure \ref{fig:rskelf3} shows the active DOFs at each level for a representative example.
\begin{figure}
 \centering
 \begin{subfigure}{0.2\textwidth}
  \includegraphics[width=\textwidth]{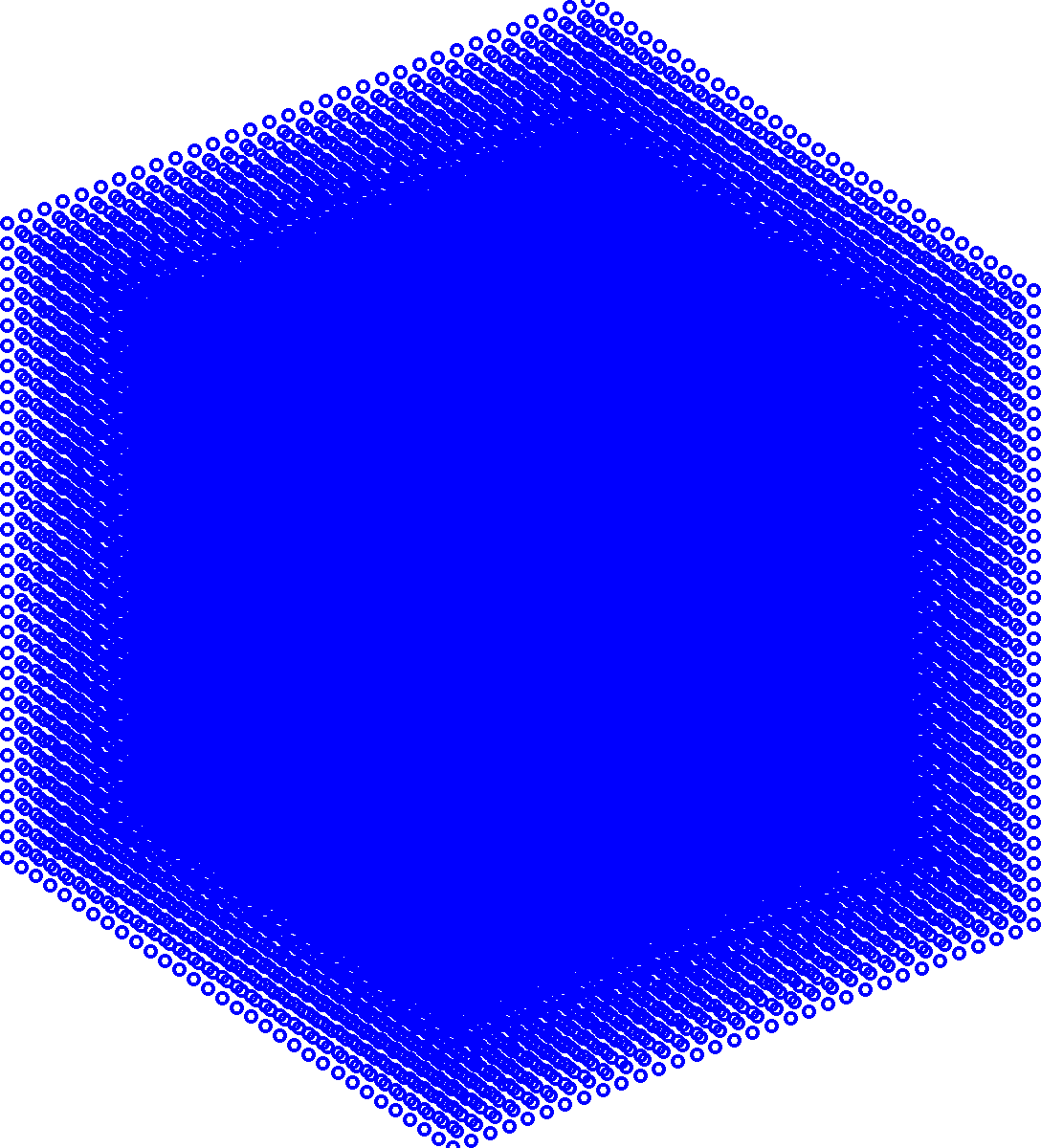}
  \caption*{$\ell = 0$}
 \end{subfigure}
 \quad
 \begin{subfigure}{0.2\textwidth}
  \includegraphics[width=\textwidth]{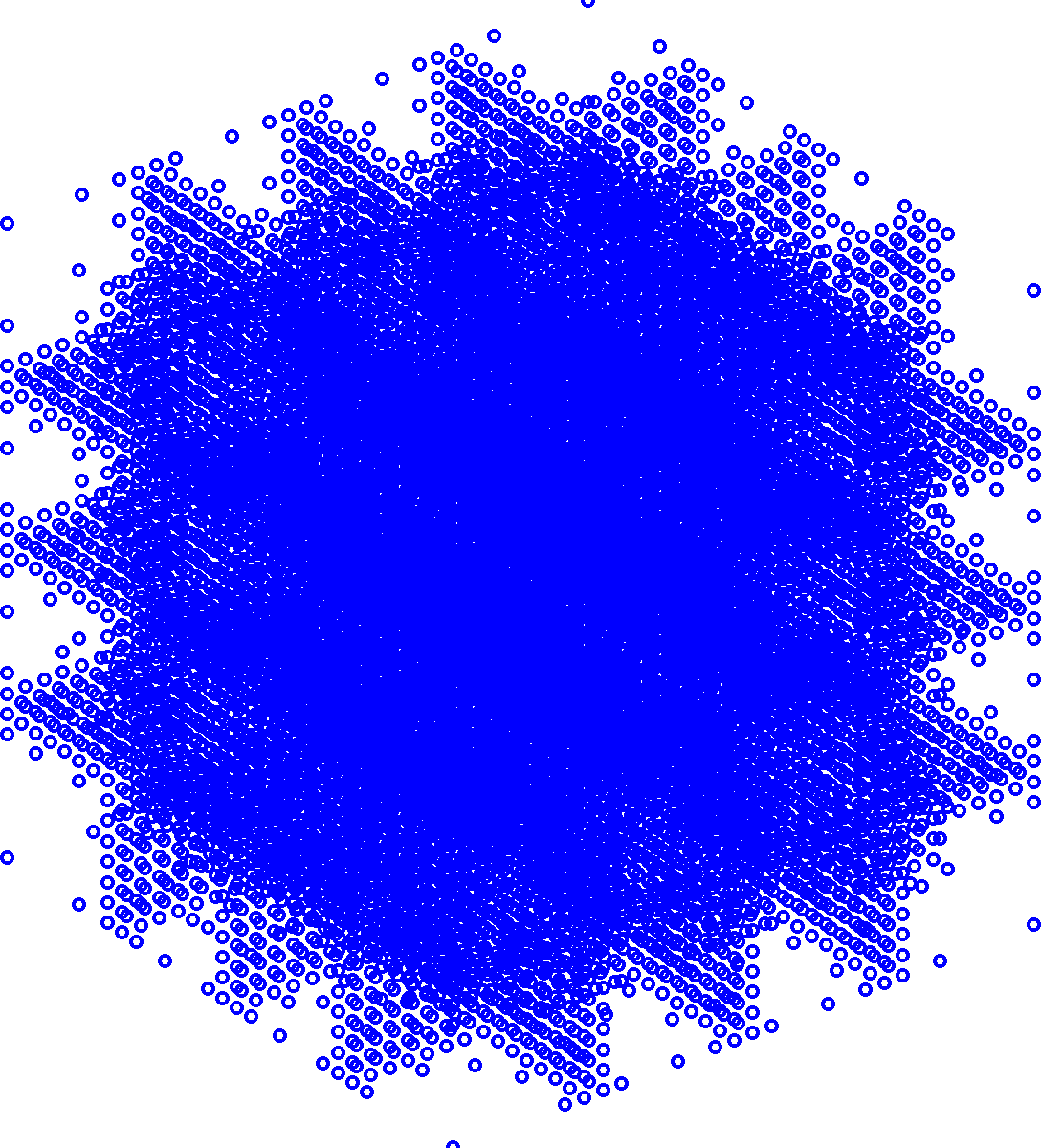}
  \caption*{$\ell = 1$}
 \end{subfigure}
 \quad
 \begin{subfigure}{0.2\textwidth}
  \includegraphics[width=\textwidth]{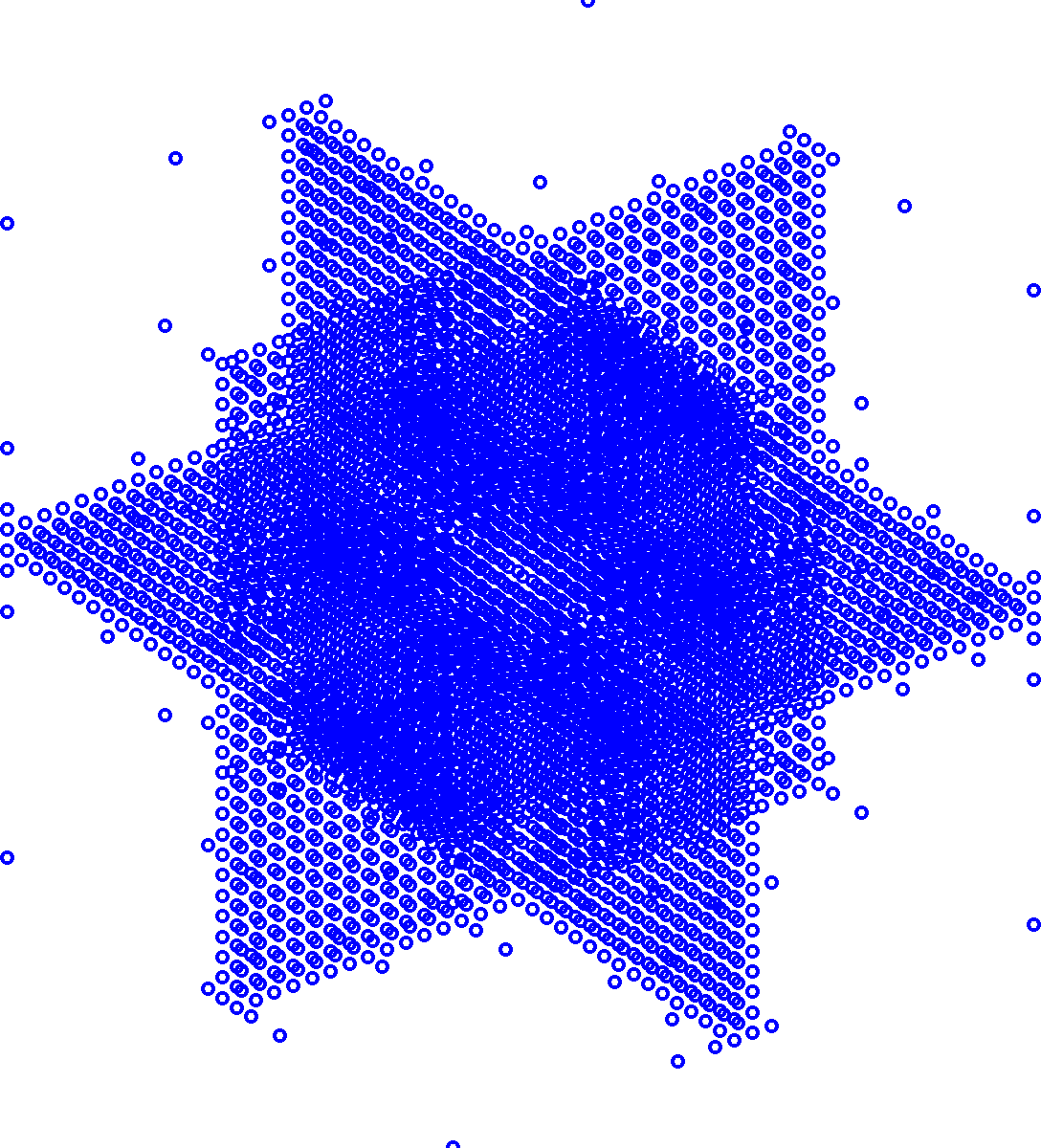}
  \caption*{$\ell = 2$}
 \end{subfigure}
 \caption{Active DOFs at each level $\ell$ of RSF in 3D.}
 \label{fig:rskelf3}
\end{figure}
The output is again a factorization of the form \eqref{eqn:rskelf}. General geometries can be treated using an adaptive octree.

\subsection{Accelerated Compression}
\label{sec:rskelf:accel-comp}
A dominant contribution to the cost of RSF is computing IDs for skeletonization. The basic operation required is the construction of an ID of
\begin{align*}
 W_{\ell,c} =
 \begin{bmatrix}
  (A_{\ell})_{c^{\cmp},c}\\
  (A_{\ell})_{c,c^{\cmp}}^{*}
 \end{bmatrix},
\end{align*}
where $c \in C_{\ell}$ and $c^{\cmp} = s_{\ell} \setminus c$, following Lemma \ref{lem:skel}. We hereafter drop the dependence on $\ell$ for notational convenience. Observe that $W_{c}$ is a tall-and-skinny matrix of size $O(N) \times |c|$, so forming its ID takes at least $O(N|c|)$ work. The total number of index sets $c \in C_{\ell}$ for all $\ell$ is $O(N)$, so considering all $W_{c}$ yields a lower bound of $O(N^{2})$ on the total work and hence on the complexity of RSF.

In principle, it is straightforward to substantially accelerate the algorithm by reconstructing an ID of $W_{c}$ from that of a much smaller matrix $Y_{c}$. All that is needed is that the rows of $Y_{c}$ span those of $W_{c}$, i.e., $\range (W_{c}^{*}) \subseteq \range (Y_{c}^{*})$, where $\range (\cdot)$ denotes the matrix range.

\begin{lemma}
 \label{lem:proxy-id}
 Let $W = XY$ with column indices $q$. If $q = \sk{q} \cup \rd{q}$ and $T_{q}$ are such that $Y_{:,\rd{q}} = Y_{:,\sk{q}} T_{q}$, then
 \begin{align*}
  W_{:,\rd{q}} = X Y_{:,\rd{q}} = X Y_{:,\sk{q}} T_{q} = W_{:,\sk{q}} T_{q}.
 \end{align*}
\end{lemma}

In other words, an ID of $Y_{c}$ gives an ID of $W_{c} = X_{c} Y_{c}$. Note that we make no explicit reference to $X_{c}$; only its existence is assumed. Of course, such a small matrix $Y_{c}$ always exists since $\rank (W_{c}) \leq |c|$; the difficulty lies in finding $Y_{c}$ {\it a priori}.

For elliptic problems, the integral kernel $K(r)$ typically satisfies some form of Green's theorem, in which its values inside a region $\mathcal{D} \in \Omega$ can be recovered from its values on the boundary $\Gamma = \partial \mathcal{D}$. Consider, for example, the Laplace kernel \eqref{eqn:laplace-kernel} and let $\varphi (x) = G(\| x - x_{0} \|)$ be the harmonic field in $\mathcal{D}$ due to an exterior source $x_{0} \in \Omega \setminus \bar{\mathcal{D}}$. Then
\begin{align*}
 \varphi (x) = \int_{\Gamma} \left[ \varphi (y) \frac{\partial G}{\partial \nu_{y}} (\| x - y \|) - \frac{\partial \varphi}{\partial \nu_{y}} (y) G(\| x - y \|) \right] d \Gamma (y), \quad x \in \mathcal{D},
\end{align*}
i.e., the ``incoming'' field $\varphi (x)$ lives in the span of single- and double-layer interactions with $\Gamma$. In practice, we will use this fact only when $x \in \mathcal{D}$ is sufficiently separated from $\Gamma$ (see below), in which case the double-layer term can often even be omitted since the corresponding discrete spaces are equal to high precision. Outgoing interactions can essentially be treated in the same way using the ``transpose'' of this idea.

In such cases, a suitable $Y_{c}$ can readily be constructed. To see this, let $B$ denote the cell containing the DOFs $c$ and draw a local ``proxy'' surface $\Gamma$ around $B$ (Figure \ref{fig:proxy-comp}).
\begin{figure}
 \includegraphics{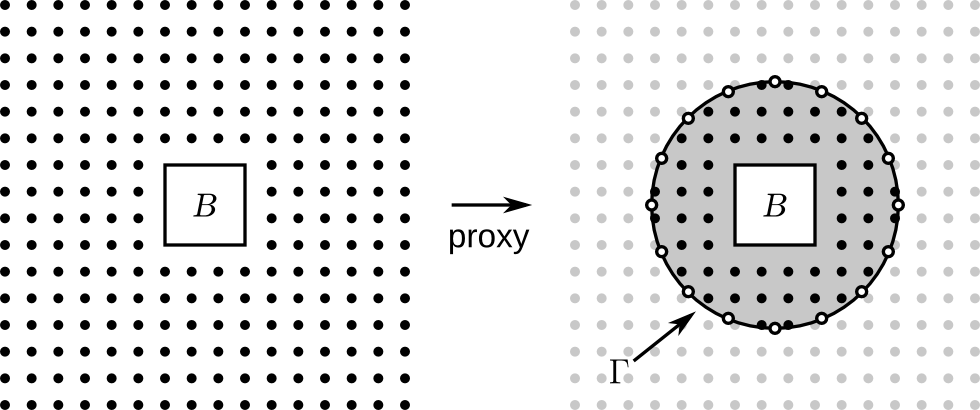}
 \caption{Accelerated compression using equivalent interactions. By Green's theorem, all off-diagonal interactions with a given box $B$ can be represented by its interactions with an artificial local proxy surface $\Gamma$ and with all DOFs interior to $\Gamma$.}
 \label{fig:proxy-comp}
\end{figure}
This partitions $c^{\cmp}$ as $c^{\cmp} = c^{\nbr} \cup c^{\far}$, where $c^{\nbr}$ consists of all DOFs interior to $\Gamma$ (the near field) and $c^{\far}$ consists of the rest (the far field). By Green's theorem, the interactions involving $c^{\far}$ can be represented by artificial ``equivalent'' interactions with $\Gamma$. Therefore, discretizing $\Gamma$ with equivalent DOFs $c^{\eqv}$, we assert that:

\begin{lemma}
 \label{lem:proxy-spec}
 Consider \eqref{eqn:ie} with $b(x) \equiv c(x) \equiv 1$ and let all quantities be as defined in the preceding discussion. Then, up to discretization error (see \cite{martinsson:2006:commun-appl-math-comput-sci}), $\range (A_{c^{\far},c}^{*}) \subseteq \range (Y_{c^{\eqv},c}^{*})$, where $(Y_{c^{\eqv},c})_{ij} = K(\| x^{\eqv}_{i} - x_{j} \|)$ for $\{ x_{j} \}$ and $\{ x^{\eqv}_{j} \}$ the points identified with the DOFs $c$ and $c^{\eqv}$, respectively.
\end{lemma}

\begin{proof}
 This immediately follows from Green's theorem upon recognizing that $A_{c^{\far},c}$ contains interactions involving only the original kernel function $K(r)$. This must be checked because $A_{:,c}$ may have Schur complement interactions (SCIs), i.e., those corresponding to the matrix $B_{\sk{p} \sk{p}}$ in \eqref{eqn:skel}, accumulated from skeletonization at previous levels, over which we do not have analytic control. However, due to the hierarchical nature of the domain partitioning, any such SCIs must be restricted to the diagonal block $A_{cc}$. Thus, Green's theorem applies.
\end{proof}

\begin{lemma}
 Consider \eqref{eqn:ie} with general $b(x)$ and $c(x)$. Then, up to discretization error, $\range (A_{c^{\far},c}^{*}) \subseteq \range (Y_{c^{\eqv},c}^{*})$ and $\range (A_{c,c^{\far}}) \subseteq \range (Y_{c,c^{\eqv}})$, where
 \begin{align*}
  (Y_{c^{\eqv},c})_{ij} = K(\| x^{\eqv}_{i} - x_{j} \|) \, c(x_{j}), \quad (Y_{c,c^{\eqv}})_{ij} = b(x_{i}) \, K(\| x_{i} - x^{\eqv}_{j} \|).
 \end{align*}
\end{lemma}

\begin{proof}
 \label{lem:proxy-gen}
 The functions $b(x)$ and $c(x)$ act as diagonal multipliers, so $A_{c^{\far},c} = B_{c^{\far}} \tilde{A}_{c^{\far},c} C_{c}$, where $\tilde{A}_{c^{\far},c}$ is the corresponding interaction matrix with $b(x) \equiv c(x) \equiv 1$ (i.e., that in Lemma \ref{lem:proxy-spec}), and $B_{c^{\far}} = \diag (b(x^{\far}_{i}))$ and $C_{c} = \diag (c(x_{i}))$ for $\{ x^{\far}_{j} \}$ the points attached to $c^{\far}$. By Lemma \ref{lem:proxy-spec}, $\tilde{A}_{c^{\far},c} = \tilde{X}_{c^{\eqv},c} \tilde{Y}_{c^{\eqv},c}$ for some $\tilde{X}_{c^{\eqv},c}$, so
 \begin{align*}
  A_{c^{\far},c} = B_{c^{\far}} \tilde{X}_{c^{\eqv},c} \tilde{Y}_{c^{\eqv},c} C_{c} = \left( B_{c^{\far}} \tilde{X}_{c^{\eqv},c} \right) \left( \tilde{Y}_{c^{\eqv},c} C_{c} \right) \equiv X_{c^{\eqv},c} Y_{c^{\eqv},c}.
 \end{align*}
 A similar argument with $A_{c,c^{\far}} = B_{c} \tilde{A}_{c,c^{\far}} C_{c^{\far}}$ analogously defined and
 \begin{align*}
  \tilde{A}_{c,c^{\far}} = \tilde{A}_{c^{\far},c}^{\trans} = \tilde{Y}_{c^{\eqv},c}^{\trans} \tilde{X}_{c^{\eqv},c}^{\trans} \equiv \tilde{Y}_{c,c^{\eqv}} \tilde{X}_{c,c^{\eqv}}
 \end{align*}
 proves that $A_{c,c^{\far}} = Y_{c,c^{\eqv}} X_{c,c^{\eqv}}$ for some $X_{c,c^{\eqv}}$.
\end{proof}

If $\Gamma$ is separated from $B$, for example as in Figure \ref{fig:proxy-comp}, then standard multipole estimates \cite{greengard:1987:j-comput-phys,greengard:1997:acta-numer} show that we only need $|c^{\eqv}| = O(\log^{d - 1} (1 / \epsilon))$ to satisfy Green's theorem to any precision $\epsilon$. In particular, for fixed $\epsilon$, we can choose $|c^{\eqv}|$ to be constant. Therefore, Lemma \ref{lem:proxy-gen} gives
\begin{align}
 W_{c} \approx X_{c} Y_{c} \equiv X_{c}
 \begin{bmatrix}
  A_{c^{\nbr},c}\\
  A_{c,c^{\nbr}}^{*}\\
  Y_{c^{\eqv},c}\\
  Y_{c,c^{\eqv}}^{*}
 \end{bmatrix}
 \label{eqn:proxy-id}
\end{align}
for some $X_{c}$, where $Y_{c}$ has size $O(|c^{\nbr}| + 1) \times |c|$ with $|c^{\nbr}| = O(|c|)$ typically. Lemma \ref{lem:proxy-id} then reduces the global compression of $W_{c}$ to the local compression of $Y_{c}$. This so-called proxy trick has also been employed by \cite{cheng:2005:siam-j-sci-comput,corona:2015:appl-comput-harmon-anal,gillman:2012:front-math-china,greengard:2009:acta-numer,ho:2012:siam-j-sci-comput,martinsson:2005:j-comput-phys,martinsson:2007:siam-j-sci-comput,pan:2012:radio-sci,ying:2004:j-comput-phys} and is crucial for reducing the asymptotic complexity. For numerical stability, we include the quadrature weights for the integral \eqref{eqn:kernel-integral} in $Y_{c^{\eqv},c}$ and $Y_{c,c^{\eqv}}$ so that the various components of $Y_{c}$ are all of the same order.

In this paper, for a cell $B$ with scaled width $1$ centered at the origin, we take as $\Gamma$ the circle of radius $3/2$ in 2D, uniformly discretized with $64$ points, and the sphere of radius $3/2$ in 3D, uniformly sampled (by projecting Gaussian random vectors) with $512$ points. These values of $|c^{\eqv}|$ have been experimentally validated to reproduce interactions via the Laplace kernel \eqref{eqn:laplace-kernel} with $\epsilon \sim 10^{-15}$. This approach is more efficient than the ``supercell'' proxy of \cite{greengard:2009:acta-numer,ho:2012:siam-j-sci-comput} by factors of $4 / \pi = 1.2732...$ in 2D and $6 / \pi = 1.9099...$ in 3D (volume ratio of the cube to the sphere of equal diameter), which takes as $\Gamma$ the outer boundary of the $3 \times 3$ ($\times 3$) cell block centered at $B$.

\subsection{Complexity Estimates}
We now investigate the computational complexity of RSF. For this, we need to estimate the skeleton size $|\sk{c}|$ for a typical index set $c \in C_{\ell}$ at level $\ell$. Denote this quantity by $k_{\ell}$ and let $n_{\ell} = (2^{\ell} m)^{d} = O(2^{d \ell})$ be the number of DOFs (both active and inactive) in each cell. From Figures \ref{fig:rskelf2} and \ref{fig:rskelf3}, it is clear that skeletons tend to cluster around cell interfaces, which can again be justified by Green's theorem, so $k_{\ell} = O(n_{\ell}^{1/2}) = O(2^{\ell})$ in 2D and $k_{\ell} = O(n_{\ell}^{2/3}) = O(2^{2 \ell})$ in 3D. Indeed, this can be verified using standard multipole estimates by noting that $k_{\ell}$ is on the order of the interaction rank between two adjacent cells at level $\ell$, which can be analyzed via recursive subdivision to expose well-separated structures (Figure \ref{fig:recur-subdiv}).
\begin{figure}
 \includegraphics{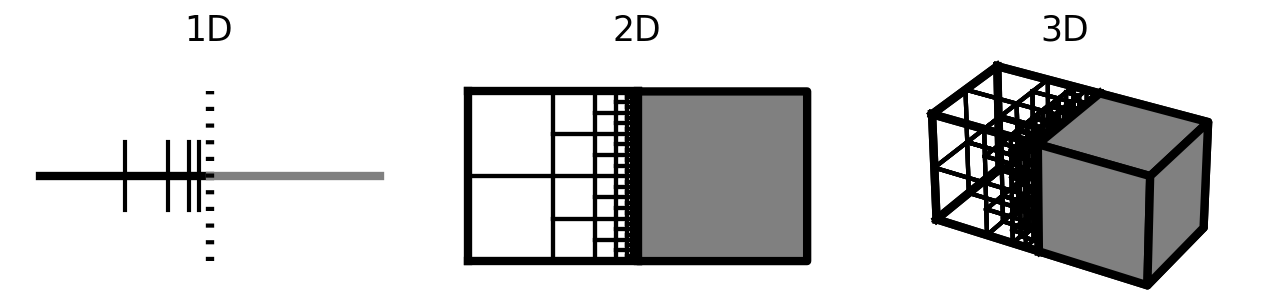}
 \caption{Recursive subdivision of source domain (white) into well-separated subdomains from the target (gray), each of which has constant interaction rank.}
 \label{fig:recur-subdiv}
\end{figure}
This yields the more detailed result
\begin{align}
 k_{\ell} =
 \begin{cases}
  O(\ell), & d = 1\\
  O(2^{(d - 1) \ell}), & d \geq 2,
 \end{cases}
 \label{eqn:inter-rank}
\end{align}
which, in fact, holds for $d$ equal to the intrinsic dimension rather than the ambient dimension.

\begin{theorem}[\cite{ho:2012:siam-j-sci-comput,martinsson:2005:j-comput-phys}]
 \label{thm:rskelf}
 Assume that \eqref{eqn:inter-rank} holds. Then the cost of constructing the factorization $F$ in \eqref{eqn:rskelf} using RSF with accelerated compression is
 \begin{align}
  t_{f} = O(2^{dL} m^{3d}) + \sum_{\ell = 0}^{L} 2^{d(L - \ell)} O(k_{\ell}^{3}) =
  \begin{cases}
   O(N), & d = 1\\
   O(N^{3(1 - 1/d)}), & d \geq 2,
  \end{cases}
  \label{eqn:rskelf-t_f}
 \end{align}
 while that of applying $F$ or $F^{-1}$ is
 \begin{align}
  t_{a/s} = O(2^{dL} m^{2d}) + \sum_{\ell = 0}^{L} 2^{d(L - \ell)} O(k_{\ell}^{2}) =
  \begin{cases}
   O(N), & d = 1\\
   O(N \log N), & d = 2\\
   O(N^{2(1 - 1/d)}), & d \geq 3.
  \end{cases}
  \label{eqn:rskelf-t_as}
 \end{align}
\end{theorem}

\begin{proof}
 Consider first the factorization cost $t_{f}$. There are $2^{d(L - \ell)}$ cells at level $\ell$, where each cell $c \in C_{\ell}$ requires the calculation of an ID of $Y_{c}$ in \eqref{eqn:proxy-id} as well as various local matrix operations at a total cost of $O(|c|^{3})$, assuming that $|c^{\nbr}| = O(|c|)$. But $|c| = m^{d}$ for $\ell = 0$, while $|c| = O(k_{\ell - 1}) = O(k_{\ell})$ for $\ell \geq 1$ since the active DOFs $c$ are obtained by merging the skeletons of $2^{d}$ cells at level $\ell - 1$. Hence, \eqref{eqn:rskelf-t_f} follows.

 A similar derivation holds for $t_{a/s}$ by observing that each $c \in C_{\ell}$ requires local matrix-vector products with cost $O(|c|^{2})$.
\end{proof}

\begin{remark}
 If a tree is used, then there is also a cost of $O(N \log N)$ for tree construction, but the associated constant is tiny and so we can ignore it for all practical purposes.
\end{remark}

The memory cost to store $F$ itself is clearly $m_{f} = O(t_{a/s})$ and so is also given by \eqref{eqn:rskelf-t_as}. From Theorem \ref{thm:rskelf}, it is immediate that RSF behaves just like MF, with the geometric growth of $k_{\ell}$ in 2D and 3D leading to suboptimal complexities.

\begin{corollary}
 \label{cor:complexity-1d}
 If
 \begin{align}
  k_{\ell} = O(k \ell)
  \label{eqn:inter-rank-1d}
 \end{align}
 for some constant $k$, then $t_{f} = O(N k^{2})$ and $t_{a/s} = O(Nk)$.
\end{corollary}

\begin{proof}
 From \eqref{eqn:rskelf-t_f}, $t_{f} = O(2^{dL} (m^{d} + k)^{3})$, so choosing $m^{d} = O(k)$ gives $N = n^{d} = (2^{L} m)^{d} = O(2^{dL} k)$ and $t_{f} = O(2^{dL} k^{3}) = O(N k^{2})$. Similarly, $t_{a/s} = O(2^{dL} (m^{d} + k)^{2}) = O(2^{dL} k^{2}) = O(Nk)$.
\end{proof}

This is a more precise version of the 1D result that will be useful later when discussing HIF-IE.

\section{Hierarchical Interpolative Factorization}
\label{sec:hifie}
In this section, we present HIF-IE, which builds upon RSF by introducing additional levels of skeletonization in order to effectively reduce all problems to 1D. Considering the 2D case for concreteness, the main idea is simply to employ an additional level $\ell + 1/2$ after each level $\ell$ by partitioning $\Omega$ according to the cell edges near which the surviving active DOFs cluster. This fully exploits the 1D geometry of the active DOFs. However, the algorithm is complicated by the fact that the cell and edge partitions are non-nested, so different index groups may now interact via SCIs. Such SCIs do not lend themselves easily to analysis and we have yet to prove a statement like \eqref{eqn:inter-rank} on their ranks. Nevertheless, extensive numerical experiments by ourselves (Section \ref{sec:results}) and others \cite{corona:2015:appl-comput-harmon-anal} reveal that very similar bounds appear to be obeyed. This suggests that SCIs do not need to be treated in any significantly different way, and we hereafter assume that interaction rank is completely determined by geometry.

The overall approach of HIF-IE is closely related to that of \cite{corona:2015:appl-comput-harmon-anal}, but our sparsification framework permits a much simpler implementation and analysis. As with RSF, we begin first in 2D before extending to 3D.

\subsection{Two Dimensions}
Assume the same setup as in Section \ref{sec:rskelf:2d}. HIF-IE supplements cell skeletonization (2D to 1D) at level $\ell$ with edge skeletonization (1D to 0D) at level $\ell + 1/2$ for each $\ell = 0, 1, \dots, L - 1$. Figure \ref{fig:hifie2} shows the active DOFs at each level for a representative example.
\begin{figure}
 \centering
 \begin{subfigure}{0.2\textwidth}
  \includegraphics[width=\textwidth]{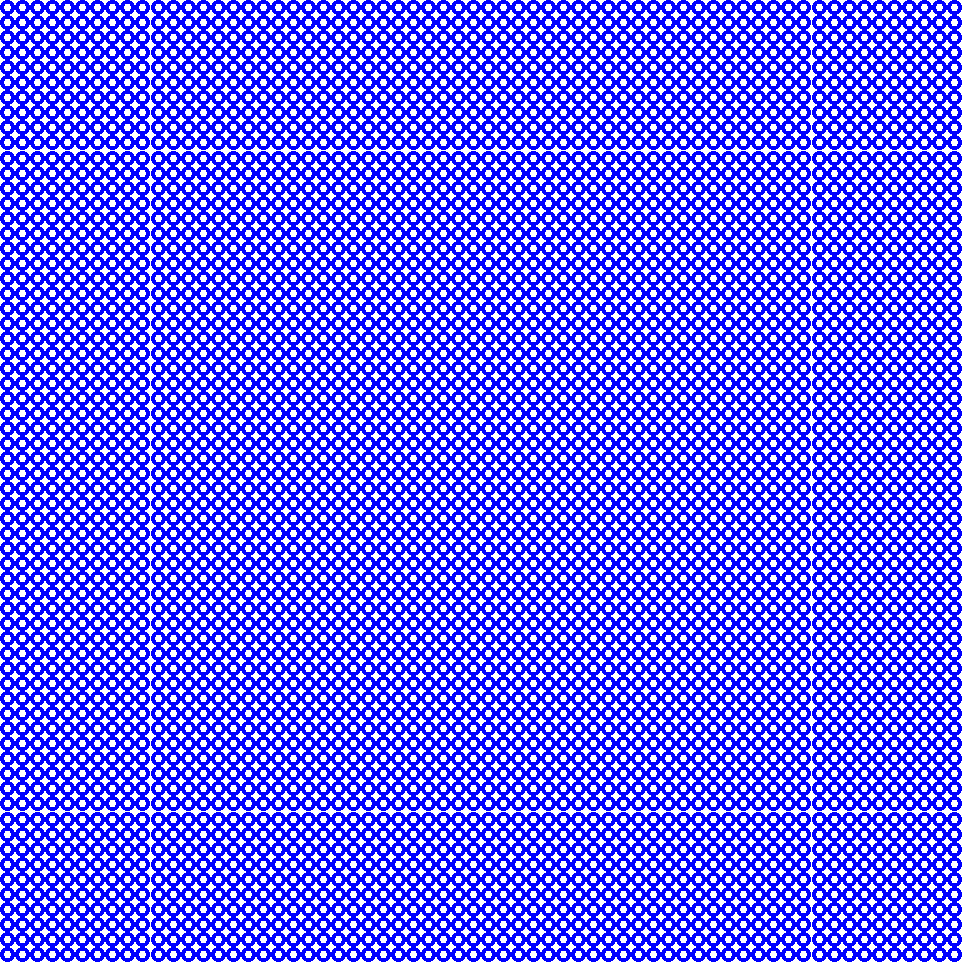}
  \caption*{$\ell = 0$}
 \end{subfigure}
 \quad
 \begin{subfigure}{0.2\textwidth}
  \includegraphics[width=\textwidth]{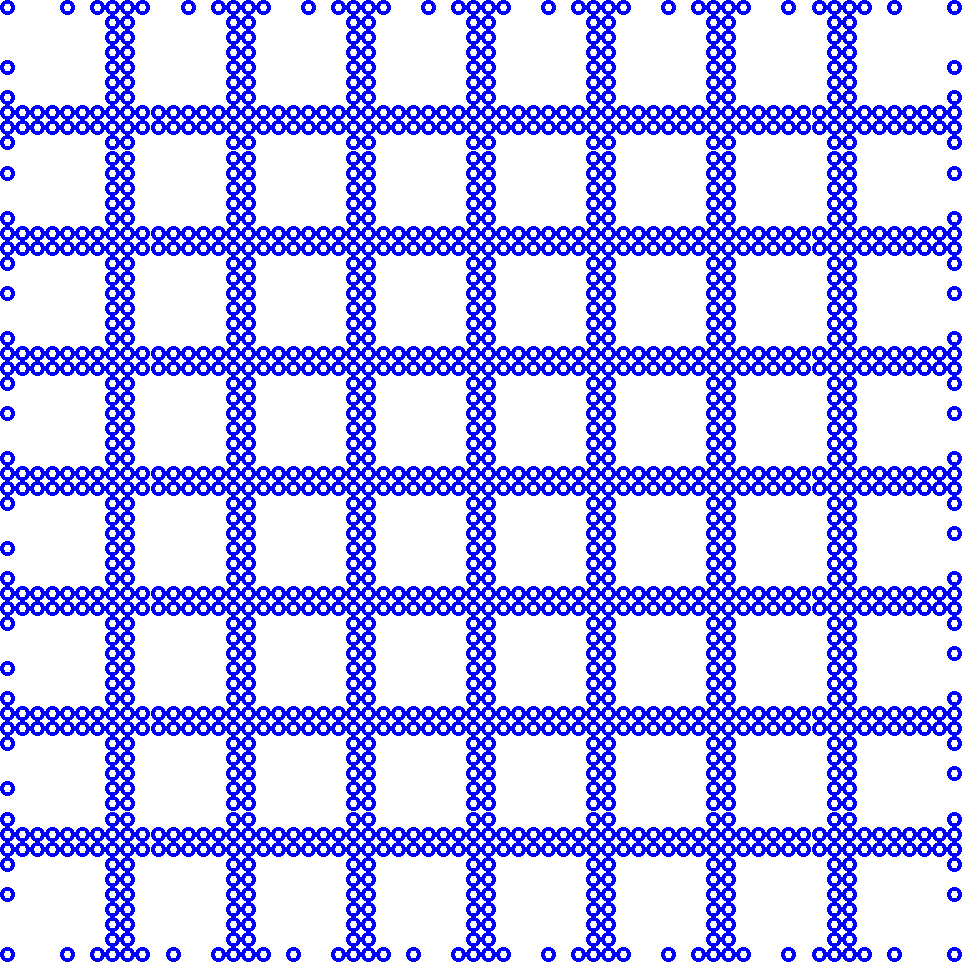}
  \caption*{$\ell = 1/2$}
 \end{subfigure}
 \quad
 \begin{subfigure}{0.2\textwidth}
  \includegraphics[width=\textwidth]{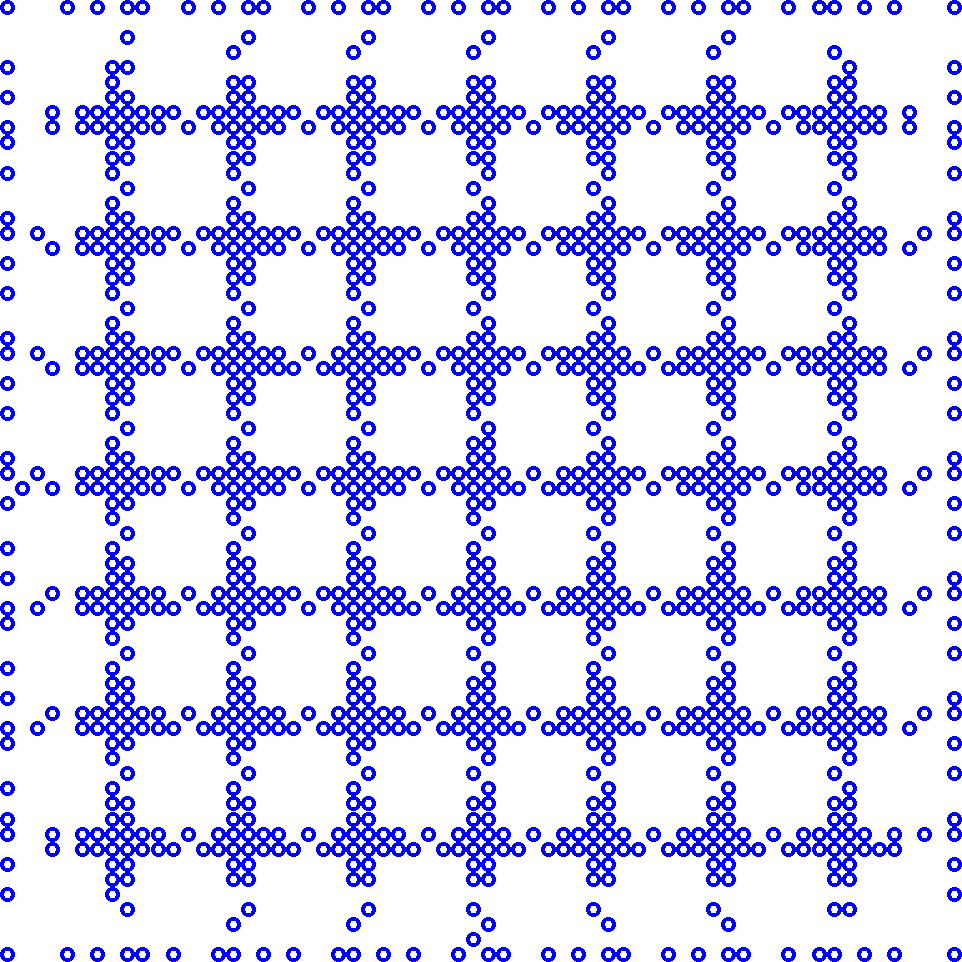}
  \caption*{$\ell = 1$}
 \end{subfigure}
 \quad
 \begin{subfigure}{0.2\textwidth}
  \includegraphics[width=\textwidth]{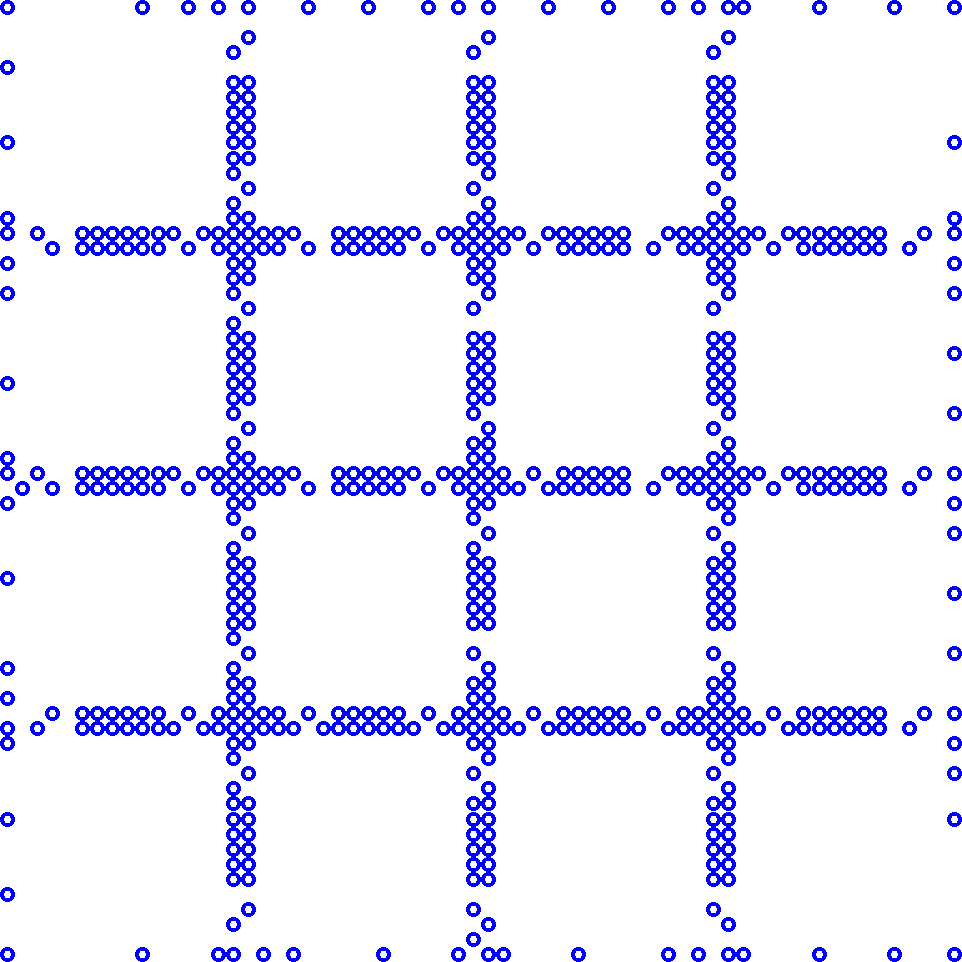}
  \caption*{$\ell = 3/2$}
 \end{subfigure}
 \\~\\~\\
 \begin{subfigure}{0.2\textwidth}
  \includegraphics[width=\textwidth]{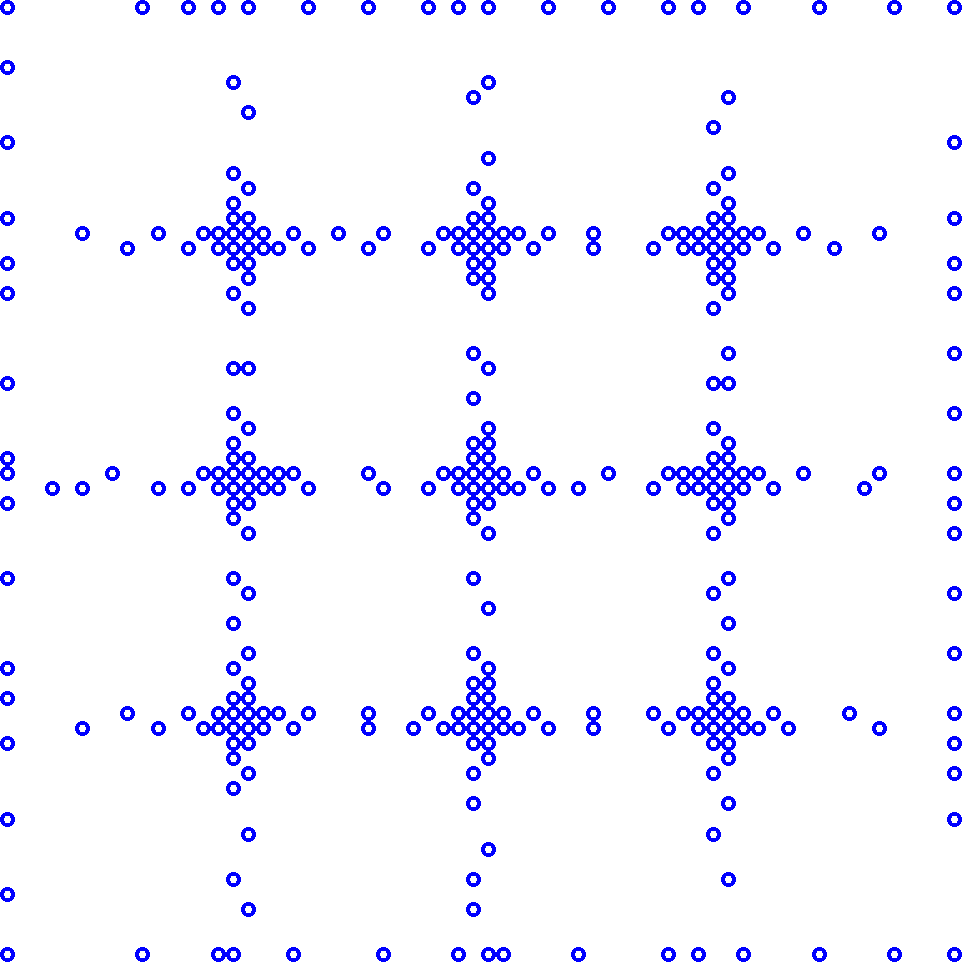}
  \caption*{$\ell = 2$}
 \end{subfigure}
 \quad
 \begin{subfigure}{0.2\textwidth}
  \includegraphics[width=\textwidth]{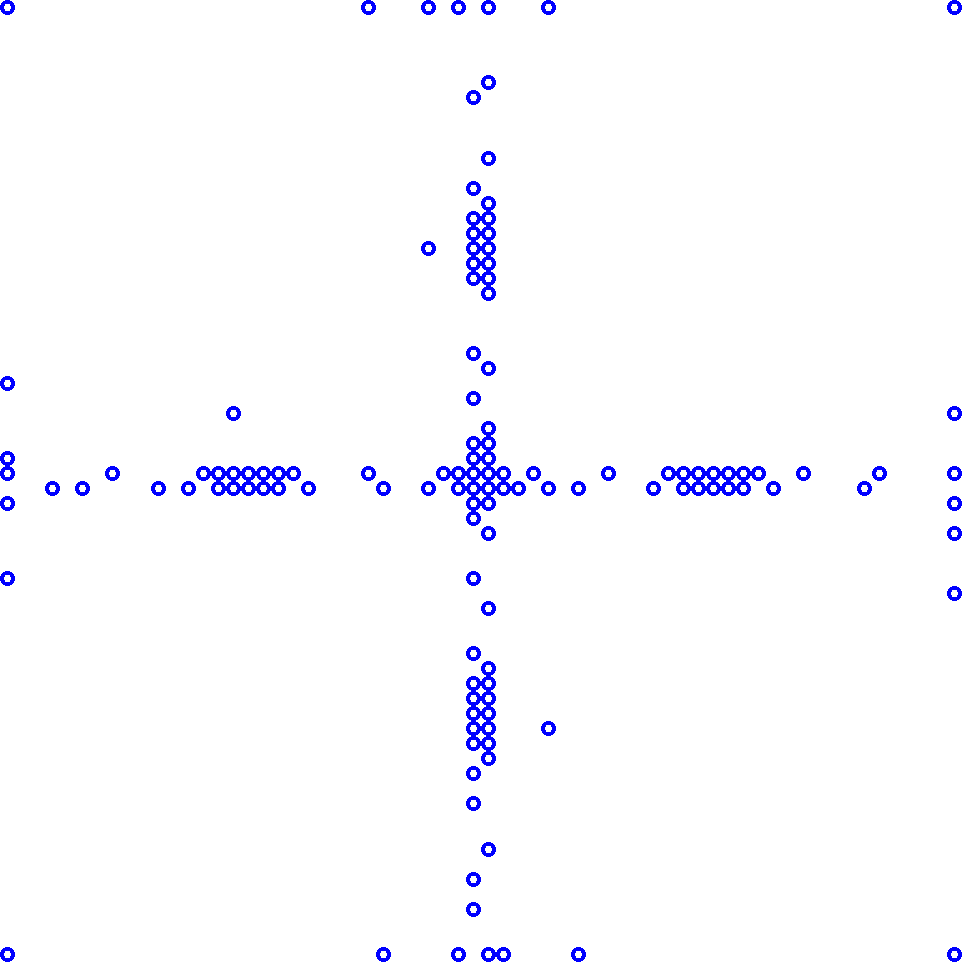}
  \caption*{$\ell = 5/2$}
 \end{subfigure}
 \quad
 \begin{subfigure}{0.2\textwidth}
  \includegraphics[width=\textwidth]{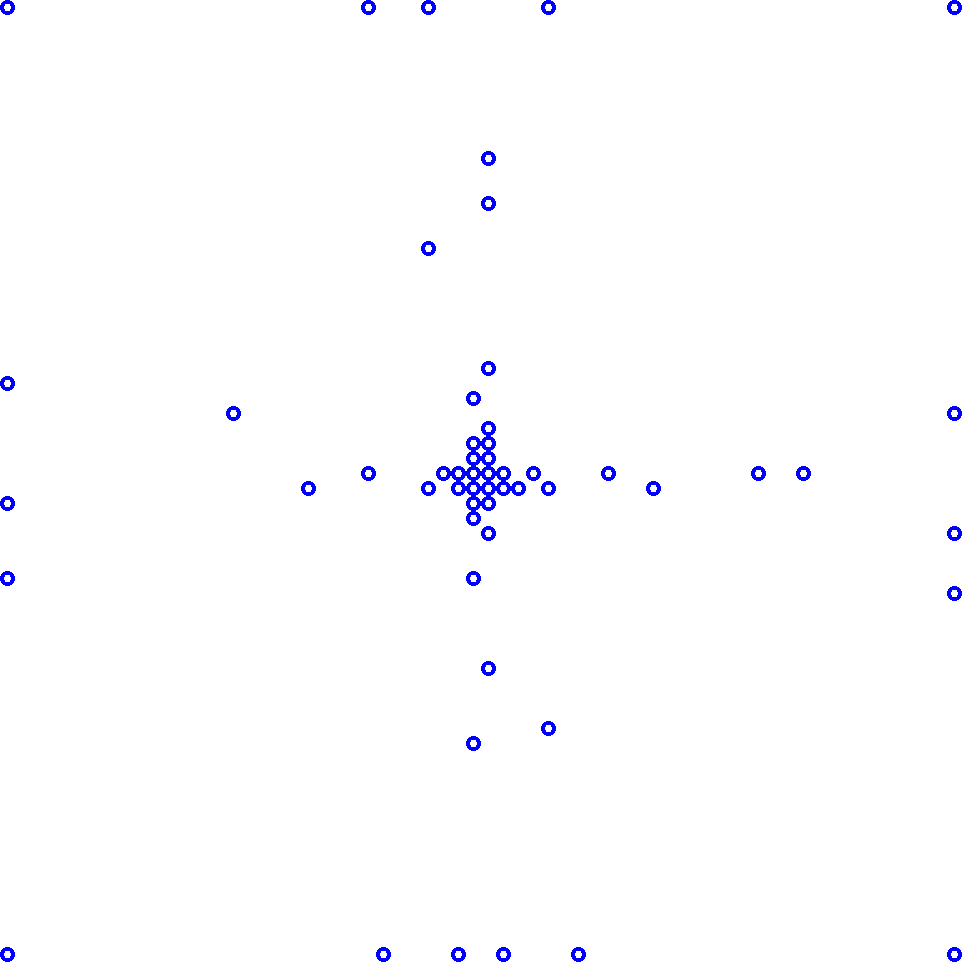}
  \caption*{$\ell = 3$}
 \end{subfigure}
 \caption{Active DOFs at each level $\ell$ of HIF-IE in 2D.}
 \label{fig:hifie2}
\end{figure}

\subsubsection*{Level $\ell$}
Partition $\Omega$ into Voronoi cells about the cell centers $2^{\ell} mh(j_{1} - 1/2, j_{2} - 1/2)$ for $1 \leq j_{1}, j_{2} \leq 2^{L - \ell}$. Let $C_{\ell}$ be the collection of index sets corresponding to the active DOFs of each cell. Skeletonization with respect to $C_{\ell}$ then gives
\begin{align*}
 A_{\ell + 1/2} = \skel_{C_{\ell}} (A_{\ell}) \approx U_{\ell}^{*} A_{\ell} V_{\ell}, \qquad U_{\ell} = \prod_{c \in C_{\ell}} Q_{c} R_{\rd{c}}, \quad V_{\ell} = \prod_{c \in C_{\ell}} Q_{c} S_{\rd{c}},
\end{align*}
where the DOFs $\bigcup_{c \in C_{\ell}} \rd{c}$ have been eliminated.

\subsubsection*{Level $\ell + 1/2$}
Partition $\Omega$ into Voronoi cells about the edge centers $2^{\ell} mh(j_{1}, j_{2} - 1/2)$ for $1 \leq j_{1} \leq 2^{L - \ell} - 1$, $1 \leq j_{2} \leq 2^{L - \ell}$ and $2^{\ell} mh(j_{1} - 1/2, j_{2})$ for $1 \leq j_{1} \leq 2^{L - \ell}$, $1 \leq j_{2} \leq 2^{L - \ell} - 1$. Let $C_{\ell + 1/2}$ be the collection of index sets corresponding to the active DOFs of each cell. Skeletonization with respect to $C_{\ell + 1/2}$ then gives
\begin{multline*}
 A_{\ell + 1} = \skel_{C_{\ell + 1/2}} (A_{\ell + 1/2}) \approx U_{\ell + 1/2}^{*} A_{\ell + 1/2} V_{\ell + 1/2},\\ U_{\ell + 1/2} = \prod_{c \in C_{\ell + 1/2}} Q_{c} R_{\rd{c}}, \quad V_{\ell + 1/2} = \prod_{c \in C_{\ell + 1/2}} Q_{c} S_{\rd{c}},
\end{multline*}
where the DOFs $\bigcup_{c \in C_{\ell + 1/2}} \rd{c}$ have been eliminated.

\subsubsection*{Level $L$}
Combining the approximation over all levels gives
\begin{align*}
 D \equiv A_{L} \approx U_{L - 1/2}^{*} \cdots U_{1/2}^{*} U_{0}^{*} A V_{0} V_{1/2} \cdots V_{L - 1/2},
\end{align*}
so
\begin{subequations}
 \label{eqn:hifie2}
 \begin{align}
  A &\approx U_{0}^{-*} U_{1/2}^{-*} \cdots U_{L - 1/2}^{-*} D V_{L - 1/2}^{-1} \cdots V_{1/2}^{-1} V_{0}^{-1} \equiv F,\\
  A^{-1} &\approx V_{0} V_{1/2} \cdots V_{L - 1/2} D^{-1} U_{L - 1/2}^{*} \cdots U_{1/2}^{*} U_{0}^{*} = F^{-1}.
 \end{align}
\end{subequations}
This is a factorization of exactly the same type as that in \eqref{eqn:rskelf} (but with twice the number of factors). The entire procedure is summarized as Algorithm \ref{alg:hifie2}.
\begin{algorithm}
 \caption{HIF-IE in 2D.}
 \label{alg:hifie2}
 \begin{algorithmic}
  \State $A_{0} = A$ \Comment{initialize}
  \For{$\ell = 0, 1, \dots, L - 1$} \Comment{loop from finest to coarsest level}
   \State $A_{\ell + 1/2} = \skel_{C_{\ell}} (A_{\ell}) \approx U_{\ell}^{*} A_{\ell} V_{\ell}$ \Comment{skeletonize cells}
   \State $A_{\ell + 1} = \skel_{C_{\ell + 1/2}} (A_{\ell + 1/2}) \approx U_{\ell + 1/2}^{*} A_{\ell + 1/2} V_{\ell + 1/2}$ \Comment{skeletonize edges}
  \EndFor
  \State $A \approx U_{0}^{-*} U_{1/2}^{-*} \cdots U_{L - 1/2}^{-*} A_{L} V_{L - 1/2}^{-1} \cdots V_{1/2}^{-1} V_{0}^{-1}$ \Comment{generalized LU decomposition}
 \end{algorithmic}
\end{algorithm}

\subsection{Three Dimensions}
Assume the same setup as in Section \ref{sec:rskelf:3d}. HIF-IE now performs two rounds of additional dimensional reduction over RSF by supplementing cell skeletonization (3D to 2D) at level $\ell$ with face skeletonization (2D to 1D) at level $\ell + 1/3$ and edge skeletonization (1D to 0D) at level $\ell + 2/3$. Figure \ref{fig:hifie3} shows the active DOFs at each level for a representative example.
\begin{figure}
 \centering
 \begin{subfigure}{0.2\textwidth}
  \includegraphics[width=\textwidth]{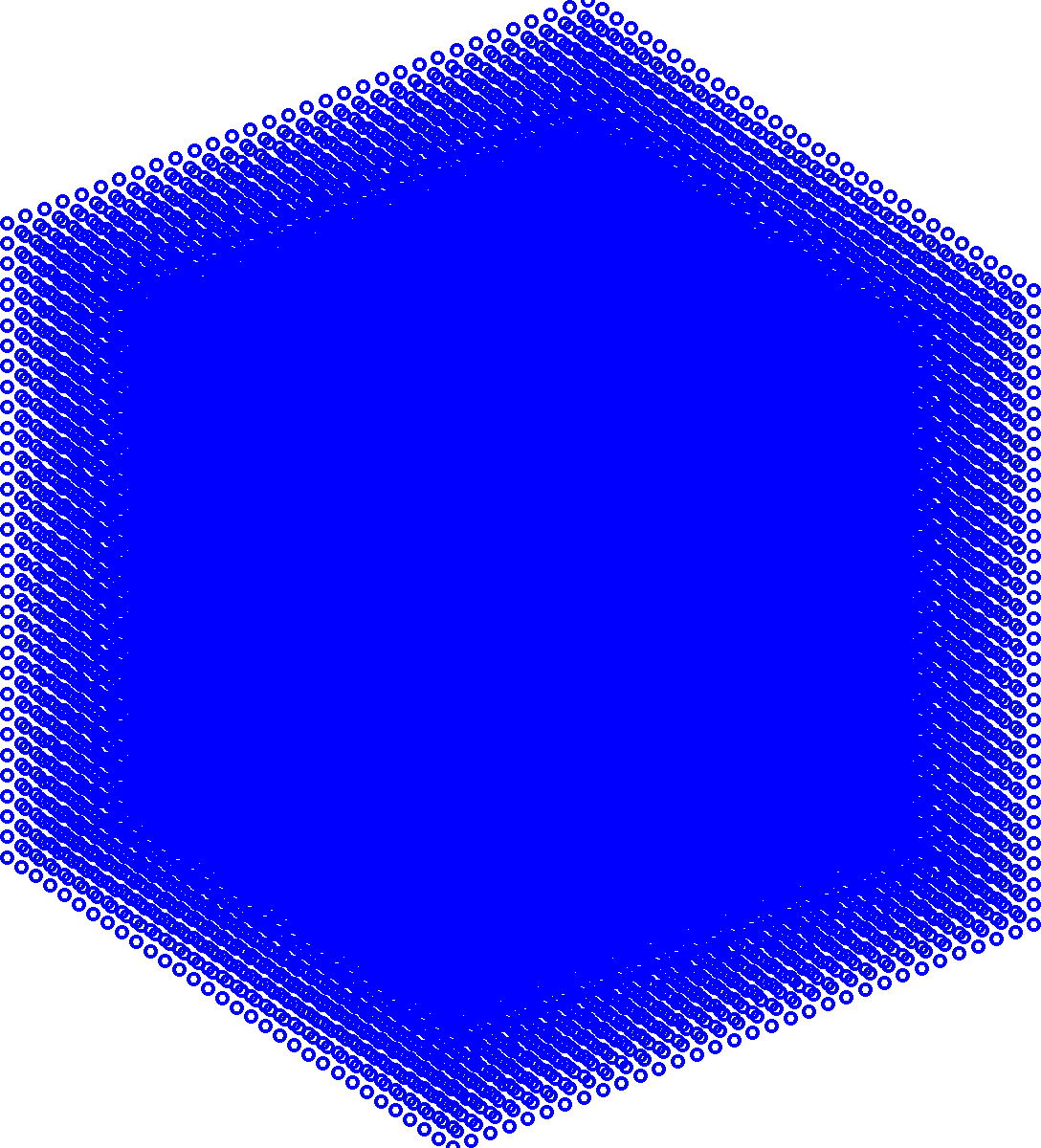}
  \caption*{$\ell = 0$}
 \end{subfigure}
 \quad
 \begin{subfigure}{0.2\textwidth}
  \includegraphics[width=\textwidth]{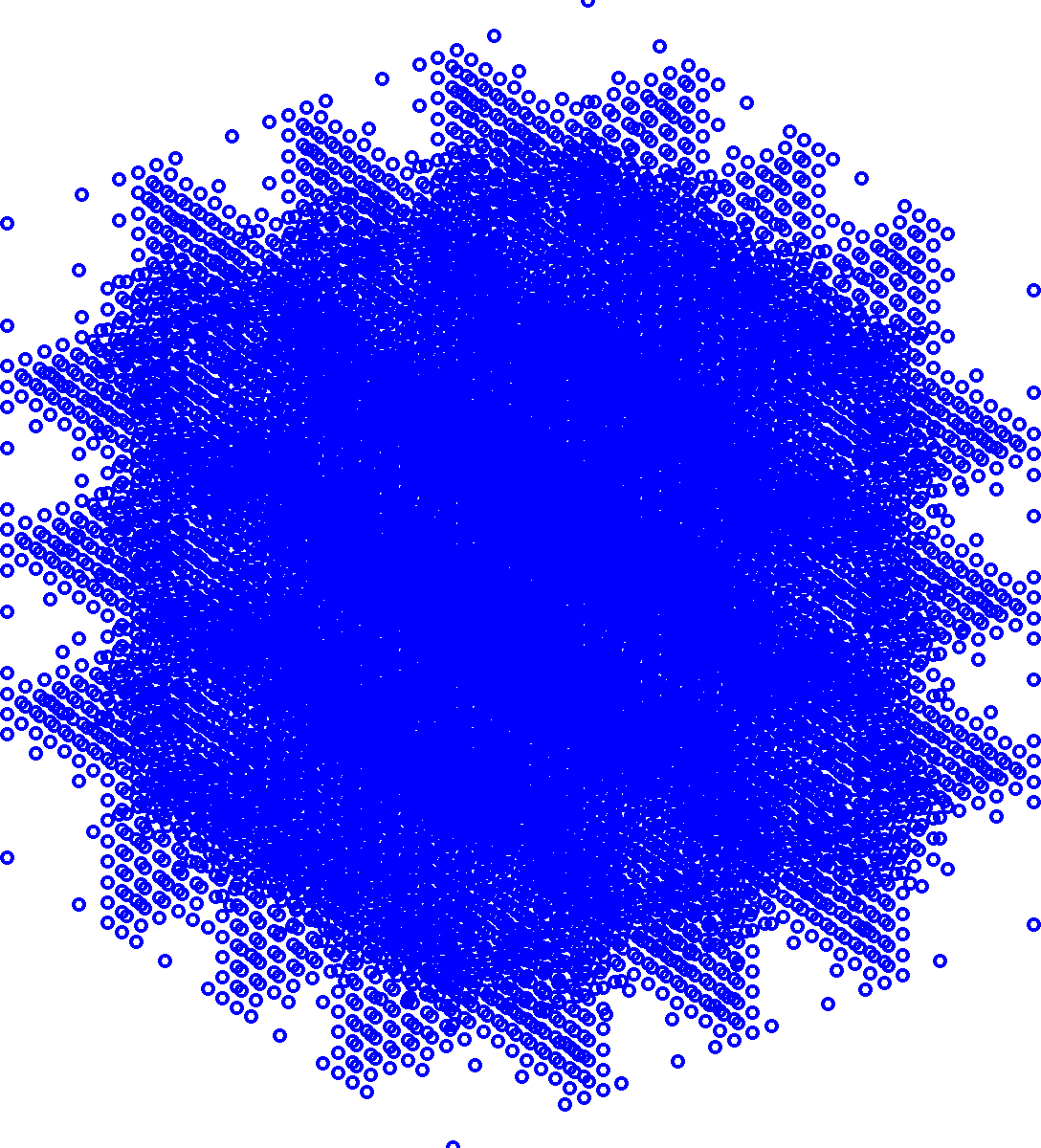}
  \caption*{$\ell = 1/3$}
 \end{subfigure}
 \quad
 \begin{subfigure}{0.2\textwidth}
  \includegraphics[width=\textwidth]{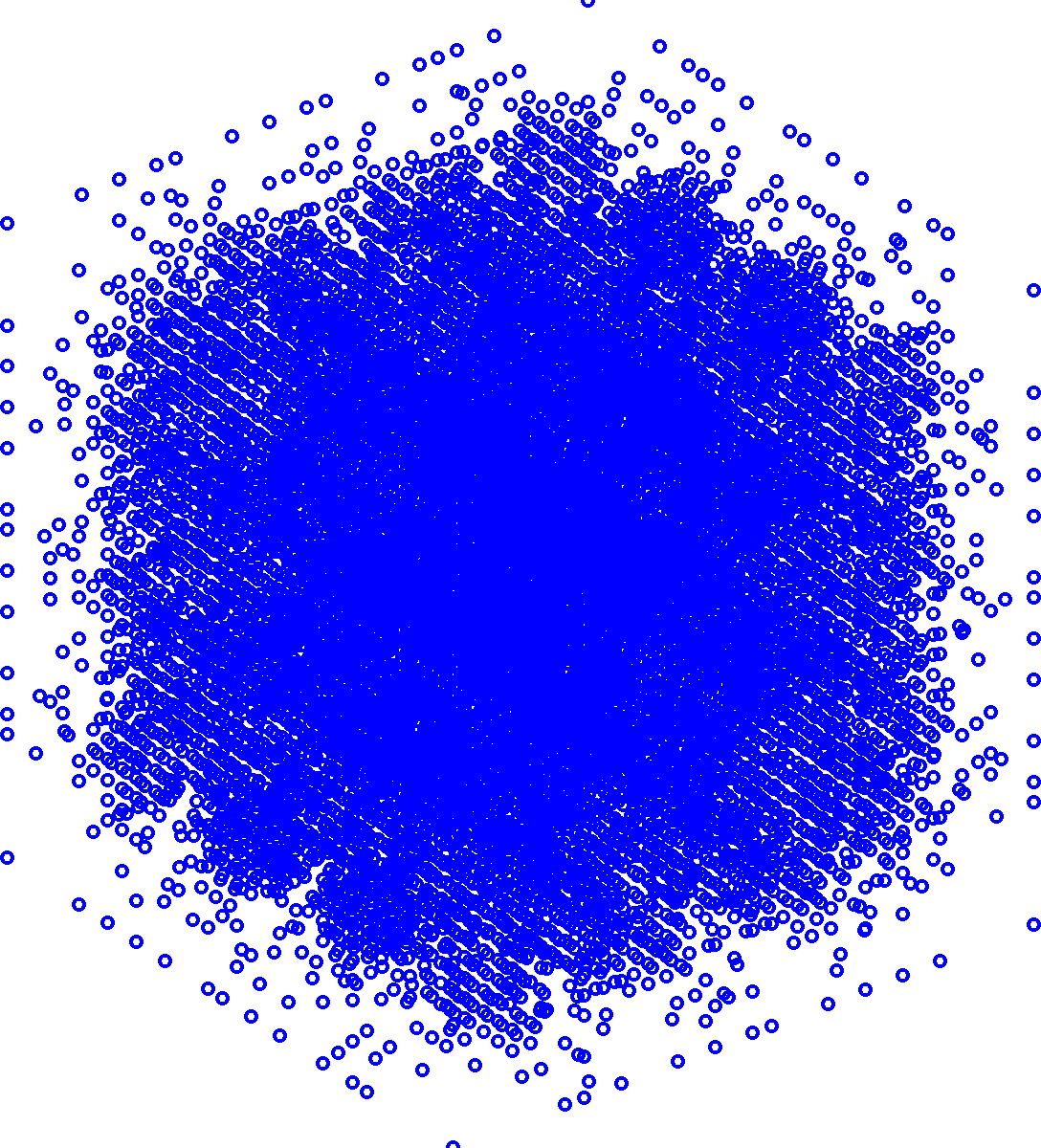}
  \caption*{$\ell = 2/3$}
 \end{subfigure}
 \quad
 \begin{subfigure}{0.2\textwidth}
  \includegraphics[width=\textwidth]{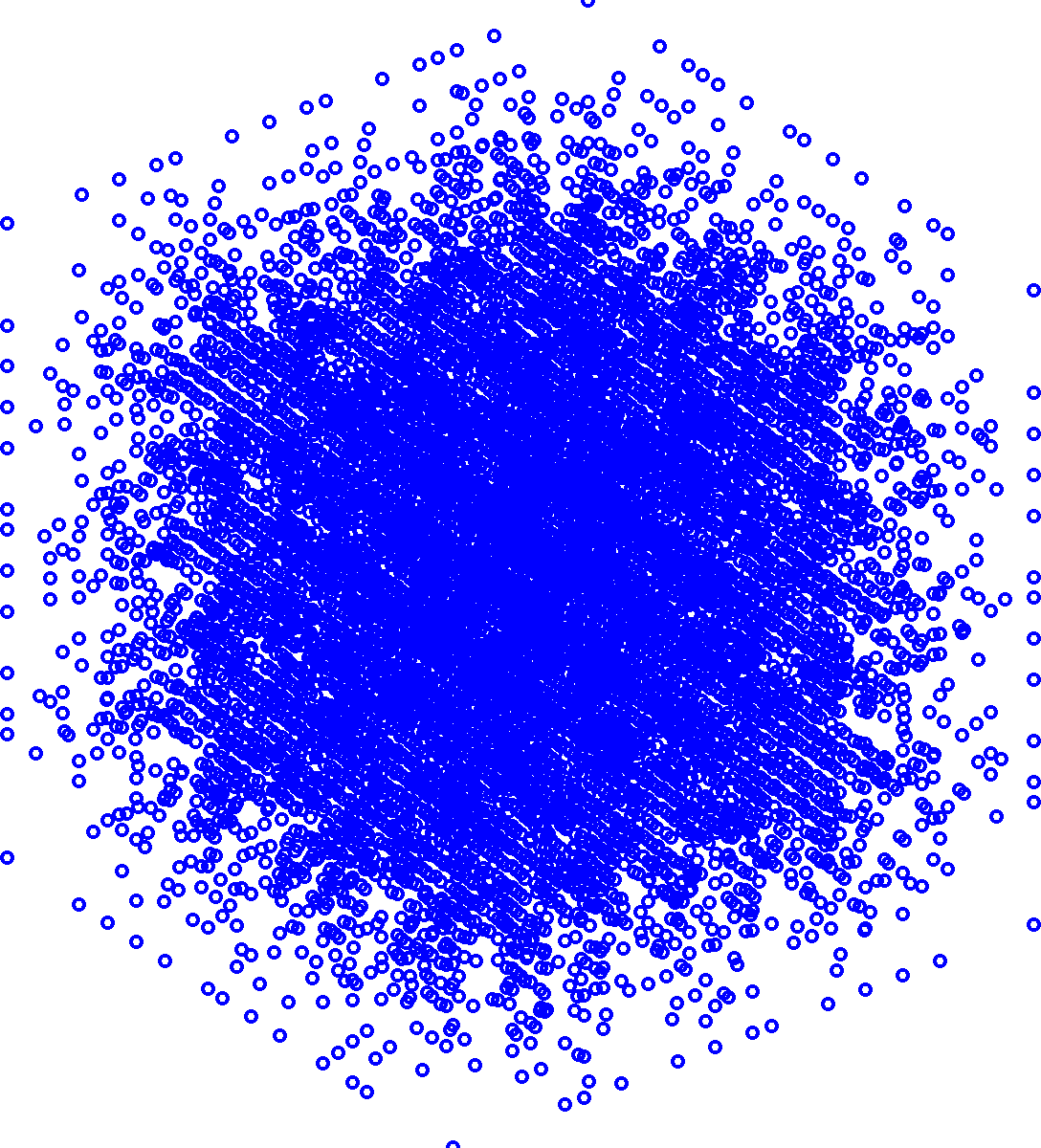}
  \caption*{$\ell = 1$}
 \end{subfigure}
 \\~\\~\\
 \begin{subfigure}{0.2\textwidth}
  \includegraphics[width=\textwidth]{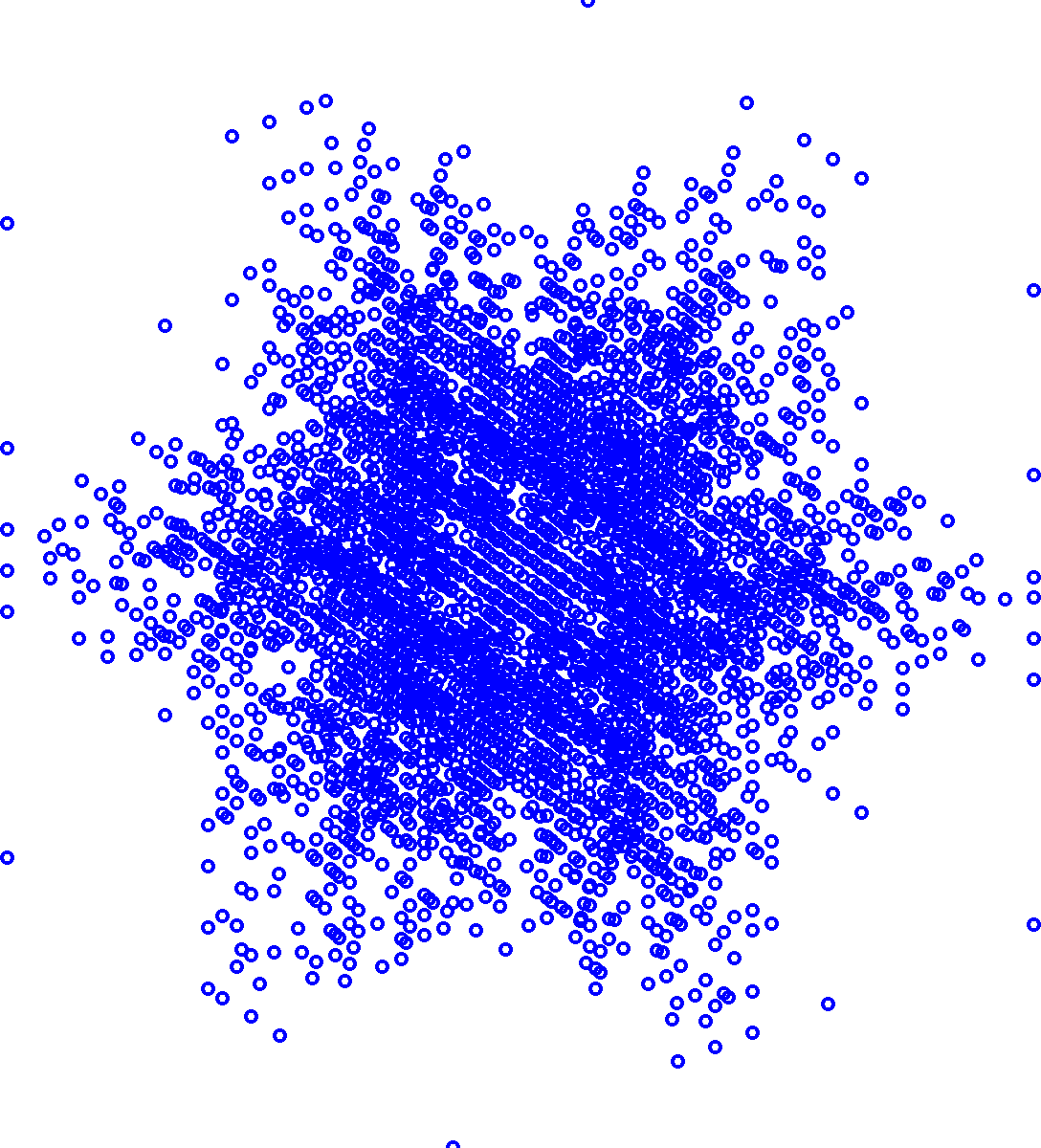}
  \caption*{$\ell = 4/3$}
 \end{subfigure}
 \quad
 \begin{subfigure}{0.2\textwidth}
  \includegraphics[width=\textwidth]{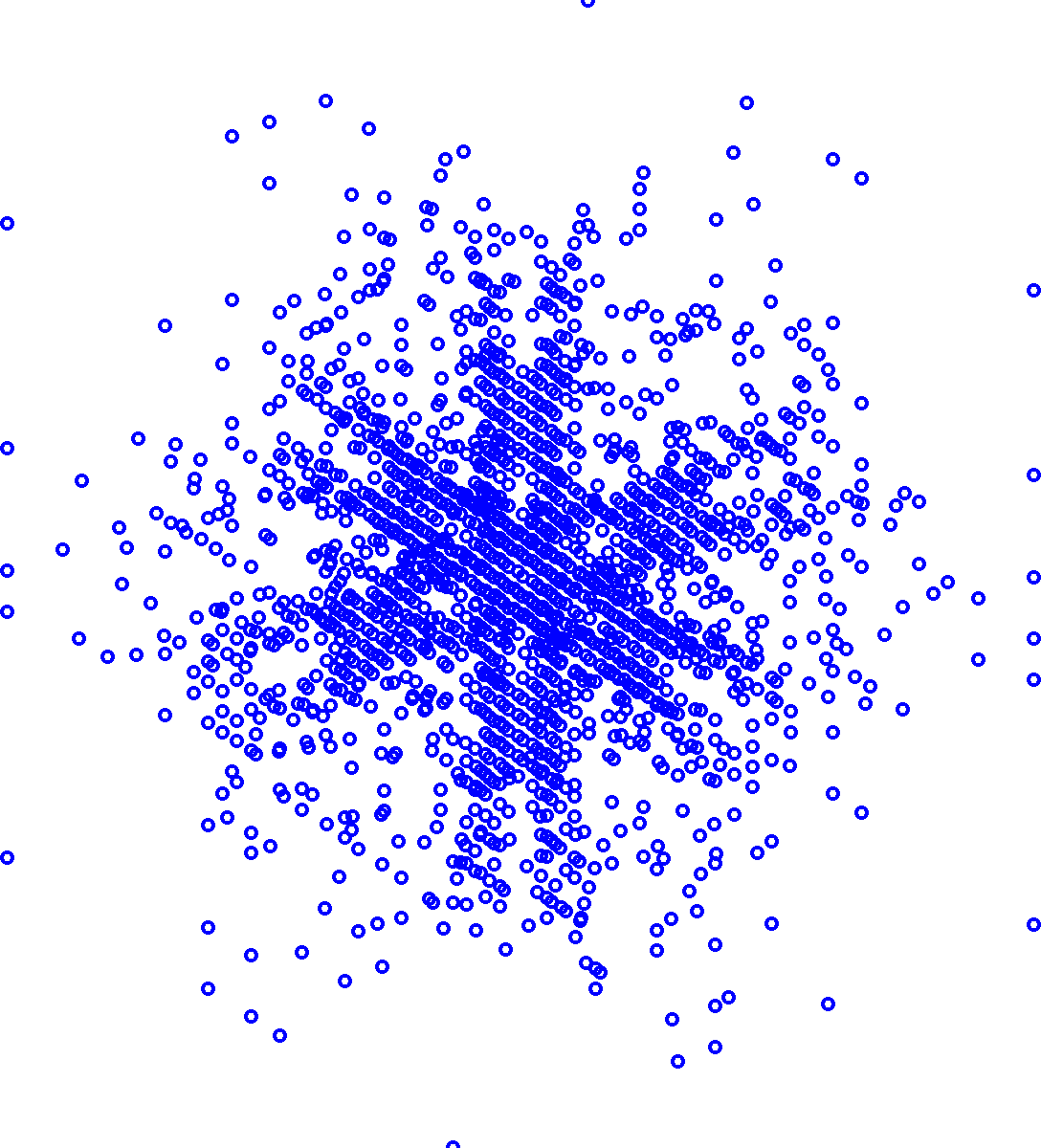}
  \caption*{$\ell = 5/3$}
 \end{subfigure}
 \quad
 \begin{subfigure}{0.2\textwidth}
  \includegraphics[width=\textwidth]{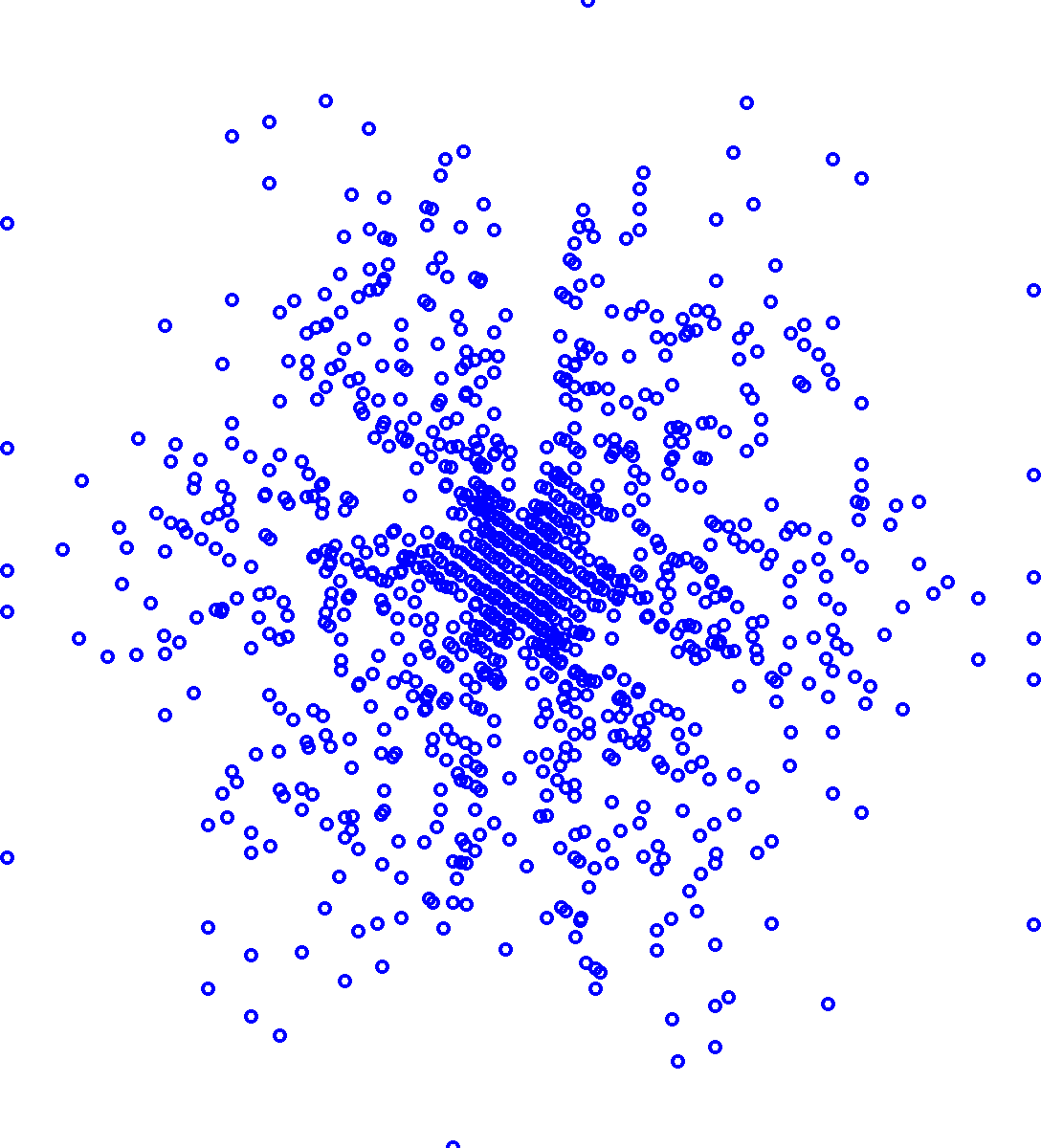}
  \caption*{$\ell = 2$}
 \end{subfigure}
 \caption{Active DOFs at each level $\ell$ of HIF-IE in 3D.}
 \label{fig:hifie3}
\end{figure}

\subsubsection*{Level $\ell$}
Partition $\Omega$ into Voronoi cells about the cell centers $2^{\ell} mh(j_{1} - 1/2, j_{2} - 1/2, j_{3} - 1/2)$ for $1 \leq j_{1}, j_{2}, j_{3} \leq 2^{L - \ell}$. Let $C_{\ell}$ be the collection of index sets corresponding to the active DOFs of each cell. Skeletonization with respect to $C_{\ell}$ then gives
\begin{align*}
 A_{\ell + 1/3} = \skel_{C_{\ell}} (A_{\ell}) \approx U_{\ell}^{*} A_{\ell} V_{\ell}, \qquad U_{\ell} = \prod_{c \in C_{\ell}} Q_{c} R_{\rd{c}}, \quad V_{\ell} = \prod_{c \in C_{\ell}} Q_{c} S_{\rd{c}},
 \end{align*}
where the DOFs $\bigcup_{c \in C_{\ell}} \rd{c}$ have been eliminated.

\subsubsection*{Level $\ell + 1/3$}
Partition $\Omega$ into Voronoi cells about the face centers
\begin{align*}
 &2^{\ell} mh \left( j_{1}, j_{2} - \frac{1}{2}, j_{3} - \frac{1}{2} \right), & &1 \leq j_{1} \leq 2^{L - \ell} - 1, & &1 \leq j_{2}, j_{3} \leq 2^{L - \ell},\\
 &2^{\ell} mh \left( j_{1} - \frac{1}{2}, j_{2}, j_{3} - \frac{1}{2} \right), & &1 \leq j_{2} \leq 2^{L - \ell} - 1, & &1 \leq j_{1}, j_{3} \leq 2^{L - \ell},\\
 &2^{\ell} mh \left( j_{1} - \frac{1}{2}, j_{2} - \frac{1}{2}, j_{3} \right), & &1 \leq j_{3} \leq 2^{L - \ell} - 1, & &1 \leq j_{1}, j_{2} \leq 2^{L - \ell}.\\
\end{align*}
Let $C_{\ell + 1/3}$ be the collection of index sets corresponding to the active DOFs of each cell. Skeletonization with respect to $C_{\ell + 1/3}$ then gives
\begin{multline*}
 A_{\ell + 2/3} = \skel_{C_{\ell + 1/3}} (A_{\ell + 1/3}) \approx U_{\ell + 1/3}^{*} A_{\ell + 1/3} V_{\ell + 1/3},\\
 U_{\ell + 1/3} = \prod_{c \in C_{\ell + 1/3}} Q_{c} R_{\rd{c}}, \quad V_{\ell + 1/3} = \prod_{c \in C_{\ell + 1/3}} Q_{c} S_{\rd{c}},
\end{multline*}
where the DOFs $\bigcup_{c \in C_{\ell + 1/3}} \rd{c}$ have been eliminated.

\subsubsection*{Level $\ell + 2/3$}
Partition $\Omega$ into Voronoi cells about the edge centers
\begin{align*}
 &2^{\ell} mh \left( j_{1}, j_{2}, j_{3} - \frac{1}{2} \right), & &1 \leq j_{1}, j_{2} \leq 2^{L - \ell} - 1, & &1 \leq j_{3} \leq 2^{L - \ell},\\
 &2^{\ell} mh \left( j_{1}, j_{2} - \frac{1}{2}, j_{3} \right), & &1 \leq j_{1}, j_{3} \leq 2^{L - \ell} - 1, & &1 \leq j_{2} \leq 2^{L - \ell},\\
 &2^{\ell} mh \left( j_{1} - \frac{1}{2}, j_{2}, j_{3} \right), & &1 \leq j_{2}, j_{3} \leq 2^{L - \ell} - 1, & &1 \leq j_{1} \leq 2^{L - \ell}.
\end{align*}
Let $C_{\ell + 2/3}$ be the collection of index sets corresponding to the active DOFs of each cell. Skeletonization with respect to $C_{\ell + 2/3}$ then gives
\begin{multline*}
 A_{\ell + 1} = \skel_{C_{\ell + 2/3}} (A_{\ell + 2/3}) \approx U_{\ell + 2/3}^{*} A_{\ell + 2/3} V_{\ell + 2/3},\\
 U_{\ell + 2/3} = \prod_{c \in C_{\ell + 2/3}} Q_{c} R_{\rd{c}}, \quad V_{\ell + 2/3} = \prod_{c \in C_{\ell + 2/3}} Q_{c} S_{\rd{c}},
\end{multline*}
where the DOFs $\bigcup_{c \in C_{\ell + 2/3}} \rd{c}$ have been eliminated.

\subsubsection*{Level $L$}
Combining the approximation over all levels gives
\begin{align*}
 D \equiv A_{L} \approx U_{L - 1/3}^{*} \cdots U_{2/3}^{*} U_{1/3}^{*} U_{0}^{*} A V_{0} V_{1/3} V_{2/3} \cdots V_{L - 1/3},
\end{align*}
so
\begin{subequations}
 \label{eqn:hifie3}
 \begin{align}
  A &\approx U_{0}^{-*} U_{1/3}^{-*} U_{2/3}^{-*} \cdots U_{L - 1/3}^{-*} D V_{L - 1/3}^{-1} \cdots V_{2/3}^{-1} V_{1/3}^{-1} V_{0}^{-1} \equiv F,\\
  A^{-1} &\approx V_{0} V_{1/3} V_{2/3} \cdots V_{L - 1/3} D^{-1} U_{L - 1/3}^{*} \cdots U_{2/3}^{*} U_{1/3}^{*} U_{0}^{*} = F^{-1}.
 \end{align}
\end{subequations}
This procedure is summarized as Algorithm \ref{alg:hifie3}.
\begin{algorithm}
 \caption{HIF-IE in 3D.}
 \label{alg:hifie3}
 \begin{algorithmic}
  \State $A_{0} = A$ \Comment{initialize}
  \For{$\ell = 0, 1, \dots, L - 1$} \Comment{loop from finest to coarsest level}
   \State $A_{\ell + 1/3} = \skel_{C_{\ell}} (A_{\ell}) \approx U_{\ell}^{*} A_{\ell} V_{\ell}$ \Comment{skeletonize cells}
   \State $A_{\ell + 2/3} = \skel_{C_{\ell + 1/3}} (A_{\ell + 1/3}) \approx U_{\ell + 1/3}^{*} A_{\ell + 1/3} V_{\ell + 1/3}$ \Comment{skeletonize faces}
   \State $A_{\ell + 1} = \skel_{C_{\ell + 2/3}} (A_{\ell + 2/3}) \approx U_{\ell + 2/3}^{*} A_{\ell + 2/3} V_{\ell + 2/3}$ \Comment{skeletonize edges}
  \EndFor
  \State $A \approx U_{0}^{-*} U_{1/3}^{-*} \cdots U_{L - 1/3}^{-*} A_{L} V_{L - 1/3}^{-1} \cdots V_{1/3}^{-1} V_{0}^{-1}$ \Comment{generalized LU decomposition}
 \end{algorithmic}
\end{algorithm}

\subsection{Accelerated Compression}
\label{sec:hifie:accel-comp}
Proxy compression still applies, provided that we make some minor modifications to account for SCIs, which we generally have access to only numerically and so cannot evaluate at arbitrary points as needed in Lemma \ref{lem:proxy-gen}. Specifically, for a given index set $c$, we now expand $c^{\nbr}$ by including all DOFs that interact with $c$ via SCIs in addition to those interior to $\Gamma$ as in Section \ref{sec:rskelf:accel-comp}. The far field $c^{\far} = c^{\cmp} \setminus c^{\nbr}$ then consists only of original kernel interactions, so Lemma \ref{lem:proxy-gen} holds. It remains to observe that SCIs are local due to the domain partitioning strategy. Thus, all $c^{\nbr}$ reside in an immediate neighborhood of $c$ and we again conclude that $|c^{\nbr}| = O(|c|)$.

Even with this acceleration, however, the ID still manifests as a computational bottleneck. To combat this, we also tried fast randomized methods \cite{halko:2011:siam-rev} based on compressing $\Phi_{c} Y_{c}$, where $\Phi_{c}$ is a small Gaussian random sampling matrix. We found that the resulting ID was inaccurate when $Y_{c}$ contained SCIs. This could be remedied by considering instead $\Phi_{c} (Y_{c} Y_{c}^{*})^{\gamma} Y_{c}$ for some small integer $\gamma = 1, 2, \dots$, but the expense of the extra multiplications usually outweighed any efficiency gains.

\subsection{Modifications for Second-Kind Integral Equations}
\label{sec:hifie:2k-mod}
The algorithms presented so far are highly accurate for first-kind IEs in that $\| A - F \| / \| A \| = O(\epsilon)$, where $\epsilon$ is the input precision to the ID (Section \ref{sec:results}). For second-kind IEs, however, we see a systematic deterioration of the relative error roughly as $O(N \epsilon)$ as $N \to \infty$. This instability can be explained as follows. Let $A$ be a typical second-kind IE matrix discretization. Then the diagonal entries of $A$ are $O(1)$, while its off-diagonal entries are $O(1/N)$. Since the interpolation matrix, say, $T_{p}$ from the ID has entries of order $O(1)$, the same is true of $B_{\rd{p} \rd{p}}$, $B_{\rd{p} \sk{p}}$, and $B_{\sk{p} \rd{p}}$ in \eqref{eqn:id-sparse}. Therefore, the entries of the Schur complement $B_{\sk{p} \sk{p}}$ in \eqref{eqn:skel} are $O(1)$, i.e., SCIs dominate kernel interactions by a factor of $O(N)$.

\begin{lemma}
 \label{lem:2k-id}
 Assume the setting of the discussion above and let $c \in C_{\ell}$ be such that $Y_{c}$ in \eqref{eqn:proxy-id} contains SCIs. Then $\| Y_{c} \| = O(1)$, so the ID of $Y_{c}$ has absolute error $\| E_{c} \| = O(\epsilon)$.
\end{lemma}

Consider now the process of ``unfolding'' the factorization $F$ from the middle matrix $D \equiv A_{L}$ outward. This is accomplished by undoing the skeletonization operation for each $c \in C_{\ell}$ in reverse order, at each step reconstructing $(A_{\ell})_{:,\rd{c}}$ and $(A_{\ell})_{\rd{c},:}$ from $(A_{\ell + 1/d})_{:,\sk{c}}$ and $(A_{\ell + 1/d})_{\sk{c},:}$. Restricting attention to 2D for concreteness, we start at level $L$ with interactions between the DOFs $s_{L}$ as depicted in Figure \ref{fig:recon-err} (left).
\begin{figure}
 \includegraphics{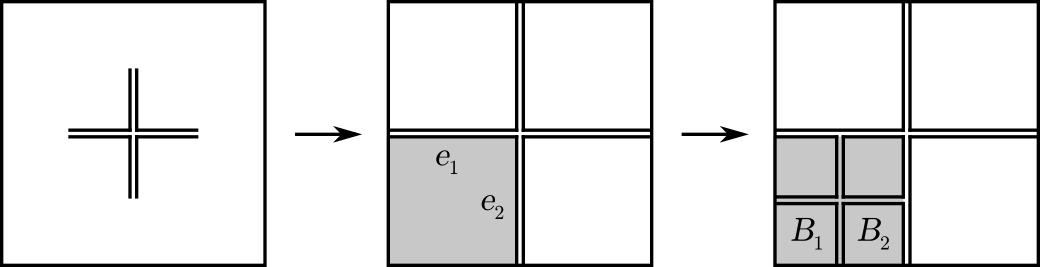}
 \caption{Matrix reconstruction from skeleton-skeleton interactions.}
 \label{fig:recon-err}
\end{figure}
By Lemma \ref{lem:2k-id}, un-skeletonizing each edge $c \in C_{L - 1/2}$ induces an error in the interactions between the edges $e_{1}$ and $e_{2}$ as labeled in the figure (center) of absolute magnitude $O(\epsilon)$. At the next level, un-skeletonizing the shaded cell $c \in C_{L - 1}$ which they bound then relies on the approximate interactions between $e_{1}$ and $e_{2}$. This spreads the $O(\epsilon)$ error over the reconstructed cell interactions, which is small for SCIs acting internally to each cell $c \in C_{L - 2}$ (omitting level $L - 3/2$ for simplicity) but not for kernel interactions between any two distinct cells $B_{1}$ and $B_{2}$ (right); indeed, the relative error for the latter is $O(N \epsilon)$. These corrupted interactions are then used for reconstruction at the next level and are eventually spread throughout the whole matrix. The same argument clearly holds in 3D.

This analysis suggests that the only fix is to skeletonize at effective precision $O(\epsilon / N)$ so that kernel interactions are accurately reconstructed. This is equivalent to ensuring that both scales in $Y_{c}$ are well approximated by the ID. Following this intuition, we decompose $Y_{c}$ as $Y_{c} = Y_{c}^{\krn} + Y_{c}^{\sch}$, where $Y_{c}^{\krn}$ consists purely of kernel interactions, and set $\rho_{c} \epsilon$ for $\rho_{c} = \min (1, \| Y_{c}^{\krn} \| / \| Y_{c}^{\sch} \|)$ as the local compression tolerance, which we note uses increased precision only when necessary.

The two-scale structure of $Y_{c}$ also enables an additional optimization as can be seen by studying the sparsity patterns of SCIs. Figure \ref{fig:split-comp} shows an example configuration in 2D after cell skeletonization at level $\ell$, which leaves a collection of edges at level $\ell + 1/2$, each composed of two half-edges consisting of skeletons from the two cells on either side (left).
\begin{figure}
 \includegraphics{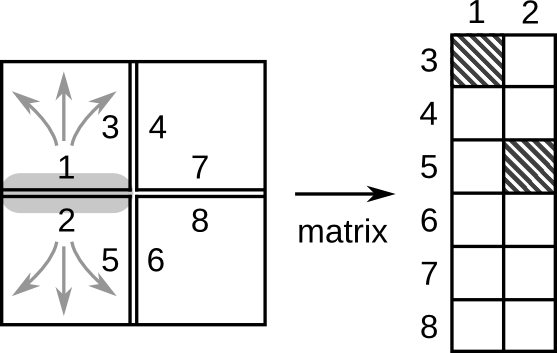}
 \caption{Sparsity pattern of SCIs. A reference domain configuration (left) is shown with each half-edge labeled from $1$--$8$. The edge of interest ($1$ and $2$) is outlined in gray along with all outgoing SCIs. The corresponding matrix view (right) shows these interactions (hatched) indexed by half-edge.}
 \label{fig:split-comp}
\end{figure}
Let $c = g_{1} \cup g_{2} \in C_{\ell + 1/2}$ be a given edge with indices partitioned by half-edge, and let $Y_{g_{j}}$ be the submatrix of $Y_{c}$ corresponding to $g_{j}$. Then $Y_{g_{1}}^{\sch}$ and $Y_{g_{2}}^{\sch}$ (analogously defined) have different nonzero structures, so $Y_{g_{1}}$ and $Y_{g_{2}}$ have large entries in different row blocks (right). The stable interpolation of $Y_{c}$ hence requires that all interpolation coefficients from one half-edge to the other be $O(1/N)$ since otherwise the reconstruction of, say, $Y_{g_{1}}$ will have large errors in rows where $Y_{g_{2}}^{\sch}$ is nonzero. As $N \to \infty$, these cross-interpolation coefficients must therefore vanish and the compression of $Y_{c}$ decouples into the compression of $Y_{g_{1}}$ and $Y_{g_{2}}$ separately. We enforce this asymptotic decoupling explicitly, which moreover provides an acceleration due to the cubic cost of the ID. The ID of $Y_{c}$ is then given by $\sk{c} = (\sk{g}_{1}, \sk{g}_{2})$, $\rd{c} = (\rd{g}_{1}, \rd{g}_{2})$, and $T_{c} = \diag (T_{g_{1}}, T_{g_{2}})$, where $g_{j} = \sk{g}_{j} \cup \rd{g}_{j}$ and $T_{g_{j}}$ define the ID of $Y_{g_{j}}$. We use the compression tolerance $\rho_{g_{j}} \epsilon$ with $\rho_{g_{j}} = \min (1, \| Y_{g_{j}}^{\krn} \| / \| Y_{g_{j}}^{\sch} \|)$ locally for each $g_{j}$.

In general, we define the subsets $\{ g_{j} \}$ algebraically according to the sparsity pattern of $Y_{c}^{\sch}$, which can be done using the matrix indicator function
\begin{align*}
 (\spy (A))_{ij} =
 \begin{cases}
  0, & A_{ij} = 0\\
  1, & A_{ij} \neq 0.
 \end{cases}
\end{align*}

\begin{lemma}
 Let $B = \spy (A)^{*} \spy (A)$ for some matrix $A$. Then $A_{:,i}$ and $A_{:,j}$ have the same sparsity pattern if and only $B_{ij} = \max (\| A_{:,i} \|_{0}, \| A_{:,j} \|_{0})$.
\end{lemma}

\subsection{Complexity Estimates}
Analysis of HIF-IE is impeded by the compression of SCIs, for which we do not have rigorous bounds on the interaction rank. Nonetheless, ample numerical evidence suggests that SCIs behave very similarly to standard kernel interactions. For the sake of analysis, we hence assume that the same rank estimates apply, from which we have \eqref{eqn:inter-rank-1d} for all $\ell \geq 1$ by reduction to 1D. We emphasize that this has yet to be proven, so all following results should formally be understood as conjectures, albeit ones with strong experimental support (Section \ref{sec:results}).

\begin{theorem}
 \label{thm:hifie}
 Assume that \eqref{eqn:inter-rank-1d} holds. Then the cost of constructing the factorization $F$ in \eqref{eqn:hifie2} or \eqref{eqn:hifie3} using HIF-IE with accelerated compression is $t_{f} = O(N)$, while that of applying $F$ or $F^{-1}$ is $t_{a/s} = O(N)$.
\end{theorem}

\begin{proof}
 This is essentially just a restatement of Corollary \ref{cor:complexity-1d} (but with the sum now taken also over fractional levels).
\end{proof}

\begin{corollary}
 \label{cor:hifie-2k}
 For second-kind IEs,
 \begin{align*}
  t_{f} =
  \begin{cases}
   O(N \log N), & d = 2\\
   O(N \log^{6} N), & d = 3,
  \end{cases} \qquad t_{a/s} =
  \begin{cases}
   O(N \log \log N), & d = 2\\
   O(N \log^{2} N), & d = 3.
  \end{cases}
 \end{align*}
\end{corollary}

\begin{proof}
 According to the modifications of Section \ref{sec:hifie:2k-mod}, there are now two effective ID tolerances: $\epsilon$ for all $c \in C_{\ell}$ such that $Y_{c}^{\sch} = 0$ and $O(\epsilon / N)$ otherwise. The former is used for all initial levels $\ell \leq \lambda$ before SCIs have become widespread (i.e., before any meaningful dimensional reduction has occurred), and the latter for all $\ell > \lambda$. But using precision $O(\epsilon / N)$ yields a rank estimate with constant of proportionality $O(\log^{\delta} N)$, where $\delta$ is the intrinsic dimension of the DOF cluster $c$ \cite{greengard:1987:j-comput-phys,greengard:1997:acta-numer}, so the amount of compression depends on $N$. Thus, $\lambda = \lambda (N)$ and our first task is to determine its form.

 The crossover level $\lambda$ can be obtained by balancing the typical size $|c|$ of an edge (2D and 3D) or face (3D only) with its skeleton size $|\sk{c}|$. In 2D, this is $2^{\lambda} \sim \lambda \log N$, where the left-hand side gives the size of an edge at level $\lambda$, and the right-hand side the estimated rank for SCI compression. Therefore, $\lambda \sim \log \log N$.

 In 3D, there are two crossover levels $\lambda_{1}$ and $\lambda_{2}$ corresponding to face and edge compression, respectively, with $\lambda = \max (\lambda_{1}, \lambda_{2})$:
 \begin{align*}
  2^{2 \lambda_{1}} \sim 2^{\lambda_{1}} \log^{2} N, \quad 2^{\lambda_{2}} \sim \lambda_{2} \log N.
 \end{align*}
 Hence, $\lambda_{1} \sim 2 \log \log N$ and $\lambda_{2} \sim \log \log N$, so $\lambda \sim 2 \log \log N$.

 The cost of constructing $F$ for second-kind IEs is then
 \begin{align*}
  t_{f} = O(2^{dL} m^{3d}) + \sum_{\ell = 0}^{\lambda} 2^{d(L - \ell)} O(2^{3(d - 1) \ell}) + \sumprime_{\ell = \lambda}^{L} O(2^{d(L - \ell)} k_{\ell}^{3}),
 \end{align*}
 where prime notation denotes summation over all levels, both integer and fractional, and $k_{\ell}$ is as given in \eqref{eqn:inter-rank-1d} with $k = O(\log N)$. The first sum corresponds to running RSF on the initial levels and reduces to
 \begin{align*}
  \sum_{\ell = 0}^{\lambda} 2^{d(L - \ell)} O(2^{3(d - 1) \ell}) =
  \begin{cases}
   O(N \log N), & d = 2\\
   O(N \log^{6} N), & d = 3,
  \end{cases}
 \end{align*}
 while the second can be interpreted as the cost of the standard HIF-IE (without modification) applied to the remaining
 \begin{align*}
  O(2^{-\lambda} N) =
  \begin{cases}
   O(N / \log N), & d = 2\\
   O(N / \log^{2} N), & d = 3
  \end{cases}
 \end{align*}
 DOFs at uniform precision $O(\epsilon / N)$. By Corollary \ref{cor:complexity-1d}, this is
 \begin{align*}
  \sumprime_{\ell = \lambda}^{L} O(2^{d(L - \ell)} k_{\ell}^{3}) =
  \begin{cases}
   O(N \log N), & d = 2\\
   O(N), & d = 3,
  \end{cases}
 \end{align*}
 so, adding all terms, we derive $t_{f}$ as claimed.

 A similar argument for
 \begin{align*}
  t_{a/s} = O(2^{dL} m^{2d}) + \sum_{\ell = 0}^{\lambda} 2^{d(L - \ell)} O(2^{2(d - 1) \ell}) + \sumprime_{\ell = \lambda}^{L} O(2^{d(L - \ell)} k_{\ell}^{2})
 \end{align*}
 completes the proof.
\end{proof}

\begin{remark}
 Like Theorem \ref{thm:rskelf} for RSF, the parameter $d$ in Corollary \ref{cor:hifie-2k} can also be regarded as the intrinsic dimension.
\end{remark}

\section{Numerical Results}
\label{sec:results}
In this section, we demonstrate the efficiency of HIF-IE by reporting numerical results for some benchmark problems in 2D and 3D. All algorithms and examples were implemented in MATLAB and are freely available at \url{https://github.com/klho/FLAM/}. In what follows, we refer to RSF as \alg{rskelf2} in 2D and \alg{rskelf3} in 3D. Similarly, we call HIF-IE \alg{hifie2} and \alg{hifie3}, respectively, with \alg{hifie2x} and \alg{hifie3x} denoting their second-kind IE counterparts. All codes are fully adaptive and built on quadtrees in 2D and octrees in 3D. The average block size $|c|$ at level $0$ (and hence the tree depth $L$) was chosen so that $|c| \sim 2|\sk{c}|$. In select cases, the first few fractional levels of HIF-IE were skipped to optimize the running time. Symmetry was exploited wherever possible by compressing
\begin{align*}
 Y_{c}' =
 \begin{bmatrix}
  A_{c^{\nbr},c}\\
  Y_{c^{\eqv},c}
 \end{bmatrix}
\end{align*}
instead of the full matrix $Y_{c}$ in \eqref{eqn:proxy-id}, which reduces the cost by about a factor of $2$. Diagonal blocks, i.e., $A_{pp}$ in Lemma \ref{lem:sparse-elim}, were factored using the (partially pivoted) LDL decomposition if $A$ is symmetric and the LU decomposition otherwise.

For each example, the following, if applicable, are given:
\begin{itemize}
 \item
  $\epsilon$: base relative precision of the ID;
 \item
  $N$: total number of DOFs in the problem;
 \item
  $|s_{L}|$: number of active DOFs remaining at the highest level;
 \item
  $t_{f}$: wall clock time for constructing the factorization $F$ in seconds;
 \item
  $m_{f}$: memory required to store $F$ in GB;
 \item
  $t_{a/s}$: wall clock time for applying $F$ or $F^{-1}$ in seconds;
 \item
  $e_{a}$: {\it a posteriori} estimate of $\| A - F \| / \| A \|$ (see below);
 \item
  $e_{s}$: {\it a posteriori} estimate of $\| I - A F^{-1} \| \geq \| A^{-1} - F^{-1} \| / \| A^{-1} \|$;
 \item
  $n_{i}$: number of iterations to solve \eqref{eqn:linear-system} using GMRES with preconditioner $F^{-1}$ to a tolerance of $10^{-12}$, where $f$ is a standard uniform random vector (ill-conditioned systems only).
\end{itemize}
The operator errors $e_{a}$ and $e_{s}$ were estimated using power iteration with a standard uniform random start vector \cite{dixon:1983:siam-j-numer-anal,kuczynski:1992:siam-j-matrix-anal-appl} and a convergence criterion of $10^{-2}$ relative precision in the matrix norm. This procedure requires the application of both $A$ and $A^{*}$, which for translation-invariant kernels was done using fast Fourier convolution \cite{brigham:1988:prentice-hall} and for non--translation-invariant kernels using an ID-based kernel-independent FMM \cite{martinsson:2007:siam-j-sci-comput,pan:2012:radio-sci} at precision $10^{-15}$. The same methods were also used to apply $A$ when solving \eqref{eqn:linear-system} iteratively.

For simplicity, all IEs were discretized using a piecewise constant collocation method as in Section \ref{sec:rskelf}. Certain near-field interactions (to be defined for each case) were computed using adaptive quadrature, while all other interactions were handled as simple one-point approximations, e.g., $K_{ij} = K(\| x_{i} - x_{j} \|) h^{d}$ in \eqref{eqn:kernel-integral}.

All computations were performed in MATLAB R2010b on a single core (without parallelization) of an Intel Xeon E7-4820 CPU at 2.0 GHz on a 64-bit Linux server with 256 GB of RAM.

\subsection{Two Dimensions}
We begin first in 2D, where we present three examples.

\subsubsection*{Example 1}
Consider \eqref{eqn:ie} with $a(x) \equiv 0$, $b(x) \equiv c(x) \equiv 1$, $K(r) = -1 / (2 \pi) \log r$, and $\Omega = (0, 1)^{2}$, i.e., a first-kind volume IE in the unit square, discretized over a uniform $n \times n$ grid. The diagonal entries $K_{ii}$ are computed adaptively, while all $K_{ij}$ for $i \neq j$ are approximated using one-point quadratures. We factored the resulting matrix $A$ using both \alg{rskelf2} and \alg{hifie2} at $\epsilon = 10^{-3}$, $10^{-6}$, and $10^{-9}$. The data are summarized in Tables \ref{tab:ie_square1-f} and \ref{tab:ie_square1-a} with scaling results shown in Figure \ref{fig:ie_square1}.

\begin{table}
 \caption{Factorization results for Example 1.}
 \label{tab:ie_square1-f}
 \scriptsize
 \begin{tabular}{cr|rcc|rcc}
  \toprule
  & & \multicolumn{3}{|c|}{\alg{rskelf2}} & \multicolumn{3}{|c}{\alg{hifie2}}\\
  $\epsilon$ & \multicolumn{1}{c|}{$N$} & \multicolumn{1}{|c}{$|s_{L}|$} & $t_{f}$ & $m_{f}$ & \multicolumn{1}{|c}{$|s_{L}|$} & $t_{f}$ & $m_{f}$\\
  \midrule
  \multirow{3}{*}{$10^{-3}$} &  $512^{2}$ &  $2058$ & $1.9$e$+2$ & $7.7$e$-1$ &  $67$ & $6.2$e$+1$ & $3.0$e$-1$\\
                             & $1024^{2}$ &  $4106$ & $1.4$e$+3$ & $3.6$e$+0$ &  $67$ & $2.5$e$+2$ & $1.2$e$+0$\\
                             & $2048^{2}$ &  $6270$ & $6.6$e$+3$ & $1.4$e$+1$ &  $70$ & $1.0$e$+3$ & $4.7$e$+0$\\
  \midrule
  \multirow{3}{*}{$10^{-6}$} &  $512^{2}$ &  $3430$ & $7.7$e$+2$ & $1.8$e$+0$ & $373$ & $2.7$e$+2$ & $8.5$e$-1$\\
                             & $1024^{2}$ &  $5857$ & $4.6$e$+3$ & $7.7$e$+0$ & $428$ & $1.2$e$+3$ & $3.5$e$+0$\\
                             & $2048^{2}$ & $11317$ & $3.0$e$+4$ & $3.3$e$+1$ & $455$ & $4.8$e$+3$ & $1.4$e$+1$\\
  \midrule
  \multirow{3}{*}{$10^{-9}$} &  $512^{2}$ &  $4162$ & $1.2$e$+3$ & $2.3$e$+0$ & $564$ & $4.3$e$+2$ & $1.2$e$+0$\\
                             & $1024^{2}$ &  $8264$ & $1.0$e$+4$ & $1.1$e$+1$ & $686$ & $2.1$e$+3$ & $4.8$e$+0$\\
                             & $2048^{2}$ & $16462$ & $8.3$e$+4$ & $5.2$e$+1$ & $837$ & $9.1$e$+3$ & $1.9$e$+1$\\
  \bottomrule
 \end{tabular}
\end{table}

\begin{table}
 \caption{Matrix application results for Example 1.}
 \label{tab:ie_square1-a}
 \scriptsize
 \begin{tabular}{cr|c|cccr}
  \toprule
  & & \multicolumn{1}{|c|}{\alg{rskelf2}} & \multicolumn{4}{|c}{\alg{hifie2}}\\
  $\epsilon$ & \multicolumn{1}{c|}{$N$} & $t_{a/s}$ & $t_{a/s}$ & $e_{a}$ & $e_{s}$ & \multicolumn{1}{c}{$n_{i}$}\\
  \midrule
  \multirow{3}{*}{$10^{-3}$} &  $512^{2}$ & $7.2$e$-1$ & $5.2$e$-1$ & $3.4$e$-04$ & $1.2$e$-1$ &  $9$\\
                             & $1024^{2}$ & $3.2$e$+0$ & $2.1$e$+0$ & $3.8$e$-04$ & $1.6$e$-1$ & $10$\\
                             & $2048^{2}$ & $1.3$e$+1$ & $1.2$e$+1$ & $4.3$e$-04$ & $1.6$e$-1$ & $10$\\
  \midrule
  \multirow{3}{*}{$10^{-6}$} &  $512^{2}$ & $9.2$e$-1$ & $9.7$e$-1$ & $3.8$e$-07$ & $5.0$e$-4$ &  $3$\\
                             & $1024^{2}$ & $4.2$e$+0$ & $4.1$e$+0$ & $3.3$e$-07$ & $6.5$e$-4$ &  $4$\\
                             & $2048^{2}$ & $2.1$e$+1$ & $1.5$e$+1$ & $5.0$e$-07$ & $4.1$e$-4$ &  $4$\\
  \midrule
  \multirow{3}{*}{$10^{-9}$} &  $512^{2}$ & $1.1$e$+0$ & $8.1$e$-1$ & $2.8$e$-10$ & $4.3$e$-7$ &  $2$\\
                             & $1024^{2}$ & $4.9$e$+0$ & $3.5$e$+0$ & $2.7$e$-10$ & $6.8$e$-7$ &  $2$\\
                             & $2048^{2}$ & $2.8$e$+1$ & $1.4$e$+1$ & $5.7$e$-10$ & $1.1$e$-6$ &  $2$\\
  \bottomrule
 \end{tabular}
\end{table}

\begin{figure}
 \includegraphics{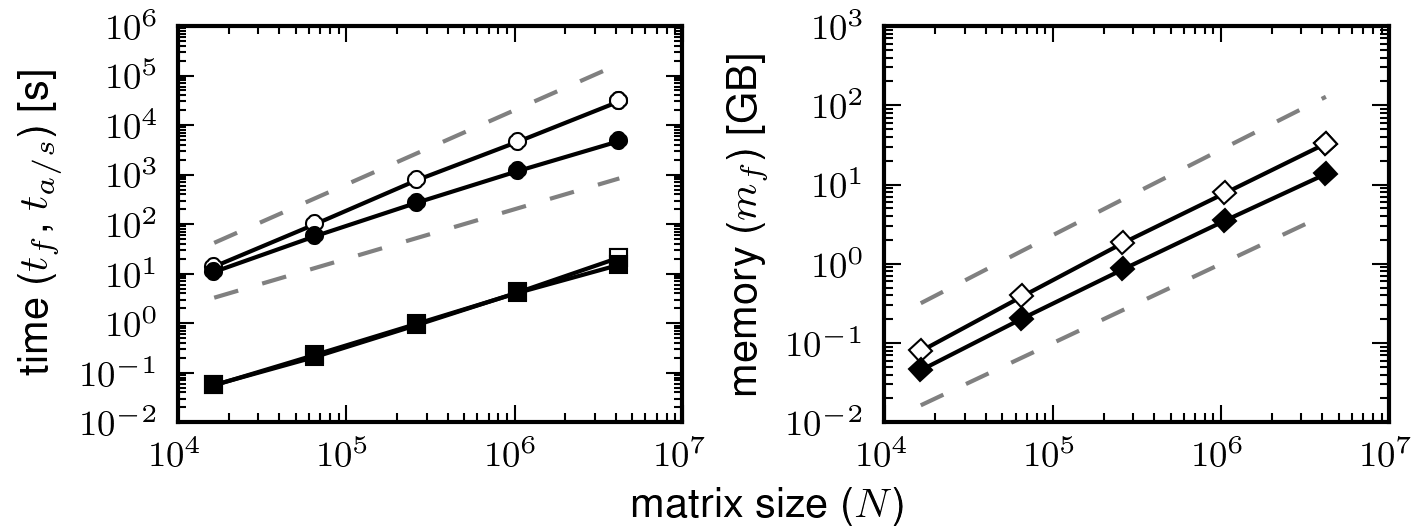}
 \caption{Scaling results for Example 1. Wall clock times $t_{f}$ ($\circ$) and $t_{a/s}$ ($\Box$) and storage requirements $m_{f}$ ($\diamond$) are shown for \alg{rskelf2} (white) and \alg{hifie2} (black) at precision $\epsilon = 10^{-6}$. Included also are reference scalings (gray dashed lines) of $O(N)$ and $O(N^{3/2})$ (left, from bottom to top), and $O(N)$ and $O(N \log N)$ (right). The lines for $t_{a/s}$ (bottom left) lie nearly on top of each other.}
\label{fig:ie_square1}
\end{figure}

It is evident that $|s_{L}| \sim k_{L}$ behaves as predicted, with HIF-IE achieving significant compression over RSF. Consequently, we find strong support for asymptotic complexities consistent with Theorems \ref{thm:rskelf} and \ref{thm:hifie}. For all problem sizes tested, $t_{f}$ and $m_{f}$ are always smaller for HIF-IE, though $t_{a/s}$ is quite comparable. This is because $t_{a/s}$ is dominated by memory access (at least in our current implementation), which also explains its relative insensitivity to $\epsilon$. Furthermore, we observe that $t_{a/s} \ll t_{f}$ for both methods, which makes them ideally suited to systems involving multiple right-hand sides.

The forward approximation error $e_{a} = O(\epsilon)$ for all $N$ and seems to increase only very mildly, if at all, with $N$. This indicates that the local accuracy of the ID provides a good estimate of the overall accuracy of the algorithm, which is not easy to prove since the multilevel matrix factors constituting $F$ are not unitary. On the other hand, we expect the inverse approximation error to scale as $e_{s} = O(\kappa (A) e_{a})$, where $\kappa (A) = \| A \| \| A^{-1} \|$ is the condition number of $A$, and indeed we see that $e_{s}$ is much larger due to the ill-conditioning of the first-kind system. When using $F^{-1}$ to precondition GMRES, however, the number of iterations required is always very small. This indicates that $F^{-1}$ is a highly effective preconditioner.

\subsubsection*{Example 2}
Consider now the same setup as in Example 1 but with $a(x) \equiv 1$. This gives a well-conditioned second-kind IE, which we factored using \alg{rskelf2}, \alg{hifie2}, and \alg{hifie2x}. The data are summarized in Tables \ref{tab:ie_square2-f} and \ref{tab:ie_square2-a} with scaling results in Figure \ref{fig:ie_square2}.

\begin{table}
 \caption{Factorization results for Example 2.}
 \label{tab:ie_square2-f}
 \scriptsize
 \begin{tabular}{cr|rcc|rcc|rcc}
  \toprule
  & & \multicolumn{3}{|c|}{\alg{rskelf2}} & \multicolumn{3}{|c}{\alg{hifie2}} & \multicolumn{3}{|c}{\alg{hifie2x}}\\
  $\epsilon$ & \multicolumn{1}{c|}{$N$} & \multicolumn{1}{|c}{$|s_{L}|$} & $t_{f}$ & $m_{f}$ & \multicolumn{1}{|c}{$|s_{L}|$} & $t_{f}$ & $m_{f}$ & \multicolumn{1}{|c}{$|s_{L}|$} & $t_{f}$ & $m_{f}$\\
  \midrule
  \multirow{3}{*}{$10^{-3}$} &  $512^{2}$ &  $2058$ & $1.9$e$+2$ & $7.7$e$-1$ &  $108$ & $6.8$e$+1$ & $3.5$e$-1$ &  $376$ & $1.1$e$+2$ & $5.1$e$-1$\\
                             & $1024^{2}$ &  $4106$ & $1.4$e$+3$ & $3.6$e$+0$ &  $135$ & $2.8$e$+2$ & $1.4$e$+0$ &  $456$ & $5.3$e$+2$ & $2.2$e$+0$\\
                             & $2048^{2}$ &  $6270$ & $6.6$e$+3$ & $1.4$e$+1$ &  $172$ & $1.2$e$+3$ & $5.7$e$+0$ &  $522$ & $2.4$e$+3$ & $9.4$e$+0$\\
  \midrule
  \multirow{3}{*}{$10^{-6}$} &  $512^{2}$ &  $3430$ & $7.7$e$+2$ & $1.8$e$+0$ &  $475$ & $2.2$e$+2$ & $8.8$e$-1$ &  $804$ & $4.7$e$+2$ & $1.4$e$+0$\\
                             & $1024^{2}$ &  $5857$ & $4.7$e$+3$ & $7.7$e$+0$ &  $580$ & $9.1$e$+2$ & $3.4$e$+0$ &  $962$ & $2.2$e$+3$ & $5.7$e$+0$\\
                             & $2048^{2}$ & $11317$ & $3.0$e$+4$ & $3.3$e$+1$ &  $614$ & $3.6$e$+3$ & $1.4$e$+1$ & $1115$ & $9.6$e$+3$ & $2.3$e$+1$\\
  \midrule
  \multirow{3}{*}{$10^{-9}$} &  $512^{2}$ &  $4162$ & $1.2$e$+3$ & $2.3$e$+0$ & $1030$ & $6.4$e$+2$ & $1.5$e$+0$ & $1087$ & $6.7$e$+2$ & $1.7$e$+0$\\
                             & $1024^{2}$ &  $8264$ & $1.0$e$+4$ & $1.1$e$+1$ & $1241$ & $3.2$e$+3$ & $6.3$e$+0$ & $1381$ & $3.6$e$+3$ & $7.2$e$+0$\\
                             & $2048^{2}$ & $16462$ & $8.2$e$+4$ & $5.2$e$+1$ & $1583$ & $1.5$e$+4$ & $2.6$e$+1$ & $1697$ & $1.8$e$+4$ & $3.1$e$+1$\\
  \bottomrule
 \end{tabular}
\end{table}

\begin{table}
 \caption{Matrix application results for Example 2.}
 \label{tab:ie_square2-a}
 \scriptsize
 \begin{tabular}{cr|c|ccc|ccc}
  \toprule
  & & \multicolumn{1}{|c|}{\alg{rskelf2}} & \multicolumn{3}{|c}{\alg{hifie2}} & \multicolumn{3}{|c}{\alg{hifie2x}}\\
  $\epsilon$ & \multicolumn{1}{c|}{$N$} & $t_{a/s}$ & $t_{a/s}$ & $e_{a}$ & $e_{s}$ & $t_{a/s}$ & $e_{a}$ & $e_{s}$\\
  \midrule
  \multirow{3}{*}{$10^{-3}$} &  $512^{2}$ & $7.2$e$-1$ & $5.4$e$-1$ & $7.8$e$-2$ & $8.5$e$-2$ & $5.3$e$-1$ & $2.6$e$-04$ & $2.9$e$-4$\\
                             & $1024^{2}$ & $3.3$e$+0$ & $2.3$e$+0$ & $8.3$e$-2$ & $9.1$e$-2$ & $2.4$e$+0$ & $2.7$e$-04$ & $3.0$e$-4$\\
                             & $2048^{2}$ & $1.1$e$+1$ & $1.2$e$+1$ & $9.8$e$-2$ & $1.1$e$-1$ & $1.2$e$+1$ & $8.0$e$-04$ & $8.7$e$-4$\\
  \midrule
  \multirow{3}{*}{$10^{-6}$} &  $512^{2}$ & $1.2$e$+0$ & $9.6$e$-1$ & $4.1$e$-4$ & $4.4$e$-4$ & $1.0$e$+0$ & $5.9$e$-07$ & $6.7$e$-7$\\
                             & $1024^{2}$ & $5.1$e$+0$ & $3.3$e$+0$ & $8.2$e$-4$ & $9.0$e$-4$ & $4.5$e$+0$ & $9.3$e$-07$ & $1.0$e$-6$\\
                             & $2048^{2}$ & $1.8$e$+1$ & $1.2$e$+1$ & $3.7$e$-3$ & $4.1$e$-3$ & $1.7$e$+1$ & $1.6$e$-06$ & $1.8$e$-6$\\
  \midrule
  \multirow{3}{*}{$10^{-9}$} &  $512^{2}$ & $1.4$e$+0$ & $8.9$e$-1$ & $3.0$e$-7$ & $3.4$e$-7$ & $1.2$e$+0$ & $2.8$e$-10$ & $3.2$e$-10$\\
                             & $1024^{2}$ & $5.4$e$+0$ & $3.7$e$+0$ & $8.4$e$-7$ & $9.6$e$-7$ & $5.0$e$+0$ & $3.5$e$-10$ & $3.9$e$-10$\\
                             & $2048^{2}$ & $2.5$e$+1$ & $1.5$e$+1$ & $1.8$e$-6$ & $2.0$e$-6$ & $1.8$e$+1$ & $1.1$e$-09$ & $1.2$e$-09$\\
  \bottomrule
 \end{tabular}
\end{table}

\begin{figure}
 \includegraphics{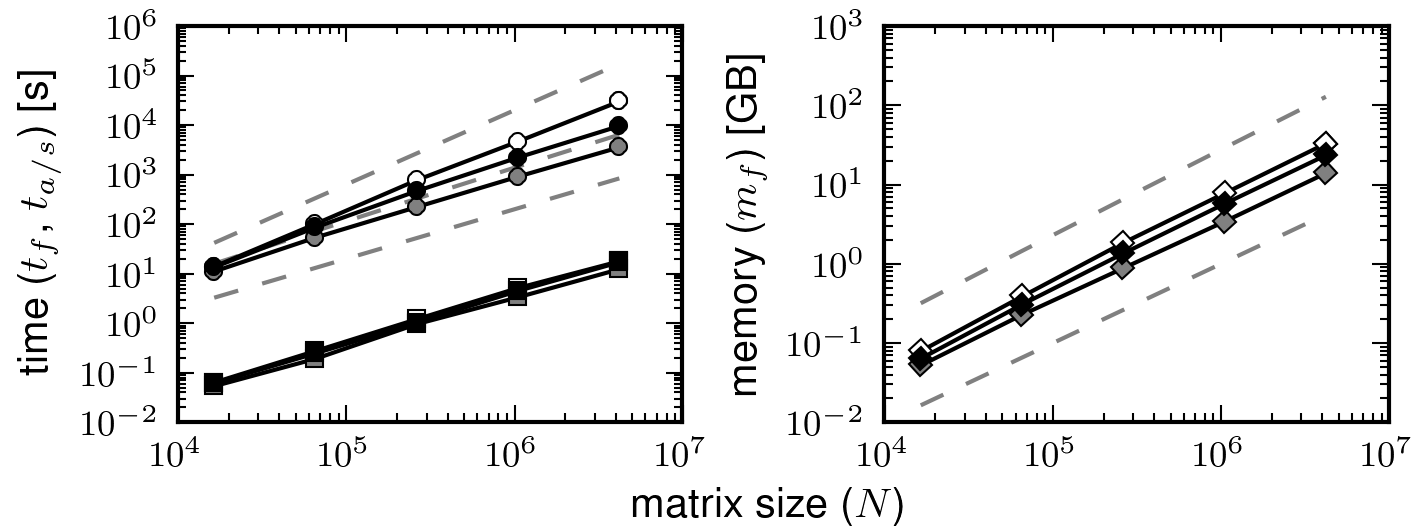}
 \caption{Scaling results for Example 2, comparing \alg{rskelf2} (white), \alg{hifie2} (gray), and \alg{hifie2x} (black) at precision $\epsilon = 10^{-6}$. Included also are reference scalings of $O(N)$, $O(N \log N)$, and $O(N^{3/2})$ (left); and $O(N)$ and $O(N \log N)$ (right). All other notation as in Figure \ref{fig:ie_square1}.}
 \label{fig:ie_square2}
\end{figure}

As expected, results for \alg{rskelf2} are essentially the same as those in Example 1 since the off-diagonal interactions at each level are identical. We also see the breakdown of \alg{hifie2}, which still has linear complexity but fails to properly approximate $A$ as predicted in Section \ref{sec:hifie:2k-mod}. This is remedied by \alg{hifie2x}, which achieves $e_{a}, e_{s} = O(\epsilon)$ but with a slight increase in cost. In particular, it appears to scale somewhat faster than linearly but remains consistent with Corollary \ref{cor:hifie-2k}.

\subsubsection*{Example 3}
We then turn to the Lippmann-Schwinger equation
\begin{align*}
 \sigma (x) + k^{2} \int_{\Omega} K(\| x - y \|) \omega (y) \sigma (y) \, d \Omega (y) = f(x), \quad x \in \Omega = (0, 1)^{2}
\end{align*}
for Helmholtz scattering, where $k = 2 \pi \kappa$ is the frequency of the incoming wave with $\kappa$ the number of wavelengths in $\Omega$; $K(r) = (i/4) H_{0}^{(1)} (kr)$ is the fundamental solution of the associated Helmholtz equation satisfying the Sommerfeld radiation condition, where $i$ is the imaginary unit and $H_{0}^{(1)} (\cdot)$ is the zeroth order Hankel function of the first kind; and $\omega (x)$ is a continuous function representing the scatterer. We refer the interested reader to \cite{colton:1992:springer} for details. Assuming that $\omega (x) \geq 0$, this can be symmetrized by the change of variables $u(x) = \sqrt{\omega(x)} \sigma (x)$ as
\begin{multline}
 u(x) + k \sqrt{\omega (x)} \int_{\Omega} K(\| x - y \|) \left[ k \sqrt{\omega (y)} \right] u(y) \, d \Omega (y) = \sqrt{\omega (x)} f(x),\\
 x \in \Omega,
 \label{eqn:lippmann-schwinger}
\end{multline}
i.e., \eqref{eqn:ie} with $a(x) \equiv 1$ and $b(x) \equiv c(x) = k \sqrt{\omega (x)}$. We took a Gaussian bump $\omega (x) = \exp (-32(x - x_{0})^{2})$ for $x_{0} = (1/2, 1/2)$ as the scatterer and discretized \eqref{eqn:lippmann-schwinger} using a uniform grid with quadratures as computed in Example 1. The frequency $k$ was increased with $n = \sqrt{N}$ to keep the number of DOFs per wavelength fixed at $32$. Data for \alg{rskelf2}, \alg{hifie2}, and \alg{hifie2x} with $\kappa = 8$, $16$, and $32$ at $\epsilon = 10^{-6}$ are shown in Tables \ref{tab:ie_square3-f} and \ref{tab:ie_square3-a}.

\begin{table}
 \caption{Factorization results for Example 3.}
 \label{tab:ie_square3-f}
 \scriptsize
 \begin{tabular}{crr|rcc|rcc|rcc}
  \toprule
  & & & \multicolumn{3}{|c|}{\alg{rskelf2}} & \multicolumn{3}{|c}{\alg{hifie2}} & \multicolumn{3}{|c}{\alg{hifie2x}}\\
  $\epsilon$ & \multicolumn{1}{c}{$N$} & \multicolumn{1}{c|}{$\kappa$} & \multicolumn{1}{|c}{$|s_{L}|$} & $t_{f}$ & $m_{f}$ & \multicolumn{1}{|c}{$|s_{L}|$} & $t_{f}$ & $m_{f}$ & \multicolumn{1}{|c}{$|s_{L}|$} & $t_{f}$ & $m_{f}$\\
  \midrule
  \multirow{3}{*}{$10^{-6}$} &  $256^{2}$ &  $8$ & $1522$ & $8.3$e$+2$ & $8.5$e$-1$ &  $551$ & $7.8$e$+2$ & $6.8$e$-1$ &  $592$ & $8.4$e$+2$ & $7.2$e$-1$\\
                             &  $512^{2}$ & $16$ & $2995$ & $5.0$e$+3$ & $4.4$e$+0$ &  $860$ & $4.0$e$+3$ & $3.0$e$+0$ &  $825$ & $4.3$e$+3$ & $3.4$e$+0$\\
                             & $1024^{2}$ & $32$ & $5918$ & $3.0$e$+4$ & $2.2$e$+1$ & $1331$ & $1.8$e$+4$ & $1.3$e$+1$ & $1229$ & $2.0$e$+4$ & $1.5$e$+1$\\
  \bottomrule
 \end{tabular}
\end{table}

\begin{table}
 \caption{Matrix application results for Example 3.}
 \label{tab:ie_square3-a}
 \scriptsize
 \begin{tabular}{crr|c|cccr|cccr}
  \toprule
  & & & \multicolumn{1}{|c|}{\alg{rskelf2}} & \multicolumn{4}{|c}{\alg{hifie2}} & \multicolumn{4}{|c}{\alg{hifie2x}}\\
  $\epsilon$ & \multicolumn{1}{c}{$N$} & \multicolumn{1}{c|}{$\kappa$} & $t_{a/s}$ & $t_{a/s}$ & $e_{a}$ & $e_{s}$ & \multicolumn{1}{c|}{$n_{i}$} & $t_{a/s}$ & $e_{a}$ & $e_{s}$ & \multicolumn{1}{c}{$n_{i}$}\\
  \midrule
  \multirow{3}{*}{$10^{-6}$} &  $256^{2}$ &  $8$ & $4.1$e$-1$ & $3.5$e$-1$ & $1.8$e$-4$ & $8.5$e$-4$ & $3$ & $4.6$e$-1$ & $7.7$e$-6$ & $3.9$e$-5$ & $3$\\
                             &  $512^{2}$ & $16$ & $2.4$e$+0$ & $1.6$e$+0$ & $8.8$e$-4$ & $5.8$e$-3$ & $6$ & $2.1$e$+0$ & $1.8$e$-5$ & $1.7$e$-4$ & $3$\\
                             & $1024^{2}$ & $32$ & $1.2$e$+1$ & $8.3$e$+0$ & $5.5$e$-3$ & $5.7$e$-2$ & $9$ & $9.3$e$+0$ & $6.5$e$-5$ & $9.6$e$-4$ & $3$\\
  \bottomrule
 \end{tabular}
\end{table}

Overall, the results are similar to those in Example 2 but with added computational expense due to working over $\mathbb{C}$ and computing $H_{0}^{(1)} (kr)$. Moreover, although \eqref{eqn:lippmann-schwinger} is formally a second-kind IE, it becomes increasingly first-kind as $k \to \infty$. Thus, the problem is somewhat ill-conditioned, as reflected in the deterioration of $e_{a}$ and $e_{s}$ even for \alg{hifie2x}. Nevertheless, $F^{-1}$ remains a very good preconditioner, with $n_{i} = O(1)$ for \alg{hifie2x}. Interestingly, despite its inaccuracy, \alg{hifie2} is also quite effective for preconditioning: experimentally, we observe that $n_{i} = O(\log N)$, which can be justified as follows.

\begin{lemma}
 \label{lem:gmres}
 If $A = I + E$ with $\epsilon = \| E \|$, then the number of iterations for GMRES to solve \eqref{eqn:linear-system} to any target precision $\epsilon_{0} > 0$ is $n_{i} \leq \log_{\epsilon} \epsilon_{0}$.
\end{lemma}

\begin{proof}
 Let $u_{k}$ be the $k$th iterate with residual $r_{k} = A u_{k} - f$. Then the relative residual satisfies
 \begin{align*}
  \frac{\| r_{k} \|}{\| f \|} \leq \min_{p \in \poly_{k}} \| p(A) \|,
 \end{align*}
 where $\poly_{k}$ is the set of all polynomials $p$ of degree at most $k$ such that $p(0) = 1$ \cite{saad:1986:siam-j-sci-stat-comput}. Consider, in particular, the choice $p(z) = (1 - z)^{k}$. Then $\| p(A) \| \leq \| I - A \|^{k} = \| E \|^{k} = \epsilon^{k}$, so $\| r_{k} \| / \| f \| \leq \epsilon^{k}$. Setting the left-hand side equal to $\epsilon_{0}$ yields $n_{i} \equiv k \leq \log_{\epsilon} \epsilon_{0}$.
\end{proof}

\begin{corollary}
 Let $F = A + E$ and $F^{-1} = A^{-1} + G$ with $\| E \| \leq CN \epsilon \| A \|$ and $\| G \| \leq CN \epsilon \kappa (A) \| A^{-1} \|$ for some constant $C$ such that $CN \epsilon \kappa (A) \ll 1$. Then the number of iterations for GMRES to solve \eqref{eqn:linear-system} with preconditioner $F^{-1}$ is
 \begin{align*}
  n_{i} \sim \left( 1 + \log_{1 / \epsilon} CN \kappa (A) \right) \log_{\epsilon} \epsilon_{0}.
 \end{align*}
\end{corollary}

\begin{proof}
 The preconditioned matrix is $F^{-1} A = F^{-1} (F - E) = I - F^{-1} E$, where
 \begin{align*}
  \| F^{-1} E \| \leq (\| A^{-1} \| + \| G \|) \| E \| \leq CN \epsilon \kappa (A) (1 + CN \epsilon \kappa (A)) \sim CN \epsilon \kappa (A),
 \end{align*}
 so Lemma \ref{lem:gmres} gives
 \begin{multline*}
  n_{i} \sim \log_{CN \epsilon \kappa (A)} \epsilon_{0} = \frac{\log \epsilon_{0}}{\log CN \epsilon \kappa (A)} = \left( \frac{1}{1 + \log_{\epsilon} CN \kappa (A)} \right) \frac{\log \epsilon_{0}}{\log \epsilon}\\
  = \left( \frac{1}{1 - \log_{1 / \epsilon} CN \kappa (A)} \right) \log_{\epsilon} \epsilon_{0}.
 \end{multline*}
 But $CN \kappa (A) \ll 1 / \epsilon$, so $\log_{1 / \epsilon} CN \kappa (A) \ll 1$. The claim now follows by first-order expansion of the term in parentheses.
\end{proof}

We remark that HIF-IE is effective only at low to moderate frequency since the rank structures employed break down as $k \to \infty$. In the limit, the only compression possible is due to Green's theorem, with HIF-IE reducing to RSF for volume IEs. The situation is yet worse for boundary IEs, for which no compression at all is available in general, and both RSF and HIF-IE revert to having $O(N^{3})$ complexity.

\subsection{Three Dimensions}
We next present three examples in 3D: a boundary IE and two volume IEs as in Examples 1 and 2.

\subsubsection*{Example 4}
Consider the second-kind boundary IE \eqref{eqn:2k-bie} on the unit sphere $\Gamma = S^{2}$, where $G(r)$ is as defined in \eqref{eqn:laplace-kernel}. It is possible to reparametrize $\Gamma$ in 2D and then use 2D algorithms, but we ran the full 3D solvers here. We represented $\Gamma$ as a collection of flat triangles and discretized via a centroid collocation scheme. Near-field interactions for all centroids within a local neighborhood of radius $h$ about each triangle, where $h$ is the average triangle diameter, were computed using fourth-order tensor-product Gauss-Legendre quadrature. This gives a linear system \eqref{eqn:linear-system} with {\em unsymmetric} $A$. Data for \alg{rskelf3}, \alg{hifie3}, and \alg{hifie3x} at $\epsilon = 10^{-3}$ and $10^{-6}$ are shown in Tables \ref{tab:ie_sphere1-f} and \ref{tab:ie_sphere1-a} with scaling results in Figure \ref{fig:ie_sphere1}.

\begin{table}
 \caption{Factorization results for Example 4.}
 \label{tab:ie_sphere1-f}
 \scriptsize
 \begin{tabular}{cr|rcc|rcc|rcc}
  \toprule
  & & \multicolumn{3}{|c|}{\alg{rskelf3}} & \multicolumn{3}{|c}{\alg{hifie3}} & \multicolumn{3}{|c}{\alg{hifie3x}}\\
  $\epsilon$ & \multicolumn{1}{c|}{$N$} & \multicolumn{1}{|c}{$|s_{L}|$} & $t_{f}$ & $m_{f}$ & \multicolumn{1}{|c}{$|s_{L}|$} & $t_{f}$ & $m_{f}$ & \multicolumn{1}{|c}{$|s_{L}|$} & $t_{f}$ & $m_{f}$\\
  \midrule
  \multirow{4}{*}{$10^{-3}$} &   $20480$ &  $3843$ & $3.3$e$+2$ & $4.7$e$-1$ &  $1143$ & $2.2$e$+2$ & $2.2$e$-1$ &  $2533$ & $3.5$e$+2$ & $3.4$e$-1$\\
                             &   $81920$ &  $7659$ & $2.7$e$+3$ & $2.2$e$+0$ &  $1247$ & $7.3$e$+2$ & $7.2$e$-1$ &  $3456$ & $1.7$e$+3$ & $1.3$e$+0$\\
                             &  $327680$ & $15091$ & $2.0$e$+4$ & $1.0$e$+1$ &  $1300$ & $3.0$e$+3$ & $2.9$e$+0$ &  $2875$ & $7.4$e$+3$ & $5.2$e$+0$\\
                             & $1310720$ & $27862$ & $1.4$e$+5$ & $4.2$e$+1$ &  $1380$ & $1.1$e$+4$ & $1.1$e$+1$ &  $2934$ & $2.6$e$+4$ & $1.8$e$+1$\\
  \midrule
  \multirow{3}{*}{$10^{-6}$} &   $20480$ &  $6939$ & $1.3$e$+3$ & $1.2$e$+0$ &  $4976$ & $1.2$e$+3$ & $8.0$e$-1$ &  $6256$ & $1.4$e$+3$ & $1.1$e$+0$\\
                             &   $81920$ & $14295$ & $1.5$e$+4$ & $6.2$e$+0$ &  $8619$ & $8.4$e$+3$ & $3.2$e$+0$ & $10748$ & $9.5$e$+3$ & $4.7$e$+0$\\
                             &  $327680$ & $28952$ & $1.3$e$+5$ & $3.1$e$+1$ & $13782$ & $5.0$e$+4$ & $1.2$e$+1$ & $13625$ & $5.4$e$+4$ & $1.9$e$+1$\\
  \bottomrule
 \end{tabular}
\end{table}

\begin{table}
 \caption{Matrix application results for Example 4.}
 \label{tab:ie_sphere1-a}
 \scriptsize
 \begin{tabular}{cr|c|ccc|ccc}
  \toprule
  & & \multicolumn{1}{|c|}{\alg{rskelf3}} & \multicolumn{3}{|c}{\alg{hifie3}} & \multicolumn{3}{|c}{\alg{hifie3x}}\\
  $\epsilon$ & \multicolumn{1}{c|}{$N$} & $t_{a/s}$ & $t_{a/s}$ & $e_{a}$ & $e_{s}$ & $t_{a/s}$ & $e_{a}$ & $e_{s}$\\
  \midrule
  \multirow{4}{*}{$10^{-3}$} &   $20480$ & $2.6$e$-1$ & $1.8$e$-1$ & $6.4$e$-3$ & $1.0$e$-2$ & $2.1$e$-1$ & $3.8$e$-4$ & $7.0$e$-4$\\
                             &   $81920$ & $1.2$e$+0$ & $5.3$e$-1$ & $4.0$e$-2$ & $5.1$e$-2$ & $6.7$e$-1$ & $1.0$e$-3$ & $1.8$e$-3$\\
                             &  $327680$ & $4.7$e$+0$ & $1.9$e$+0$ & $8.8$e$-2$ & $1.1$e$-1$ & $3.3$e$+0$ & $4.2$e$-4$ & $8.1$e$-4$\\
                             & $1310720$ & $2.2$e$+1$ & $7.2$e$+0$ & $2.4$e$-1$ & $3.3$e$-1$ & $1.1$e$+1$ & $6.0$e$-4$ & $7.1$e$-4$\\
  \midrule
  \multirow{3}{*}{$10^{-6}$} &   $20480$ & $5.6$e$-1$ & $4.3$e$-1$ & $3.7$e$-6$ & $6.8$e$-6$ & $4.9$e$-1$ & $4.1$e$-7$ & $8.0$e$-7$\\
                             &   $81920$ & $2.9$e$+0$ & $1.8$e$+0$ & $1.3$e$-5$ & $2.4$e$-5$ & $2.1$e$+0$ & $3.7$e$-7$ & $6.1$e$-7$\\
                             &  $327680$ & $1.5$e$+1$ & $6.5$e$+0$ & $5.6$e$-5$ & $1.0$e$-4$ & $1.1$e$+1$ & $5.9$e$-7$ & $1.0$e$-6$\\
  \bottomrule
 \end{tabular}
\end{table}

\begin{figure}
 \includegraphics{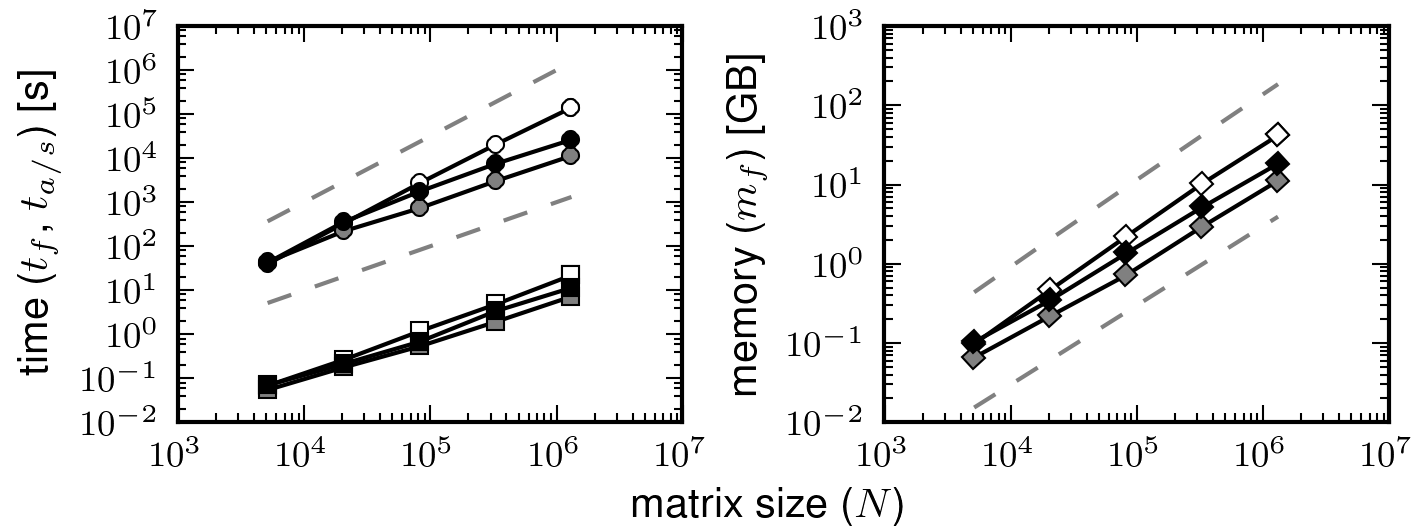}
 \caption{Scaling results for Example 4, comparing \alg{rskelf3} (white), \alg{hifie3} (gray), and \alg{hifie3x} (black) at precision $\epsilon = 10^{-3}$; all other notation as in Figure \ref{fig:ie_square1}.}
 \label{fig:ie_sphere1}
\end{figure}

Since $\Gamma$ is a 2D surface, $d = 2$ in Theorem \ref{thm:rskelf}, so we can expect RSF to have $O(N^{3/2})$ complexity, as observed. However, the skeleton size is substantially larger than in 2D, so the corresponding costs are much higher. The same is true for HIF-IE, which achieves quasilinear complexity as predicted in Theorem \ref{thm:hifie} and Corollary \ref{cor:hifie-2k}. As before, $e_{a}, e_{s} = O(\epsilon)$ for \alg{hifie3x} but suffer for \alg{hifie3}.

We also tested the accuracy of our algorithms in solving the associated PDE \eqref{eqn:dirich-laplace} by constructing an interior harmonic field
\begin{align*}
 v(x) = \sum_{j} G(\| x - y_{j} \|) q_{j}, \quad x \in \mathcal{D}
\end{align*}
due to $16$ random exterior sources $\{ y_{j} \}$ with $\| y_{j} \| = 2$, where the ``charge'' strengths $q_{j}$ were drawn from the standard uniform distribution. This induces the boundary data $f(x) = v(x) \vert_{x \in \Gamma}$, which returns the charge density $\sigma (x)$ upon solving \eqref{eqn:2k-bie}. The field $u(x)$ due to $\sigma (x)$ via the double-layer potential \eqref{eqn:double-layer} is then, in principle, identical to $v(x)$ by uniqueness of the boundary value problem. This equality was assessed by evaluating both $u(x)$ and $v(x)$ at $16$ random interior targets $\{ z_{j} \}$ with $\| z_{j} \| = 1/2$. The relative error between $\{ u(z_{j}) \}$ and $\{ v(z_{j}) \}$ is shown in Table \ref{tab:ie_sphere1-e}, from which we observe that \alg{rskelf3} and \alg{hifie3x} are both able to solve the PDE up to the discretization or approximation error.

\begin{table}
 \caption{Relative errors against exact solutions for the PDE in Example 4.}
 \label{tab:ie_sphere1-e}
 \scriptsize
 \begin{tabular}{cr|c|c|c}
  \toprule
  $\epsilon$ & \multicolumn{1}{c|}{$N$} & \alg{rskelf3} & \alg{hifie3} & \alg{hifie3x}\\
  \midrule
  \multirow{4}{*}{$10^{-3}$} &   $20480$ & $7.6$e$-4$ & $2.8$e$-3$ & $7.8$e$-4$\\
                             &   $81920$ & $3.0$e$-4$ & $3.0$e$-2$ & $4.2$e$-4$\\
                             &  $327680$ & $1.2$e$-4$ & $8.1$e$-2$ & $2.1$e$-4$\\
                             & $1310720$ & $4.8$e$-4$ & $3.1$e$-1$ & $2.0$e$-4$\\
  \midrule
  \multirow{3}{*}{$10^{-6}$} &   $20480$ & $7.9$e$-4$ & $7.9$e$-4$ & $7.8$e$-4$\\
                             &   $81920$ & $3.7$e$-4$ & $3.7$e$-4$ & $3.7$e$-4$\\
                             &  $327680$ & $1.8$e$-4$ & $1.8$e$-4$ & $1.8$e$-4$\\
  \bottomrule
 \end{tabular}
\end{table}

\subsubsection*{Example 5}
Now consider the 3D analogue of Example 1, i.e., \eqref{eqn:ie} with $a(x) \equiv 0$, $b(x) \equiv c(x) = 1$, $K(r) = 1/(4 \pi r)$, and $\Omega = (0, 1)^{3}$, discretized over a uniform grid with adaptive quadratures for the diagonal entries. Data for \alg{rskelf3} and \alg{hifie3} at $\epsilon = 10^{-3}$ and $10^{-6}$ are given in Tables \ref{tab:ie_cube1-f} and \ref{tab:ie_cube1-a} with scaling results in Figure \ref{fig:ie_cube1}.

\begin{table}
 \caption{Factorization results for Example 5.}
 \label{tab:ie_cube1-f}
 \scriptsize
 \begin{tabular}{cr|rcc|rcc}
  \toprule
  & & \multicolumn{3}{|c|}{\alg{rskelf3}} & \multicolumn{3}{|c}{\alg{hifie3}}\\
  $\epsilon$ & \multicolumn{1}{c|}{$N$} & \multicolumn{1}{|c}{$|s_{L}|$} & $t_{f}$ & $m_{f}$ & \multicolumn{1}{|c}{$|s_{L}|$} & $t_{f}$ & $m_{f}$\\
  \midrule
  \multirow{3}{*}{$10^{-3}$} &  $32^{3}$ &   $5900$ & $5.4$e$+2$ & $1.0$e$+0$ &   $969$ & $1.6$e$+2$ & $2.7$e$-1$\\
                             &  $64^{3}$ &  $24005$ & $3.9$e$+4$ & $1.9$e$+1$ &  $1970$ & $3.4$e$+3$ & $2.6$e$+0$\\
                             & $128^{3}$ & \tabdash &        --- &        --- &  $3981$ & $5.5$e$+4$ & $2.5$e$+1$\\
  \midrule
  \multirow{2}{*}{$10^{-6}$} &  $32^{3}$ &  $11132$ & $2.4$e$+3$ & $2.8$e$+0$ &  $6108$ & $2.1$e$+3$ & $1.4$e$+0$\\
                             &  $64^{3}$ & \tabdash &        --- &        --- & $16401$ & $1.0$e$+5$ & $2.0$e$+1$\\
  \bottomrule
 \end{tabular}
\end{table}

\begin{table}
 \caption{Matrix application results for Example 5.}
 \label{tab:ie_cube1-a}
 \scriptsize
 \begin{tabular}{cr|c|cccr}
  \toprule
  & & \multicolumn{1}{|c|}{\alg{rskelf3}} & \multicolumn{4}{|c}{\alg{hifie3}}\\
  $\epsilon$ & \multicolumn{1}{c|}{$N$} & $t_{a/s}$ & $t_{a/s}$ & $e_{a}$ & $e_{s}$ & \multicolumn{1}{c}{$n_{i}$}\\
  \midrule
  \multirow{3}{*}{$10^{-3}$} &  $32^{3}$ & $4.0$e$-1$ & $1.6$e$-1$ & $3.1$e$-4$ & $2.7$e$-2$ & $6$\\
                             &  $64^{3}$ & $6.2$e$+0$ & $1.5$e$+0$ & $3.6$e$-4$ & $4.4$e$-2$ & $7$\\
                             & $128^{3}$ &        --- & $1.4$e$+1$ & $1.2$e$-3$ & $7.2$e$-2$ & $8$\\
  \midrule
  \multirow{2}{*}{$10^{-6}$} &  $32^{3}$ & $1.1$e$+0$ & $5.2$e$-1$ & $1.2$e$-7$ & $2.8$e$-5$ & $3$\\
                             &  $64^{3}$ &        --- & $6.1$e$+0$ & $2.4$e$-7$ & $9.5$e$-5$ & $3$\\
  \bottomrule
 \end{tabular}
\end{table}

\begin{figure}
 \includegraphics{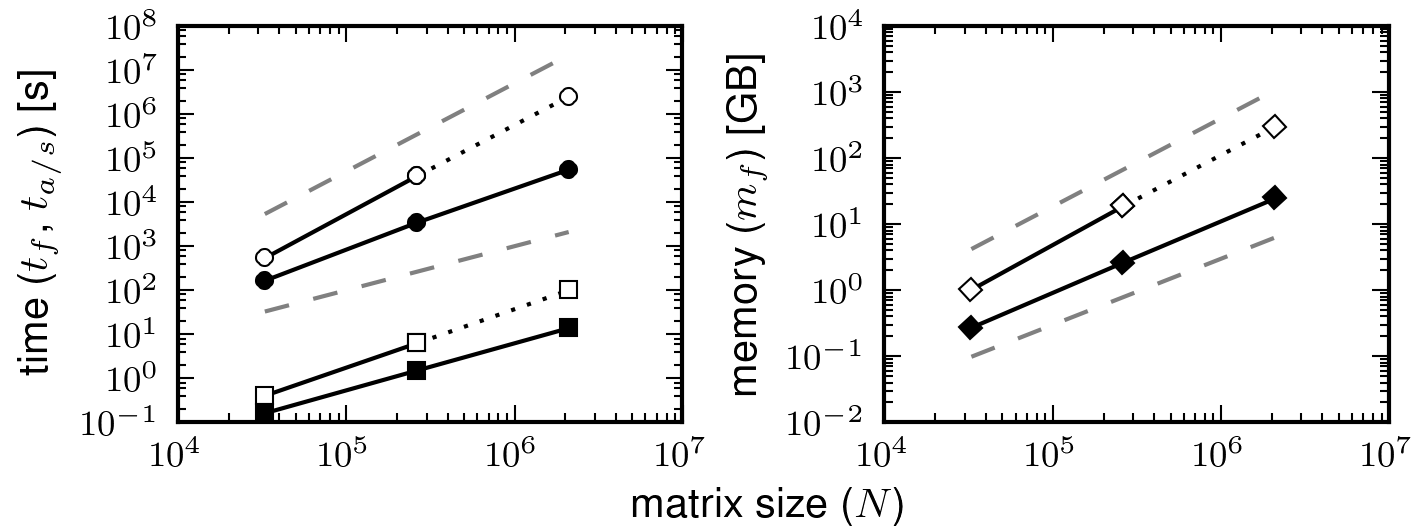}
 \caption{Scaling results for Example 5, comparing \alg{rskelf3} (white) and \alg{hifie3} (black) at precision $\epsilon = 10^{-3}$. Dotted lines denote extrapolated values. Included also are reference scalings of $O(N)$ and $O(N^{2})$ (left), and $O(N)$ and $O(N^{4/3}$) (right); all other notation as in Figure \ref{fig:ie_square1}.}
 \label{fig:ie_cube1}
\end{figure}

It is immediate that $t_{f} = O(N^{2})$ and $t_{a/s} = O(N^{4/3})$ for RSF, which considerably degrades its performance for large $N$. Indeed, we were unable to run \alg{rskelf3} for $N = 128^{3}$ because of the excessive memory cost. In contrast, HIF-IE scales much better though does not quite achieve $O(N)$ complexity as stated in Theorem \ref{thm:hifie}: the empirical scaling for $t_{f}$ at $\epsilon = 10^{-3}$, for instance, is approximately $O(N^{1.3})$. We believe this to be a consequence of the large interaction ranks in 3D, which make the asymptotic regime rather difficult to reach. Still, even the experimental growth rate of $k_{\ell} \simeq O(2^{\ell})$ would be sufficient for theoretical $O(N \log N)$ complexity. In parallel with Example 1, $e_{a} = O(\epsilon)$ but $e_{s}$ is somewhat larger due to ill-conditioning. We found $F^{-1}$ to be a very effective preconditioner throughout.

\subsubsection*{Example 6}
Finally, we consider the 3D analogue of Example 2, i.e., Example 5 but with $a(x) \equiv 1$. This is a well-conditioned second-kind IE, which we factored using \alg{rskelf3}, \alg{hifie3}, and \alg{hifie3x}. The data are summarized in Tables \ref{tab:ie_cube2-f} and \ref{tab:ie_cube2-a} with scaling results shown in Figure \ref{fig:ie_cube2}.

\begin{table}
 \caption{Factorization results for Example 6.}
 \label{tab:ie_cube2-f}
 \scriptsize
 \begin{tabular}{cr|rcc|rcc|rcc}
  \toprule
  & & \multicolumn{3}{|c|}{\alg{rskelf3}} & \multicolumn{3}{|c}{\alg{hifie3}} & \multicolumn{3}{|c}{\alg{hifie3x}}\\
  $\epsilon$ & \multicolumn{1}{c|}{$N$} & \multicolumn{1}{|c}{$|s_{L}|$} & $t_{f}$ & $m_{f}$ & \multicolumn{1}{|c}{$|s_{L}|$} & $t_{f}$ & $m_{f}$ & \multicolumn{1}{|c}{$|s_{L}|$} & $t_{f}$ & $m_{f}$\\
  \midrule
  \multirow{3}{*}{$10^{-3}$} &  $32^{3}$ &   $5900$ & $5.4$e$+2$ & $1.0$e$+0$ &  $1271$ & $2.1$e$+2$ & $3.9$e$-1$ &  $3127$ & $5.0$e$+2$ & $6.6$e$-1$\\
                             &  $64^{3}$ &  $24005$ & $4.0$e$+4$ & $1.9$e$+1$ &  $2023$ & $3.3$e$+3$ & $3.7$e$+0$ &  $7141$ & $1.3$e$+4$ & $8.5$e$+0$\\
                             & $128^{3}$ & \tabdash &        --- &        --- &  $5105$ & $5.2$e$+4$ & $3.6$e$+1$ & $17491$ & $3.5$e$+5$ & $1.1$e$+2$\\
  \midrule
  \multirow{2}{*}{$10^{-6}$} &  $32^{3}$ &  $11132$ & $2.4$e$+3$ & $2.8$e$+0$ &  $5611$ & $1.6$e$+3$ & $1.4$e$+0$ &  $8620$ & $2.4$e$+3$ & $2.2$e$+0$\\
                             &  $64^{3}$ & \tabdash &        --- &        --- & $12558$ & $5.4$e$+4$ & $1.6$e$+1$ & $25797$ & $8.6$e$+4$ & $3.4$e$+1$\\
  \bottomrule
 \end{tabular}
\end{table}

\begin{table}
 \caption{Matrix application results for Example 6.}
 \label{tab:ie_cube2-a}
 \scriptsize
 \begin{tabular}{cr|c|ccc|ccc}
  \toprule
  & & \multicolumn{1}{|c|}{\alg{rskelf3}} & \multicolumn{3}{|c}{\alg{hifie3}} & \multicolumn{3}{|c}{\alg{hifie3x}}\\
  $\epsilon$ & \multicolumn{1}{c|}{$N$} & $t_{a/s}$ & $t_{a/s}$ & $e_{a}$ & $e_{s}$ & $t_{a/s}$ & $e_{a}$ & $e_{s}$\\
  \midrule
  \multirow{3}{*}{$10^{-3}$} &  $32^{3}$ & $4.0$e$-1$ & $2.0$e$-1$ & $4.6$e$-3$ & $5.0$e$-3$ & $2.2$e$-1$ & $1.1$e$-4$ & $1.3$e$-4$\\
                             &  $64^{3}$ & $6.6$e$+0$ & $1.8$e$+0$ & $4.4$e$-2$ & $4.7$e$-2$ & $3.1$e$+0$ & $6.2$e$-4$ & $6.8$e$-4$\\
                             & $128^{3}$ &        --- & $1.7$e$+1$ & $6.7$e$-2$ & $7.3$e$-2$ & $5.1$e$+1$ & $1.7$e$-3$ & $1.9$e$-3$\\
  \midrule
  \multirow{2}{*}{$10^{-6}$} &  $32^{3}$ & $1.0$e$+0$ & $5.7$e$-1$ & $8.5$e$-6$ & $9.7$e$-6$ & $7.4$e$-1$ & $2.9$e$-7$ & $3.4$e$-7$\\
                             &  $64^{3}$ &        --- & $6.4$e$+0$ & $5.9$e$-5$ & $6.8$e$-5$ & $1.2$e$+1$ & $1.5$e$-6$ & $1.8$e$-6$\\
  \bottomrule
 \end{tabular}
\end{table}

\begin{figure}
 \includegraphics{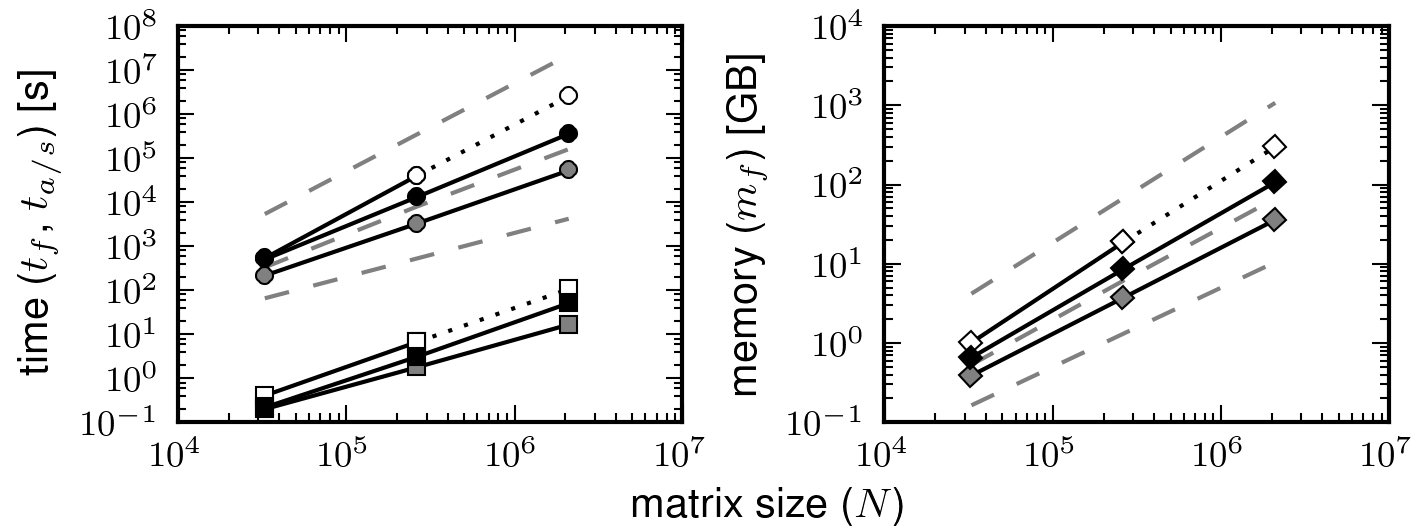}
 \caption{Scaling results for Example 6, comparing \alg{rskelf3} (white), \alg{hifie3} (gray), and \alg{hifie3x} (black) at precision $\epsilon = 10^{-3}$. Included also are reference scalings of $O(N)$, $O(N \log^{6} N)$, and $O(N^{2})$ (left); and $O(N)$, $O(N \log^{2} N)$, and $O(N^{4/3})$ (right). All other notation as in Figure \ref{fig:ie_cube1}.}
 \label{fig:ie_cube2}
\end{figure}

Algorithms \alg{rskelf3} and \alg{hifie3} behave very similarly as in Example 5 but with some error propagation for \alg{hifie3} as discussed in Section \ref{sec:hifie:2k-mod}. Full accuracy is restored using \alg{hifie3x} but at the cost of significantly larger skeleton sizes. The empirical complexity of \alg{hifie3x} hence suffers but remains quite favorable compared to that of \alg{rskelf3}. We also find a good fit with the complexity estimates of Corollary \ref{cor:hifie-2k}, though the presumed penalty for not yet reaching the asymptotic regime may imply that the proposed bounds are overly pessimistic.

\section{Generalizations and Conclusions}
\label{sec:conclusion}
In this paper, we have introduced HIF-IE for the efficient factorization of discretized integral operators associated with elliptic PDEs in 2D and 3D. HIF-IE combines a novel matrix sparsification framework with recursive dimensional reduction to construct an approximate generalized LU decomposition at estimated quasilinear cost. The latter enables significant compression over RS and is critical for improving the asymptotic complexity, while the former substantially simplifies the algorithm and permits its formulation as a factorization. This representation allows the rapid application of both the matrix and its inverse, and therefore provides a generalized FMM, direct solver, or preconditioner, depending on the accuracy. We have also presented RSF, a factorization formulation of RS \cite{gillman:2012:front-math-china,greengard:2009:acta-numer,ho:2012:siam-j-sci-comput,martinsson:2005:j-comput-phys} that is closely related to MF \cite{duff:1983:acm-trans-math-software,george:1973:siam-j-numer-anal} for sparse matrices. Indeed, a key observation underlying both RSF and HIF-IE is that structured dense matrices can be sparsified very efficiently via the ID. This suggests that well-developed sparse techniques can be applied, and we anticipate that fully exploring this implication will lead to new fast algorithms for dense linear algebra.

The skeletonization operator at the core of RSF and HIF-IE can be interpreted in several ways. For example, we can view it as an approximate local change of basis in order to gain sparsity. Unlike traditional approaches \cite{alpert:1993:siam-j-sci-comput,beylkin:1991:comm-pure-appl-math,dahmen:1997:acta-numer}, however, this basis is determined optimally on the fly using the ID. Skeletonization can also be regarded as adaptive numerical upscaling or as implementing specialized restriction and prolongation operators in the context of multigrid methods \cite{hackbusch:1985:springer}.

Although we have presently only considered matrices arising from IEs, the same methods can also be applied (with minor modification) to various general structured matrices such as those encountered in Gaussian process modeling \cite{ambikasaran:arxiv,chen:2014:siam-j-sci-comput} or sparse differential formulations of PDEs \cite{bebendorf:2003:numer-math,gillman:2014:adv-comput-math,xia:2009:siam-j-matrix-anal-appl}. In particular, HIF-IE can be heavily specialized to the latter setting by explicitly taking advantage of existing sparsity. The resulting \defn{hierarchical interpolative factorization for differential equations} (HIF-DE) is described in the companion paper \cite{ho:comm-pure-appl-math} and likewise achieves estimated linear or quasilinear complexity in 2D and 3D.

Some important directions for future research include:
\begin{itemize}
 \itemsep 1ex
 \item
  Obtaining analytical estimates of the interaction rank for SCIs, even for the simple case of the Laplace kernel \eqref{eqn:laplace-kernel}. This would enable a much more precise understanding of the complexity of HIF-IE, which has yet to be rigorously established.
 \item
  Parallelizing RSF and HIF-IE, both of which are organized according to a tree structure where each node at a given level can be processed independently of the rest. The parallelization of HIF-IE holds particular promise and should have significant impact on practical scientific computing.
 \item
  Investigating alternative strategies for reducing skeleton sizes in 3D, which can still be quite large, especially at high precision. New ideas may be required to build truly large-scale direct solvers.
 \item
  Understanding the extent to which our current techniques can be adapted to highly oscillatory kernels, which possess rank structures of a different type than that exploited here \cite{engquist:2009:comm-math-sci,engquist:2007:siam-j-sci-comput}. Such high-frequency problems can be extremely difficult to solve by iteration and present a prime target area for future fast direct methods.
\end{itemize}







\ack


We would like to thank Leslie Greengard for many helpful discussions, Lenya Ryzhik for providing computing resources, and the anonymous referees for their careful reading of the manuscript, which have improved the paper tremendously. K.L.H.\ was partially supported by the National Science Foundation under award DMS-1203554. L.Y.\ was partially supported by the National Science Foundation under award DMS-1328230 and the U.S.\ Department of Energy's Advanced Scientific Computing Research program under award DE-FC02-13ER26134/DE-SC0009409.


\frenchspacing
\bibliographystyle{plain}

\end{document}